\documentclass[11pt,letterpaper]{amsart}

\usepackage{amscd, amssymb, amsmath, amsthm}
\usepackage{enumerate,latexsym, mathrsfs}
\usepackage{xypic}
\usepackage[colorlinks=true, linkcolor=blue, citecolor=blue]{hyperref}

\newtheorem{theorem}{Theorem}[section]
\newtheorem{lemma}[theorem]{Lemma}
\newtheorem{proposition}[theorem]{Proposition}
\newtheorem{corollary}[theorem]{Corollary}

\theoremstyle{definition}
\newtheorem{definition}[theorem]{Definition}
\newtheorem{convention}[theorem]{Convention}

\newtheorem{example}[theorem]{Example}
\newtheorem{question}[theorem]{Question}
\newtheorem{remark}[theorem]{Remark}

\numberwithin{equation}{section}

\newcommand{\Natural}{{\mathbb N}}
\newcommand{\Real}{{\mathbb R}}
\newcommand{\Rational}{{\mathbb Q}}
\newcommand{\Complex}{{\mathbb C}}
\newcommand{\Integral}{{\mathbb Z}}

\newcommand{\mcg}{{\mathrm{Mod}}}
\newcommand{\outg}{{\mathrm{Out}}}
\newcommand{\inng}{{\mathrm{Inn}}}
\newcommand{\autg}{{\mathrm{Aut}}}
\newcommand{\univ}{{\mathrm{univ}}}

\newcommand{{\FPS}}{{\mathrm{Fix}}}
\newcommand{{\PPS}}{{\mathrm{Per}}}
\newcommand{{\FPC}}{{\mathscr{F}\mathrm{ix}}}
\newcommand{{\PPC}}{{\mathscr{P}\mathrm{er}}}
\newcommand{{\POC}}{{\mathscr{O}\mathrm{rb}}}
\newcommand{\conjorb}{{\mathrm{Orb}}}

\newcommand{{\periodic}}{{\mathtt{per}}}
\newcommand{{\pA}}{{\mathtt{pA}}}
\newcommand{{\reduction}}{{\mathtt{fDt}}}
\newcommand{{\NT}}{\mathtt{NT}}
\newcommand{{\gd}}{{\mathtt{gd}}}
\newcommand{\dilatation}{{\mathrm{Dil}}}
\newcommand{\deviation}{{\mathrm{Dev}}}

\title[Mapping classes and finite quotient actions]{Mapping classes are almost determined \\by their finite quotient actions}
     
\author[Yi Liu]{%
        Yi Liu} 
\address{%
        Beijing International Center for Mathematical Research, Peking University\\
				Beijing 100871, China P.R.} 
\email{%
    liuyi@bicmr.pku.edu.cn}

\thanks{Partially supported by NSFC Grant 11925101, 
and National Key R\&D Program of China 2020YFA0712800.}

\date{%
 \today}

\begin{document}

\begin{abstract} 
	Given any connected compact orientable surface, 
	a pair of mapping classes are said to be procongruently conjugate
	if they induce a conjugate pair of outer automophisms 
	on the profinite completion of the fundamental group of the surface.
	For example, 
	this occurs if they induce conjugate outer automorphisms
	on every characteristic finite quotient of the fundamental group.
	In this paper, it is shown that every procongruent conjugacy class of mapping classes,
	as a subset of the surface mapping class group,
	is the disjoint union of at most finitely many conjugacy classes
	of mapping classes.
	For any pseudo-Anosov mapping class of a connected closed orientable surface,
	several topological features 
	are confirmed to	depend only on the procongurent conjugacy class of the mapping class,
	including: the stretching factor,
	the topological type of the prong singularities, 
	the transverse orientability of the invariant foliations,
	and the isomorphism type of the symplectic Floer homology.
\end{abstract}

\maketitle

\section{Introduction}
The topology of irreducible compact connected $3$--manifolds 
are almost determined by their fundamental groups.
In the closed case, lens spaces with isomorphic fundamental groups 
fall into finitely many classified homeomorphism classes,
and other manifolds with isomorphic fundamental groups are unique up to homeomorphism.
The bounded case can also be completely described. 
Many important topological invariants 
turn out to depend only on the fundamental group up to isomorphism.
For irreducible orientable compact connected $3$--manifolds 
with empty or tori boundary,
simplicial volume and Thurston norm are good examples of this kind.
In fact, all these have been known since the proof of the geometrization theorem,
(see \cite[Chapter 2]{AFW_book_group}).

More recent development in $3$--manifold topology
has brought new insight to the group theoretic perspective.
A list of geometric topological properties 
have been confirmed to depend only on 
the collection of the finite quotients of the fundamental group.
Such properties 
include closedness, fiberedness, the geometric decomposition graph, and 
the type of geometries on the pieces
\cite{BRW,JZ,WZ_geometry,WZ_decomposition}.
The phenomena naturally lead to the guess that 
$3$--manifold groups might be determined by their finite quotients up to finite ambiguity.
To be precise, we say that 
a pair of finitely generated residually finite groups $G_A,G_B$ are \emph{profinitely isomorphic}
if every finite quotient group of $G_A$ is isomorphic to a finite quotient group of $G_B$,
and vice versa. 
For any prescribed class $\mathcal{G}$ of finitely generated residually finite groups,
a group $G$ of $\mathcal{G}$ is said to be \emph{profinitely rigid} among $\mathcal{G}$ 
if every group $G'$ of $\mathcal{G}$ which is profinitely isomorphic to $G$ is also isomorphic to $G$.
We say that $G$ of $\mathcal{G}$
is \emph{profinitely almost rigid} among $\mathcal{G}$
if there exist some finite set $\mathcal{S}$ of groups of $\mathcal{G}$,
and if every group $G'$ of $\mathcal{G}$ which is profinitely isomorphic to $G$ is 
isomorphic to some group in $\mathcal{S}$.

\begin{question}\label{profinite_almost_rigidity}
	Among fundamental groups of irreducible compact connected $3$--manifolds,
	is every group profinitely almost rigid?
\end{question}

For Anosov torus bundles, virtually trivial surface bundles, and non-geometric graph manifolds,
it is known that their fundamental groups may fail to be profinitely rigid
\cite{Stebe_integer_matrices,Funar_torus_bundles,Hempel_quotient,Wilkes_graph}.
However, profinite almost rigidity still holds for those examples
\cite[Theorem A]{Wilkes_graph}.
Seifert fiber spaces of non-vanishing orbifold Euler number,
and hyperbolic once-punctured-torus bundles 
have profinitely rigid fundamental groups \cite{Wilkes_sf,BRW}.
It has been confirmed more recently that the fundamental group of the Weeks manifold is profinitely rigid, 
actually among finitely generated residually finite groups,
\cite{BMRS}.
(The Weeks manifold is the closed orientable hyperbolic $3$--manifold of the smallest volume.)
We refer the reader to Reid's survey \cite{Reid_survey} for profinite rigidity and related topics.
Question \ref{profinite_almost_rigidity} remains open beyond the above mentioned cases.
For general closed or cusped hyperbolic $3$--manifolds,
it is not clear how to extract sufficient geometric information
from the profinite fundamental group.
One motivation of the present paper
is to seek for some helpful clue
by investigating a more accessible similar question.  

We introduce procongruent conjugacy between mapping classes of surfaces,
as an analogue of profinite isomorphism between $3$--manifold groups.
The analogous theme is therefore
to compare procongruent conjugacy with conjugacy---
to estimate their difference, and to recognize conjugacy properties
which remain invariant under the new notion.
As it turns out,
this seemingly plain change of context
allows us to make a good amount of progress.
It also connects up with several interesting aspects 
of low dimensional topology. 
Among other things,
we are able to obtain 
Theorem \ref{main_procongruent_almost_rigidity}
as an affirmative answer 
to the analogue of Question \ref{profinite_almost_rigidity}.

Let us briefly explain what we mean by procongruent conjugacy.
More details and further discussion appear in Section \ref{Sec-procongruent_conjugacy}.
Suppose that $S$ is an orientable connected compact surface.
The mapping class group $\mcg(S)$ 
consists of the isotopy classes of orientation-preserving self-homeomorphisms of $S$.
Every mapping class of $S$ naturally induces 
an outer automorphism of the profinite completion of the fundamental group $\widehat{\pi_1(S)}$.
We say that a pair of mapping classes $[f_A],[f_B]$ of $S$ are \emph{procongruently conjugate}
if they induce a conjugate pair in the outer automorphism group of $\widehat{\pi_1(S)}$.
For example, this will be the case 
if they induce a conjugate pair of outer automorphisms for every characteristic finite quotient $\pi_1(S)/K$,
(see Proposition \ref{characterization_congruent_conjugacy}).
Parallel to the former terminology,
we say that a mapping class $[f]\in\mcg(S)$ is \emph{procongruently rigid} 
if every mapping class $[f']\in\mcg(S)$ which is procongruently conjugate to $[f]$
is also conjugate to $[f]$ in $\mcg(S)$.
We say that $[f]\in\mcg(S)$ is \emph{procongruently almost rigid}
if there exists a finite subset $\mathcal{R}$ of $\mcg(S)$,
and if every mapping class $[f']\in\mcg(S)$ which is procongruently conjugate to $[f]$
is conjugate to some mapping class in $\mathcal{R}$.

\begin{theorem}\label{main_procongruent_almost_rigidity}
Given an orientable connected compact surface $S$,
every mapping class of $S$ is procongruently almost rigid.
\end{theorem}

Since the natural homomorphism $\mcg(S)\to \outg(\widehat{\pi_1(S)})$ factors through 
the profinite completion of $\mcg(S)$,
Theorem \ref{main_procongruent_almost_rigidity} implies obviously
an \emph{almost} version of the conjugacy separability of $\mcg(S)$,
(that is, distinguishing all but finitely many conjugacy classes 
from any given conjugacy class in $\mcg(S)$, using finite quotients of $\mcg(S)$).
If $S$ is a torus, the conjugacy separability of $\mcg(S)\cong\mathrm{SL}(2,\Integral)$ 
is due to Stebe \cite{Stebe_integer_matrices}, 
but the general case remains open.

There is a related notion that helps comparing
Theorem \ref{main_procongruent_almost_rigidity} and Question \ref{profinite_almost_rigidity}.
Regular profinite isomorphism,
as a more restrictive variation of profinite isomorphism,
has been introduced by Boileau--Friedl \cite{Boileau--Friedl_fiberedness}.
A pair of finitely generated residually finite groups $G_A,G_B$ are said to be \emph{regularly profinite isomorphic}
if there is a group isomorphism $\widehat{G}_A\cong\widehat{G}_B$ between the profinite completions,
and if the abelianized group isomorphism 
$\widehat{H}_1(\widehat{G}_A;\widehat{\Integral})\cong \widehat{H}_1(\widehat{G}_B;\widehat{\Integral})$
admits a (necessarily unique) lift to the ordinary group homology $H_1(G_A;\Integral)\cong H_1(G_B;\Integral)$.
We speak of \emph{regular profinite rigidity} and \emph{regular profinite almost rigidity} similarly as before.
Since procongruent conjugacy between mapping classes can be characterized using
profinite fundamental groups of the mapping tori (see Proposition \ref{characterization_mapping_tori}),
the following corollary is an immediate consequence of Theorem \ref{main_procongruent_almost_rigidity}
and \cite[Theorem 1.1]{Boileau--Friedl_fiberedness}.

\begin{corollary}
	Among fundamental groups of irreducible compact connected $3$--manifolds,
	the fundamental groups of any surface bundle over a circle is regularly profinitely almost rigid.
\end{corollary}

A list of dynamical features of mapping classes 
turns out to be invariant under procongruent conjugacy.
For simplicity, we state only for pseudo-Anosov mappings of closed orientable surfaces.
By saying that a property is \emph{determined} by a certain equivalence class,
we mean that for every equivalent pair of objects, the property holds for both or for neither.

\begin{theorem}\label{main_entropy}
Given an orientable connected closed surface $S$ of negative Euler characteristic,
for any pseudo-Anosov mapping $f\colon S\to S$,
the following qualities and quantities about $f$ 
are all determined by the procongruent conjugacy class of $[f]\in\mcg(S)$:
\begin{itemize}
\item the stretching factor
\item the number of fixed points for each index
\item the topological type of singularities of the invariant foliations
\item the transverse orientability of the invariant foliations
\end{itemize}
\end{theorem}

For any pseudo-Anosov mapping $f\colon S\to S$, the stretching factor occurs as an eigenvalue
of the homological action $f_*\colon H_1(S;\Complex)\to H_1(S;\Complex)$
when the (stable or unstable) invariant foliations are transversely orientable.
One might expect that the stretching factor always occured as a virtual homological eigenvalue,
namely, 
an eigenvalue of $f'_*$ for some elevation $f'\colon S'\to S'$ 
of $f$ to some finite finite cover $S'$ of $S$.
However, this is false when the invariant foliations have at least one prong singularity of odd order.
More strikingly, McMullen \cite{McMullen_entropy} 
shows that in this case, the stretching factor is strictly greater than
the supremum of $|\mu|$, 
for $\mu$ ranging over all the virtual homological eigenvalues of $f$.
Despite the failure of virtual homological approximation,
Theorem \ref{main_entropy} points out that 
it is still possible to recover the stretching factor,
roughly speaking, from all the finite quotient actions.

Given any connected closed symplectic smooth surface $(\Sigma,\omega)$,
the symplectic Floer homology $\mathrm{HF}_*(\varphi)$ is defined 
for any monotone symplectomorphism $\varphi\colon \Sigma\to \Sigma$ whose fixed points are all nondegenerate,
(see \cite[Section 2]{Cotton-Clay} and references thereof).
If $\varphi$ is isotopic to a pseudo-Anosov mapping $f\colon S\to S$,
the isomorphism type of $\mathrm{HF}_*(\varphi)$, 
as a $\Integral/2\Integral$--graded module over $\Integral/2\Integral$,
is completely determined by the indexed fixed point numbers of $f$.
This result is due to Cotton-Clay \cite{Cotton-Clay} and Fel'shtyn \cite[Theorem 4.1]{Felshtyn_HF}.
Therefore, the following corollary is implied by Theorem \ref{main_entropy}:

\begin{corollary}\label{sympl_FH}
	For any triple $(\Sigma,\omega,\varphi)$ as above,
	if $\varphi$ is isotopic to a pseudo-Anosov mapping,
	then the isomorphism type of the symplectic Floer homology
	$\mathrm{HF}_*(\varphi)$ is determined by	the procongruent conjugacy class of $[\varphi]\in\mcg(\Sigma)$.
\end{corollary}

Note that $\mathrm{HF}_*(\varphi)$ shows up as part of 
the Heegaard Floer homology $\mathrm{HF}^+(M_\varphi)$ for the mapping torus $M_\varphi$
(see \cite[Section 1.2]{Cotton-Clay} and references thereof).
Note also that $\mathrm{HF}_*(\varphi^m)$
are determined by the procongruent conjugacy class of $[\varphi]\in\mcg(\Sigma)$
for all $m\in\Natural$, as implied by Corollary \ref{sympl_FH}.
These facts make us suspect that
the Heegaard Floer homology $\mathrm{HF}^+(M)$ is determined
up to isomorphism 
by the regularly profinite isomorphism class of $\pi_1(M)$,
for every orientable closed irreducible $3$--manifold $M$.
We even wonder whether the profinite isomorphism class suffices.

Theorem \ref{main_procongruent_almost_rigidity} is the ultimate goal of this paper.
Its proof is completed in Section \ref{Sec-procongruent_almost_rigidity} based on
a queue of results in preceding sections.
As a major step,
Theorem \ref{main_entropy} yields a special case of Theorem \ref{main_procongruent_almost_rigidity}
for pseudo-Anosov mapping classes, (see Corollary \ref{pA_procongruent_almost_rigidity}).
In the rest of the introduction, we summarize our strategy and explain some key ideas.
Generally speaking, 
to obtain a finiteness result as Theorem \ref{main_procongruent_almost_rigidity},
we need some topological information about how complicated a mapping class is,
and we expect to read off the information from the profinite completion
of the mapping torus groups.
The first part becomes possible because of 
fixed point theory and the Nielsen--Thurston classification.
A particularly useful ingredient for us is the classification
of essential fixed point classes and their indices for Nielsen--Thurston normal forms, 
due to Jiang--Guo \cite{Jiang--Guo}.
The second part requires more recent 
techniques in $3$--manifold group theory.
In particular, we invoke
conjugacy separability of $3$--manifold groups due to Hamilton--Wilton--Zalesskii \cite{HWZ_conjugacy_separability}
and the profinite detection of the geometric decomposition due to Wilton--Zalesskii \cite{WZ_decomposition}.
These ingredients, in turn, rely on the virtual specialization of $3$--manifold groups
due to Agol \cite{Agol_VHC}, Wise \cite{Wise_book,Wise_notes}, Przytycki--Wise \cite{Przytycki--Wise_mixed},
and on the hereditary conjugacy separability of right-angled Artin groups
due to Minasyan \cite{Minasyan_RAAG}.

To see the basic idea, suppose that $[f]\in\mcg(S)$ is a mapping class
of an orientable connected closed surface $S$ of negative Euler characteristic.
For any $m\in\Natural$, the number of essential $m$--periodic orbit classes of $[f]$ is called
the $m$--th Nielsen number, denoted as $N_m([f])$, (see Section \ref{Subsec-fixed_point_theory} and (\ref{N_m_def})). 
It is known that $N_m([f])$ 
grows at most exponentially fast as $m\in\Natural$ tends to $\infty$.
The limit superior of $N_m([f])^{1/m}$
recovers the \emph{dilatation} of $[f]$,
which equals the greatest stretching factor of $[f]$
restricted to the pseudo-Anosov part, $1$ if empty,
with respect to the Nielsen--Thurston decomposition, (see Section \ref{Subsec-NT_complexity}).
In general, dilatation is insufficient to determine the conjugacy class of $[f]$ 
up to finite ambiguity.
One may compose $[f]$ with Dehn twists along the Nielsen--Thurston reduction curves,
and dilatation remains unchanged under this operation.
A secondary complexity can be introduced by 
suitably counting the power of the Dehn twist applied along each reduction curve.
The normalized power is sometimes called the fractional Dehn twist coefficient along the reduction curve,
and we refer to their maximum as the \emph{deviation} of $[f]$, (see Section \ref{Subsec-NT_complexity}).
On the mapping torus, the orbits of reduction curves correspond to the decomposition
tori or Klein bottles of the geometric decomposition.
Assume for the moment that there are no decomposition Klein bottles,
and that each decomposition torus is adjacent to two distinct geometric pieces.
Then for each decomposition torus,
each adjacent geometric piece determines 
a primitive essential periodic trajectory of the suspension flow on the torus.
The primitive trajectory is considered as a free-homotopy class of a loop
on the torus, sometimes called the slope of degeneracy.
(When the piece is $\mathbb{H}^2\times\mathbb{E}^1$ geometric,
the slope of degeneracy determined by the piece on the torus
is parallel to the regular Seifert fiber;
when the piece is $\mathbb{H}^3$ geometric,
the slope of degeneracy is determined by the invariant foliation
and the restricted pseudo-Anosov mapping near the corresponding reduction curve.)
By $3$--manifold topology,
it is easy to express the fractional Dehn twist coefficient along any reduction curve
in terms of the geometric intersection number between the slopes of degeneracy 
coming from the adjacent geometric pieces.
In particular, deviation can be determined 
using the geometric decomposition of the mapping torus
and those essential periodic trajectories carried by the geometric decomposition tori.
Dilatation and deviation together determines a mapping class up to finite ambiguity
(Proposition \ref{bound_finiteness}).

To read off the dilatation from the profinite completion of the mapping torus group $\widehat{\pi_1(M_f)}$,
we observe that the essential $m$--periodic orbit classes of $[f]$ 
correspond to the essential $m$--periodic trajectory classes of $M_f$,
for any $m\in\Natural$.
The latter can be considered as conjugation orbits of $\pi_1(M_f)$. 
For any finite quotient $\gamma\colon\pi_1(M_f)\to\Gamma$,
the essential $m$--periodic trajectory classes give rise to
at most $N_m([f])$ conjugation orbits of $\Gamma$,
and the equality is achievable for some $\Gamma$ by 
the conjugacy separability of $\pi_1(M_f)$.
Note that the number of irreducible complex characters of $\Gamma$ 
equals the number of conjugation orbits of $\Gamma$. 
Therefore,
we are hinted to detect $N_m(f)$ using the twisted $m$--periodic Lefschetz numbers
$L_m(f;\gamma^*\chi_\rho)$ 
for all the irreducible complex characters $\chi_\rho$ of $\Gamma$,
(see Section \ref{Subsec-twisted_Lefschetz}).
As argued in Section \ref{Sec-indexed_orbit_numbers},
the detection does work, and it can be refined to detect periodic orbit numbers
separately for each index (Theorem \ref{N_m_dialogue}).
However, 
we still have to show the procongruent conjugacy invariance of twisted periodic Lefschetz numbers,
with suitable formulation to align the twisting representation.
This is done in Sections \ref{Sec-TH_mapping_tori}, \ref{Sec-TRT_mapping_tori}, and \ref{Sec-L_m_mapping_class},
by establishing the procongruent conjugacy invariance
first for certain twisted homology of $M_f$ (Theorem \ref{aligned_mapping_torus_homology}),
and then for certain twisted Reidemeister torsion of $M_f$ (Theorem \ref{TRT_dialogue}),
and finally, for the twisted periodic Lefschetz numbers (Theorem \ref{L_m_dialogue}).
We emphasize that the assumptions of these theorems are carefully posed
to make the results correct and strong enough.
In particular, a crucial part of our argument
is to establish Theorem \ref{TRT_dialogue} over fields of characteristic $0$.
By contrast,
it seems that there should be no natural isomorphism 
on the twisted homology level with those coefficients, 
(such as a na\"ive analogue of Theorem \ref{aligned_mapping_torus_homology}).

To read off the deviation from $\widehat{\pi_1(M_f)}$,
we study the position relation between the essential periodic trajectories
and the decomposition torus or Klein bottles on the profinite level.
The idea is to translate the above topological picture about 
geometric decomposition and slopes of degeneracy
using group-theoretic terms, and to obtain their profinite analogues.
For procongruently conjugate pairs of mapping classes $[f_A],[f_B]\in\mcg(S)$,
we establish correspondences between their profinite geometric-decomposition trees
and between their indexed periodic orbit sets (Theorem \ref{correspondence_profinite_objects}).
The group-theoretic translation of the topological picture
is essential done in Lemma \ref{carry_trajectory_universal} and part of Lemma \ref{complexity_equal},
and the profinite analogue is essentially done in
Lemma \ref{carry_trajectory_profinite} and part of Lemma \ref{complexity_equal}.

With the procongruent conjugacy invariance of indexed periodic orbit numbers (Theorem \ref{N_m_dialogue}),
we are able to prove Theorem \ref{main_entropy} in Section \ref{Sec-application_pA}.
Theorem \ref{main_procongruent_almost_rigidity} follows with the extra ingredient
of the procongruent conjugacy invariance of deviation, 
which is the main point of Section \ref{Sec-procongruent_almost_rigidity}.

The organization of the paper is as follows.
In Section \ref{Sec-mc_and_mt}, we summarize some basic structures
to be considered in the study of mapping classes and mapping tori,
and introduce some basic settings.
In Section \ref{Sec-procongruent_conjugacy},
we introduce procongrent conjugacy and mention some general characterizations.
In Sections \ref{Sec-review_twisted}, \ref{Sec-TH_mapping_tori}, and \ref{Sec-TRT_mapping_tori},
we establish procongruent conjugacy invariance for certain twisted invariants of mapping tori.
In Sections \ref{Sec-fpt_review}, \ref{Sec-L_m_mapping_class}, and \ref{Sec-indexed_orbit_numbers},
we establish procongruent conjugacy invariants for twisted periodic Lefschetz numbers
and indexed periodic orbit numbers.
In Sections \ref{Sec-review_NT} and \ref{Sec-application_pA},
we complete the proof of Theorem \ref{main_entropy}.
In Sections \ref{Sec-profinite_structures} and \ref{Sec-procongruent_almost_rigidity},
we complete the proof of Theorem \ref{main_procongruent_almost_rigidity}.
Among the above, Sections \ref{Sec-review_twisted}, \ref{Sec-fpt_review}, and \ref{Sec-review_NT}
are background review. 
We also explain there the notations and definitions that we adopt.
The core arguments of this paper are contained in 
Sections \ref{Sec-TH_mapping_tori}, \ref{Sec-TRT_mapping_tori}, \ref{Sec-L_m_mapping_class}, \ref{Sec-indexed_orbit_numbers},
\ref{Sec-profinite_structures}, and \ref{Sec-procongruent_almost_rigidity}.

\subsection*{Acknowledgement} 
The author would like to thank Yang Huang, Dawid Kielak, and Alan Reid for valuable communications.
The author would like to thank the anonymous referees for many constructive suggestions.

\section{Mapping classes and mapping tori}\label{Sec-mc_and_mt}
In this section, we summarize some basic structures related to 
mapping classes and mapping tori of surface automorphisms.
We introduce two settings to be considered in subsequent sections,
as Convention \ref{settings}.
Throughout this paper,
we almost always take the covering space perspective
when speaking of fundamental groups,
so they refer to deck transformation groups 
acting on fixed universal covering spaces.

Let $S$ be an orientable connected compact surface.
Denote by $\mcg(S)$ the mapping class group of $S$, 
which consists the isotopy classes of orientation-preserving self-homeomorphisms of $S$.
Fix a universal covering space $\widetilde{S}_{\univ}$ of $S$.
Denote by $\kappa_{\univ}$
the covering projection map $\widetilde{S}_{\univ}\to S$, 
and by $\pi_1(S)$ 
the group of deck transformations acting on $\widetilde{S}_{\univ}$.

For any orientation-preserving self-homeomorphism $f\colon S\to S$,
there is a canonically induced outer automorphism of $\pi_1(S)$,
which is described as follows.
Take an elevation $\tilde{f}\colon \widetilde{S}_\univ\to \widetilde{S}_\univ$ of $f$ to $\widetilde{S}_\univ$,
(meaning $f\circ\kappa_\univ=\kappa_\univ\circ\tilde{f}$).
Then the assignment $\gamma\mapsto \tilde{f}\circ\gamma\circ\tilde{f}^{-1}$ defines an automorphism of $\pi_1(S)$.
Since the elevations of $f$ to $\widetilde{S}_\univ$ are precisely $\tilde{f}\circ\sigma$ for all $\sigma\in\pi_1(S)$,
the automorphism changes only by an inner automorphism factor for other choices of the elevation.
Since any isotopy starting from $f$ can be elevated to $\widetilde{S}_\univ$,
starting from $\tilde{f}$ and varying continuously,
the outer automorphism class depends only on the isotopy class of $f$.
Therefore, there is a well-defined group homomorphism:
\begin{equation}\label{mcg2outg}
\mcg(S)\to \outg(\pi_1(S)),
\end{equation}
which sends any $[f]\in\mcg(S)$
to $[\gamma\mapsto\tilde{f}\circ\gamma\circ \tilde{f}^{-1}]\in\outg(\pi_1(S))$.

For any orientation-preserving self-homeomorphism $f\colon S\to S$,
denote by
\begin{equation}\label{M_f_def}
M_f=\frac{S\times\Real}{(x,r+1)\sim(f(x),r)}
\end{equation}
the mapping torus of $f$ (following the dynamical convention).
Topologically $M_f$ is a compact $3$--manifold with empty or tori boundary.
The forward suspension flow 
\begin{equation}\label{theta_t_def}
\theta_t\colon M_f\to M_f
\end{equation}
parametrized by $t\in\Real$ is induced by $(x,r)\mapsto (x,r+t)$,
so $\theta_1$ can also be induced by $(x,r)\mapsto (f(x),r)$.
The projection $S\times \Real\to \Real$ onto the second factor 
induces a distinguished bundle structure of $M_f$
over the oriented circle $\Real/\Integral$ with a distinguished fiber projected by $S\times\{0\}$.
The homotopy class of the bundle projection
therefore represents a distinguished first integral cohomology class
\begin{equation}\label{phi_f_def}
\phi_f\in H^1(M_f;\Integral).
\end{equation}
The product space $\widetilde{S}_{\univ}\times\Real$ provides a distinguished universal cover of $M_f$.
The covering projection is the composition of maps
$\widetilde{S}_{\univ}\times\Real\longrightarrow S\times\Real \longrightarrow M_f,$
where the first map is the product $\kappa_{\univ}\times\mathrm{id}_{\Real}$,
and the second map the quotient projection.

Denote by $\pi_1(M_f)$ the group of deck transformations acting
on $\widetilde{S}_{\univ}\times\Real$.
The cohomology class $\phi_f$ induces a distinguished quotient homomorphism
of $\pi_1(M_f)$ onto the infinite cyclic group ${\Integral}$,
(which is the deck transformation group for $\Real\to\Real/\Integral$ acting on $\Real$ by translation).
The kernel of the quotient homomorphism is the normal subgroup of $\pi_1(M_f)$ 
which stabilizes every fiber $\widetilde{S}_{\univ}\times\{r\}$.
We identify the fiber at $0$ canonically with $\widetilde{S}_\univ$
and identify its stabilizer equivariantly with $\pi_1(S)$.
Written together with the profinite completions,
there is a commutative diagram of distinguished group homomorphisms:
\begin{equation}\label{ses_mapping_torus_group}
\xymatrix{
\{1\} \ar[r] & \pi_1(S) \ar[r] \ar[d]^{\textrm{incl.}} & \pi_1(M_f) \ar[r]^{\phi_f} \ar[d]^{\textrm{incl.}} &  {\Integral} \ar[r] \ar[d]^{\textrm{incl.}} & \{0\} \\
\{1\} \ar[r] & \widehat{\pi_1(S)} \ar[r] & \widehat{\pi_1(M_f)} \ar[r]^{\widehat{\phi}_f} &  \widehat{\Integral} \ar[r] & \{0\} 
}
\end{equation}
whose rows are short exact sequences.

There are two kinds of situations that we encounter 
repeatedly in the rest of this paper, 
either to verify properties of individual mapping classes 
that are determined by their procongruent conjugacy classes,
or to identify common features for procongurently conjugate pairs of mapping classes.
Frequently we also study proconguent conjugacy between mapping classes
through profinite completions of mapping-torus groups.
We formalize these situations for the convenience of subsequent discussion:

\begin{convention}\label{settings}\
	\begin{enumerate}
	\item A \emph{monologue setting} $(S,f)$ refers to the following hypotheses and notations.
	Suppose that $S$ is an orientable connected compact surface, together 
	with a fixed universal covering space $\widetilde{S}_\univ\to S$;
	suppose that $f\colon S\to S$ is an orientation-preserving self-homeomorphism of $S$.
	Denote by $M_f$ the mapping torus of $f$, 
	and by $\phi_f$ the distinguished cohomology class of $M_f$,
	and by $\pi_1(M_f)$ the deck transformation group 
	which acts on	the distinguished universal covering space $\widetilde{S}_\univ\times\Real$ of $M_f$.
	\item A \emph{dialogue setting} $(S,f_A,f_B)$ refers to the following hypotheses and notations.
	Suppose that $S$ is an orientable connected compact surface, together 
	with a fixed universal covering space $\widetilde{S}_\univ\to S$;
	suppose that $f_A,f_B\colon S\to S$ are
	a pair of orientation-preserving self-homeomorphisms of $S$.
	Denote by $M_A$ the mapping torus of $f_A$,
	and by $\phi_A$ the distinguished cohomology class of $M_A$,
	and by $\pi_1(M_A)$ the deck transformation group acting on $\widetilde{S}_\univ\times\Real$.
	Denote by $M_B,\phi_B,\pi_1(M_B)$ those objects associated to $f_B$, accordingly.
	\end{enumerate}
\end{convention}

\section{Procongruent conjugacy}\label{Sec-procongruent_conjugacy}
In this section, we introduce the notion of procongruent conjugacy.
We first define this relation 
on outer automorphism groups (Definition \ref{procongruent_conjugacy_outer}),
and then on mapping class groups (Definition \ref{procongruent_conjugacy_mcg}).
We provide some preliminary characterizations and comments
to justify our formulation.

\subsection{Procongruent conjugacy for outer automorphisms}
In group theory, the isomorphism type of a residually finite group has been studied through 
the set of all the isomorphism classes of its finite quotient groups,
sometimes called the \emph{genus} of the group.
For finitely generated groups,
the same amount of information is provided in a more organized form
by the profinite completion of the group.
We refer the reader to \cite[Chapter 2]{Ribes--Zalesskii_book}
for general background on profinite group theory.
Procongruent conjugacy between outer automorphisms
is supposed to be an appropriate analogue
of profinite isomorphism between groups.
 
We recall the following standard terminology in group theory.
For any (abstract) group $G$,
the \emph{outer automorphism group} of $G$ refers to 
the quotient group of the automorphism group $\autg(G)$ by the inner automorphism group $\inng(G)$,
denoted as $\outg(G)$. 
A subgroup $K$ of $G$ is said to be \emph{characteristic}
if it is invariant under every automorphism of $G$.
For any characteristic subgroup $K$, 
there is a canonical group homomorphism $\outg(G)\to\outg(G/K)$,
which descends from $\autg(G)\to\autg(G/K)$.
The \emph{profinite completion} of $G$ refers to 
the group which is 
the inverse limit of the inverse system of all the finite quotient groups of $G$,
denoted as $\widehat{G}$.
Unless otherwise declared, we do not regard $\widehat{G}$ as a topological group,
(see Remark \ref{topological_group_category}).
There is a natural group homomorphism $G\to \widehat{G}$.
We say that $G$ is \emph{residually finite} if $G\to \widehat{G}$ is injective.
There is a canonical group homomorphism $\outg(G)\to\outg(\widehat{G})$,
which descends from $\autg(G)\to\autg(\widehat{G})$.

\begin{definition}\label{procongruent_conjugacy_outer}
	Let $G$ be a finitely generated residually finite group.
	\begin{enumerate}
	\item
	For any characteristic finite-index subgroup $K$ of $G$,
	a pair of outer automorphisms of $G$ are said to be
	\emph{congruently conjugate} modulo $K$,
	if they induce a conjugate pair of outer automorphisms in $\outg(G/K)$.
	\item
	A pair of outer automorphisms of $G$ are 
	said to be \emph{procongruently conjugate}
	if they induce a conjugate pair of outer automorphisms in $\outg(\widehat{G})$.
	\end{enumerate}
\end{definition}

\begin{proposition}\label{characterization_congruent_conjugacy}
	Let $G$ be a finitely generated residually finite group.
	Then, 
	a pair of outer automorphisms of $G$ are procongurently conjugate 
	if and only if for every finite-index subgroup $H$ of $G$,
	there is a finite-index characteristic subgroup $K$ of $G$ contained in $H$,
	such that the pair of outer automorphisms
	are congruently conjugate modulo $K$.
\end{proposition}

\begin{proof}
	As $G$ is finitely generated and residually finite, 
	the subgroups 
	$$K_n=\bigcap\{H\leq G\colon [G:H]\leq n\}$$ of $G$
	are characteristic of finite-index, for all $n\in\Natural$. 
	The profinite completion $\widehat{G}$ of $G$ 
	can be identified with the inverse limit
	of the inverse system $G\to G/K_n$, indexed by $n\in\Natural$.
	Moreover,
	there are only finitely many subgroups $\widehat{H}$ in $\widehat{G}$
	of any given finite index \cite[Proposition 2.5.1]{Ribes--Zalesskii_book}.
	In particular, 
	the intersection $\bigcap\{\widehat{H}\leq \widehat{G}\colon [\widehat{G}:\widehat{H}]\leq n\}$
	can be identified as the kernel of $\widehat{G}\to G/K_n$,
	which is also isomorphic to the profinite completion $\widehat{K}_n$ of $K_n$.
	Observe that any automorphism of $\widehat{G}$ preserves $\widehat{K}_n$.
	Therefore, 
	we have natural homomorphisms 
	$$\autg(\widehat{G})\to\autg(\widehat{G}/\widehat{K}_n)\cong\autg(G/K_n),$$
	and 
	$$\outg(\widehat{G})\to\outg(\widehat{G}/\widehat{K}_n)\cong\outg(G/K_n),$$
	for all $n\in\Natural$.
	
	Below we adopt the notation $[\tau]\in\outg(G)$ for outer automorphisms of $G$,
	where $\tau\in\autg(G)$ indicates some representative.
	For any $[\tau]\in\outg(G)$, denote by 
	$$[\tau]_n\in\outg(G/K_n)$$
	the image under the natural homomorphism $\outg(G)\to\outg(G/K_n)$.
			
	To prove the `only if' direction, it suffices to show that any procongruently conjugate
	pair $[\tau_A],[\tau_B]\in\outg(G)$
	give rise to conjugate pairs $[\tau_A]_n,[\tau_B]_n\in\outg(G/K_n)$ for all $n\in\Natural$.
	This follows immediately from the fact that the canonical homomorphism $\outg(G)\to\outg(G/K)$ 
	factorizes canonically as $\outg(G)\to\outg(\widehat{G})\to\outg(G/K_n)$,
	for all $n\in\Natural$.

	The `if' direction follows 
	from a general argument by sequential compactness.
	To be precise,
	let $[\tau_A],[\tau_B]\in\outg(G)$ be a pair of outer automorphisms of $G$.
	Applying the assumption to any $H=K_n$,
	we see that they are congruently conjugate modulo some finite-index characteristic subgroup
	of $G$ that is contained in $K_n$.
	In particular, 
	$[\tau_A]_n,[\tau_B]_n\in\outg(G/K_n)$ are conjugate for all $n\in\Natural$.
	In other words,
	$[\tau_A]_n=[\sigma_n][\tau_B]_n[\sigma_n]^{-1}$ for some $[\sigma_n]\in\outg(G/K_n)$,	
	and for all $n\in\Natural$.
	As $\outg(G/K_1)$ is finite,
	there exists some $[\hat{\sigma}_1]\in\outg(G/K_1)$ 
	and $[\sigma_{n'}]_1=[\hat{\sigma}_1]$ holds for
	some infinite subsequence $[\sigma_{n'}]\in\outg(G/K_{n'})$,
	indexed by $n'\in I'$ for some $I'\subset\Natural$.
	As $\outg(G/K_2)$ is finite,
	there exists some $[\hat{\sigma}_2]\in\outg(G/K_2)$,
	and $[\sigma_{n''}]_2=[\hat{\sigma}_2]$ holds for
	some further infinite subsequence $[\sigma_{n''}]\in\outg(G/K_{n''})$,
	indexed by $n''\in I''$ for some $I''\subset I'$.
	Moreover, we may require that $\autg(G/K_2)\to\autg(G/K_1)$
	sends $\hat{\sigma}_2$ to $\hat{\sigma}_1$,
	possibly after modifying the representative $\hat{\sigma}_2$ 
	by composition with an inner automorphism of $G/K_2$.
	Continuing in this way,
	we extract increasingly further infinite subsequences,
	and obtain elements $[\hat{\sigma}_n]\in\outg(G/K_n)$ for all $n\in\Natural$.
	The construction guarantees that 
	$\autg(G/K_n)\to\autg(G/K_m)$ sends $\hat{\sigma}_n$ to $\hat{\sigma}_m$,
	for any $m,n\in\Natural$ with $m\leq n$.
	Then the sequence $\hat{\sigma}_n\in\outg(G/K_n)$ defines a unique $\hat{\sigma}\in\autg(\widehat{G})$
	by the inverse limit.
	The construction guarantees
	$[\tau_A]_n=[\hat{\sigma}_n][\tau_B]_n[\widehat{\sigma}_n]^{-1}$ in $\outg(G/K_n)$, for all $n\in\Natural$.
	Therefore,
	$[\widehat{\tau}_A]=[\hat{\sigma}][\widehat{\tau}_B][\hat{\sigma}]^{-1}$
	holds in $\outg(\widehat{G})$, 
	where $\widehat{\tau}_A,\widehat{\tau}_B\in\autg(\widehat{G})$ 
	stand for the profinite completion extension of $\tau_A,\tau_B$, respectively.
	In other words, $[\tau_A],[\tau_B]\in\outg(G)$ are procongruently conjugate.
\end{proof}

\begin{remark}\label{profinitely_characteristic}\
\begin{enumerate}
\item
The condition in Proposition \ref{characterization_congruent_conjugacy}
is \emph{not} equivalent to 
congruent conjugacy modulo every characteristic finite-index subgroup.
Rather,
one may introduce \emph{profinitely characteristic} finite-index subgroups
for any finitely generated residually finite group $G$.
These refer to subgroups of the form $G\cap \widehat{K}$ 
where $\widehat{K}$ is any finite-index characteristic subgroup of $\widehat{G}$.
For example, 
the level--$n$ characteristic subgroups $K_n$ in the above proof
are all profinitely characteristic.
With this notion, one may actually show (with essentially the same argument) that profinite conjugacy 
is equivalent to congruent conjugacy modulo every profinitely characteristic finite-index subgroup.
Profinitely characteristic finite-index subgroups
are always characteristic,
but the converse is false in general.
See Appendix \ref{Sec-profinitely_characteristic} for a counter example.
\item
The proof of Proposition \ref{characterization_congruent_conjugacy}
implies $\outg(\widehat{G})=\varprojlim_n \outg(G/K_n)$.
It follows that every conjugacy class in $\outg(\widehat{G})$ 
is the inverse limit of (an inverse system of) conjugacy classes in $\outg(G/K_n)$.
\end{enumerate}
\end{remark}

\begin{remark}\label{topological_group_category}
Working in the topological group category 
actually leads to the same notion of procongruent conjugacy,
and the same characterization via congruent conjugacy modulo
characteristic finite-index subgroups.
This is explained as follows.

Let $G$ be a finitely generated residually finite group.
The profinite completion $\widehat{G}$ of $G$
can be naturally furnished with the profinite topology,
which is the weakest topology to keep
all the quotient homomorphisms onto discrete finite groups continuous.
This makes $\widehat{G}$ a profinite topological group, 
or equivalently, a totally disconnected, compact, Hausdorff group.
It is topologically finitely generated, and indeed,
$G$ is naturally included a finitely generated dense subgroup, by residual finiteness.
The abstract automorphisms of $\widehat{G}$ are all homeomorphisms with respect to the profinite topology,
thanks to a deep result of Nikolov--Segal \cite{Nikolov--Segal}.
Therefore, the group of topological automorphisms of $\widehat{G}$ 
coincides with the group of abstract automorphisms $\autg(\widehat{G})$.
Now $\autg(\widehat{G})$ can be naturally furnished with the congruence subgroup topology,
which agrees with the compact-open topology,
as $\autg(\widehat{G})$ acts continuously on $\widehat{G}$.
This makes $\autg(\widehat{G})$ a profinite topological group.
The natural group homomorphism $\widehat{G}\to \autg(\widehat{G})$
that represents $\widehat{G}$ as inner automorphisms
is continuous. Since $\widehat{G}$ is compact with closed center
and $\autg(\widehat{G})$ Hausdorff,
the inner automorphism subgroup $\inng(\widehat{G})$
is a closed normal subgroup of $\autg(\widehat{G})$ 
under the congruence subgroup topology.
Therefore, the topological outer automorphism group 
(by definition, the quotient of the topological group of topological automorphisms
by the closure of the normal subgroup of inner automorphisms)
coincides with the outer abstract-automorphism group $\outg(\widehat{G})$.
Furnished with the quotient topology,
$\outg(\widehat{G})$  is again a profinite topological group.
Therefore, Definition \ref{procongruent_conjugacy_outer} 
will remain equivalent 
if we declare $\outg(\widehat{G})$ as the topological outer automorphism group,
and Proposition \ref{characterization_congruent_conjugacy} will remain true 
in that sense.
\end{remark}

\subsection{Procongurent conjugacy for mapping classes}
We introduce procongruent conjugacy 
for mapping classes through the induced outer automorphisms.
We characterize this relation in terms of mapping tori.

\begin{definition}\label{procongruent_conjugacy_mcg}
In any dialogue setting $(S,f_A,f_B)$,
we say that $f_A$ and $f_B$ are \emph{procongruently conjugate} 
if they induce a pair of outer automorphism classes in $\outg(\pi_1(S))$
which are procongurently conjugate.
(See Convention \ref{settings}, Definition \ref{procongruent_conjugacy_outer}, and (\ref{mcg2outg}).)
\end{definition}

\begin{definition}\label{profinitely_aligned_isomorphic}
In a dialogue setting $(S,f_A,f_B)$,
a group isomorphism $\Psi$ between the profinite completions
of $\pi_1(M_A)$ and $\pi_1(M_B)$ is said to be \emph{aligned} 
if it satisfies the following commutative diagram of group homomorphisms:
$$\xymatrix{
\widehat{\pi_1({M_A})} \ar[r]^{\widehat{\phi}_{A}} \ar[d]^\Psi_\cong & \widehat{{\Integral}} \ar[d]^{\mathrm{id}} \\
\widehat{\pi_1({M_B})} \ar[r]^{\widehat{\phi}_{B}} & \widehat{{\Integral}}
}$$
If an aligned isomorphism exists, 
we say that $\pi_1({M_A})$ and $\pi_1({M_B})$ are \emph{profinitely aligned-isomorphic}.
(See Convention \ref{settings} and (\ref{ses_mapping_torus_group}).)
\end{definition}

\begin{proposition}\label{characterization_mapping_tori}
In any dialogue setting $(S,f_A,f_B)$, the following statements are equivalent:
\begin{enumerate}
\item
The mapping-torus groups $\pi_1({M_A})$ and $\pi_1({M_B})$ are profinitely aligned-isomorphic. 
\item The mapping classes of $f_A$ and $f_B$ are procongruently conjugate.
\end{enumerate}
\end{proposition}

\begin{proof}
Choose an elevation $\tilde{f}_A$ of $f_A$ to $\widetilde{S}_\univ$.
We obtain an automorphism $\tau_A\in\autg(\pi_1(S))$, 
acting on $\pi_1(S)$ by $\tau_A(\gamma)=\tilde{f}_A\circ\gamma\circ\tilde{f}_A^{-1}$.
We also obtain a deck transformation $t_A\in\pi_1({M_{A}})$, acting on $\widetilde{S}_\univ\times\Real$
by $t_A(\tilde{x},r)=(\tilde{f}_A^{-1}(\tilde{x}),r+1)$.
Observe that $\phi_A(t_A)$ acts on the oriented line $\Real$ as translation by $+1$,
so we obtain $\phi_{f_A}(t_A)=1$,
and moreover,
$$\widehat{\phi}_{f_A}(t_A)=1.$$
The relation $\gamma\,t_A = t_A \,\tau_A(\gamma)$ in $\pi_1({M_{A}})$ gives rise to 
a decomposition $\pi_1({M_{A}})=\langle t_A\rangle \ltimes \pi_1(S)$,
as a semi-direct product of subgroups.
In this way, every element of $\pi_1({M_{A}})$ can be written uniquely as $t_A^m\gamma$ 
with $m\in\Integral$ and $\gamma\in\pi_1(S)$.
Passing to the profinite completion, every element of  $\widehat{\pi_1({M_{A}})}$
can be written uniquely as $t_A^mg$ with $m\in\widehat{\Integral}$ and $g\in\widehat{\pi_1(S)}$,
and the relation
$$g\,t_A = t_A \,\widehat{\tau}_A(g)$$
holds in  $\widehat{\pi_1({M_{A}})}$.
Here we write $\widehat{\tau}_A\in\autg(\widehat{\pi_1(S)})$ for the completion extension of $\tau_A$.
This gives rise to a decomposition
$$\widehat{\pi_1({M_{A}})}=\widehat{\langle t_A\rangle}\ltimes \widehat{\pi_1(S)},$$
again as a semi-direct product of subgroups.
By choosing $\tilde{f}_B$ 
and obtaining $\tau_B,t_B$ for $f_B$, we decompose $\widehat{\pi_1({M_{B}})}$ similarly. 

To show that the statement (1) implies (2),
suppose that there is some isomorphism $\Psi\colon \widehat{\pi_1({M_{A}})}\to \widehat{\pi_1(M_f)}$
fitting into the required commutative diagram.
Let $\Psi_0\in\autg(\widehat{\pi_1(S)})$ be the restriction of $\Psi$, and $h\in\widehat{\pi_1(S)}$
be as determined by the equation $\Psi(t_A)=t_Bh$.
By $\Psi(g\,t_A)=\Psi(t_A\,\widehat{\tau}_A(g))$, we obtain
$\Psi_0(g)\,t_Bh=t_Bh\,\Psi_0(\widehat{\tau}_A(g))$.
Denote by $I_h\in\inng(\widehat{\pi_1(S)})$ the inner automorphism defined by $h$, namely,
$I_h(u)=huh^{-1}$.
It follows $\Psi_0(g)\,t_B=t_B I_h(\Psi_0(\widehat{\tau}_A(g)))$,
and hence $t_B \widehat{\tau}_B(\Psi_0(g))=t_B\,I_h(\Psi_0(\widehat{\tau}_A(g)))$
for all $g\in\widehat{\pi_1(S)}$.
Therefore, we obtain $\widehat{\tau}_B\,\Psi_0=I_h\,\Psi_0\,\widehat{\tau}_A$ in $\autg(\widehat{\pi_1(S)})$,
or equivalently,
$$\widehat{\tau}_B=I_h\,\Psi_0\,\widehat{\tau}_A\,\Psi_0^{-1}.$$
Note that $[\widehat{\tau}_A],[\widehat{\tau}_B]\in\outg(\widehat{\pi_1(S)})$
are by definition the outer automorphisms of $\widehat{\pi_1(S)}$ induced by $[f_A],[f_B]\in\mcg(S)$.
As the outer automorphism classes $[\widehat{\tau}_A],[\widehat{\tau}_B]$ 
are conjugate by $[\Psi_0]\in\outg(\widehat{\pi_1(S)})$,
we see that the mapping classes $[f_A],[f_B]$ are procongruently conjugate.

To show that the statement (2) implies (1), suppose that 
the mapping classes $[f_A],[f_B]\in\mcg(S)$ are procongruently conjugate.
This means that there are 
some $\Psi_0\in\autg(\widehat{\pi_1(S)})$ and $h\in\widehat{\pi_1(S)}$ with the property
$\widehat{\tau}_B=I_h\,\Psi_0\,\widehat{\tau}_A\,\Psi_0^{-1}$,
where $I_h\in\inng(\widehat{\pi_1(S)})$ stands for the inner automorphism given by $h$.
Reversing the calculation in the former implication,
we see that the assignments $\Psi(t_A)=t_Bh$ and $\Psi(g)=\Psi_0(g)$ for $g\in\widehat{\pi_1(S)}$ 
determine a unique group isomorphism
$$\Psi\colon \widehat{\langle t_A\rangle}\ltimes \widehat{\pi_1(S)}\to \widehat{\langle t_B\rangle}\ltimes \widehat{\pi_1(S)}.$$
According to the way the semi-direct products arise,
$\Psi$ is also a group isomorphism between $\widehat{\pi_1({M_{A}})}$ and $\widehat{\pi_1({M_{B}})}$,
and fits into the required commutative diagram.
\end{proof}

\begin{example}\label{Stebe_example}
	Let $S$ be a torus. By parametrizing $S$ as $\Real^2/\Integral^2$, 
	there are canonical identifications $\pi_1(S)\cong \Integral^2$ and $\mcg(S)\cong\mathrm{SL}(2,\Integral)$.
	The characteristic finite-index subgroups of $\pi_1(S)$ 
	are precisely the sublattices $(n\Integral)^2$ for all $n\in\Natural$.
	The outer automorphism groups of the congruence quotients 
	are evidently $\outg((\Integral/n\Integral)^2)\cong \mathrm{GL}(2,\Integral/n\Integral)$
	for all $n\in\Natural$.
	
	In this case, there are pairs of mapping classes $[f_A],[f_B]\in\mcg(S)$
	which are procongruently conjugate but not conjugate in $\mcg(S)$.
	The first example is due to Stebe \cite{Stebe_integer_matrices}:
	\begin{equation*}
	\begin{array}{cc}
	{[f_A]=\left[\begin{array}{cc}188&275\\121&177\end{array}\right],}&
	{[f_B]=\left[\begin{array}{cc}188&11\\3025&177\end{array}\right].}
	\end{array}
	\end{equation*}
	Funar discovers an infinite family of such examples \cite[Proposition 1.3]{Funar_torus_bundles}.
	On the other hand, given any mapping class $[f]\in\mcg(S)$,
	any mapping class $[f']\in\mcg(S)$ which is procongurently conjugate to $[f]$
	falls into one of finitely many conjugation orbits of $\mcg(S)$.
	This fact is not hard to be proved directly; 
	it also follows from \cite[Chapter 8, Proposition 8.6]{Platonov--Rapinchuk_book}
	and \cite[Proposition 1.2]{Funar_torus_bundles}.
\end{example}

Example \ref{Stebe_example} shows that procongruent conjugacy 
is different from conjugacy for mapping classes of a torus.
On the other hand, Bridson--Reid--Wilton \cite[Theorem A]{BRW} shows that
for a once-punctured torus, procongruently conjugate pairs of mapping classes 
always induce conjugate pairs of outer automorphisms.
For orientable connected closed surfaces of negative Euler characteristic,
we do not know whether 
the same conclusion as the once-punctured torus case should hold in general.

\section{Review on twisted invariants}\label{Sec-review_twisted}
In this section, we review twisted homology and twisted Reidemeister torsion.
Example \ref{example_TH} and \ref{example_TRT} specialize the general theory
to the cases of our interest.
We pay particular attention to the bimodule version of twisted homology,
and the homotopy invariance of the twisted Reidemeister torsion that we use.
These points are relevant to our application,
but are scattered, somewhat implicitly, in the literature.

\subsection{Twisted homology}\label{Subsec-twisted_homology}
Let $X$ be a path-connected topological space with a universal covering space.
Fix a universal covering space $\widetilde{X}_\univ$ of $X$,
and denote by $\kappa_\univ$ the covering projection map,
and by $\pi_1(X)$ the group of deck transformations. 

For any right $\Integral\pi_1(X)$--module $V$ and any dimension $n\in\Integral$, 
the $n$--th (singular) \emph{twisted homology} of $X$ 
with coefficients in $V$ is defined as the abelian group
$$H_n(X;V)=H_n\left(V\otimes_{\Integral\pi_1(X)}C_\bullet(\widetilde{X}_\univ),\,\mathbf{1}\otimes\partial_\bullet\right).$$ 
Here $(C_\bullet(\widetilde{X}_\univ),\partial_\bullet)$ stands for the singular $\Integral$--chain complex
of $\widetilde{X}_\univ$, furnished with the canonical left $\Integral\pi_1(X)$--module structure
given by $g\sigma = g\circ \sigma$ for 
all singular $n$--simplices $\sigma\colon \Delta^n\to \widetilde{X}_\univ$ and 
all deck transformations $g\in\pi_1(X)$.

For any path-connected topological space $Y$ together with a chosen universal covering space $\widetilde{Y}_\univ$,
any group homomorphism $f_\sharp\colon \pi_1(Y)\to\pi_1(X)$
together with an $f_\sharp$--equivariant map $\tilde{f}\colon \widetilde{Y}_\univ\to\widetilde{X}_\univ$
induce a map $f\colon Y\to X$, and also a homomorphism of graded abelian groups
$$f_*\colon H_*(Y;V)\to H_*(X;V),$$
where the right $\Integral\pi_1(Y)$--module structure on $V$ is induced via $f_\sharp$.
In fact, $f_*$ depends only on $f_\sharp$ and the homotopy class of $f$.
If one fixes an underlying abelian group $V$, 
twisted homology with coefficients in $V$ 
can be formally treated as a functor, which sends any object
$(X,\widetilde{X}_\univ,\Integral\pi_1(X)\to\mathrm{End}(V)^{\mathrm{op}})$ to a graded abelian group $H_*(X;V)$,
and any morphism $(f,\tilde{f},f_\sharp^*)$ to a graded group homomorphism $f_*$.
Cellular and simplicial twisted homology can be defined likewise
when $X$ is enriched with a CW or simplicial complex structure.
The resulting homology groups are naturally isomorphic to the singular version.
We refer the reader to \cite[Chapter 3, Section 3.H]{Hatcher_book} 
for a comprehensive introduction to twisted homology,
(called \emph{homology with local coefficients} there).

Let $\mathcal{A}$ be an associative ring (with identity).
If $W$ is an $(\mathcal{A},\Integral\pi_1(X))$--bimodule,
(namely, a left $\mathcal{A}$--module 
with an $\mathcal{A}$--linear right multiplication of $\Integral\pi_1(X)$),
the twisted homology $H_n(X;W)$ is also enriched 
with a left $\mathcal{A}$--module structure.
This is naturally induced by the left multiplication of $\mathcal{A}$ 
on the first factor of $W\otimes_{\Integral\pi_1(X)}C_\bullet(\widetilde{X}_\univ)$,
which commutes with $\mathbf{1}\otimes\partial_\bullet$.
In this setting, the functorial homomorphisms $f_*\colon H_*(Y;W)\to H_*(X;W)$ 
as above are automatically $\mathcal{A}$--linear.

\begin{example}\label{example_TH}
Let $R$ be a commutative ring.
Denote by $R[t^{\pm1}]$ the Laurent polynomial ring over $R$ 
in a fixed indeterminant $t$.

Let $S$ be an orientable connected compact surface,
and $f\colon S\to S$ be an orientation-preserving self-homeomorphism.
For quotient homomorphism $\gamma\colon \pi_1(M_f)\to \Gamma$ of the mapping-torus group $\pi_1(M_f)$
onto a finite group $\Gamma$,
we form an $(R[t^{\pm1}]\Gamma,\Integral\pi_1(M_f))$--bimodule $R[t^{\pm1}]\Gamma$,
which is the left regular $R[t^{\pm1}]\Gamma$--module
enriched with the right action $\gamma\otimes\phi_f$ of $\pi_1(X)$.
In other words, for any $g\in\pi_1(M_f)$, the right action of $g$ on $R[t^{\pm1}]\Gamma$
is given by right multiplication of $t^{\phi_f(g)}\gamma(g)$,
where $\phi_f\in H^1(M_f;\Integral)$ stands for the distinguished cohomology class.

The twisted homology 
$$H_*\left(M_f;R\left[t^{\pm1}\right]\Gamma\right)=\bigoplus_{n\in\Integral}H_n\left(M_f;R\left[t^{\pm1}\right]\Gamma\right)$$ 
is therefore a graded left $R[t^{\pm1}]\Gamma$--module.
Occasionally if we want to emphasize the twisting,
we also adopt the notation $H^{\gamma\otimes\phi_f}_*(M_f;R[t^{\pm1}]\Gamma)$.
\end{example}

Twisted versions of group homology and profinite-group homology are also available.
In the sequel, 
we only recall their detail
when it is really necessary.
For systematic treatments on those topics,
see \cite[Chapter III]{Brown_book} and \cite[Chapter 6]{Ribes--Zalesskii_book}.

\subsection{Twisted Reidemeister torsion}
Let $X$ be a connected finite cell complex 
(with unordered and unoriented cells).
Fix a universal covering space $\widetilde{X}_\univ$ of $X$
with the deck transformation group denoted by $\pi_1(X)$.

Let $k$ be a natural number, and $\Omega$ be a (commutative) field.
Let $\alpha\colon \pi_1(X)\to \mathrm{GL}(k,\Omega)$ 
be a representation of $\pi_1(X)$ on $\Omega^k$.
Choose an orientation and a lifting 
for each cell of $X$ to $\widetilde{X}_\univ$,
and an ordering of the lifted cells for each dimension.
We denote by $\mathfrak{c}_n$ the ordered set of the oriented lifted $n$--cells,
and by $\mathfrak{c}_\bullet$ the choice altogether.
Then the $\alpha$--twisted cellular chain complex 
$(\Omega^k\otimes_{\Integral\pi_1(X)}C_\bullet(\widetilde{X}_\univ),\mathbf{1}\otimes\partial_\bullet)$
is spanned over $\Omega$ by a distinguished finite basis,
(which can be written down explicitly 
using the standard basis of $\Omega^k$ and $\mathfrak{c}_\bullet$).
When the derived twisted homology $H_*^\alpha(X;\Omega^k)$ vanishes for all dimensions,
there is a well-defined element
$\tau^\alpha(X,\mathfrak{c}_\bullet;\Omega^r)\in \Omega^\times$,
known as the \emph{torsion} 
of the acyclic based $\alpha$--twisted cellular chain complex of $X$.
Any other choice $\mathfrak{c}'_\bullet$ gives rise to some
$\tau^\alpha(X,\mathfrak{c}'_\bullet;\Omega^k)$
which equals $\tau^\alpha(X,\mathfrak{c}_\bullet;\Omega^r)$
up to a factor $\pm\det_\Omega(\alpha(g))\in\Omega^\times$ for some $g\in\pi_1(X)$.
Being rough on the effect of choices,
it therefore makes sense to speak of the \emph{twisted Reidemeister torsion} 
for $(X,\alpha)$, assuming acyclicity and 
allowing certain multiplicative indeterminacy. 
(See \cite[Chapter I, Section 1]{Turaev_book_torsion}.)

\begin{example}\label{example_TRT}
Let $X$ be a connected finite cell complex.
Let $\Omega$ be the field of rational functions $\mathbb{F}(t)$,
over a field $\mathbb{F}$ in a fixed indeterminant $t$.
For any cohomology class $\phi\in H^1(X;\Integral)$ and 
any representation $\rho\colon \pi_1(X)\to \mathrm{GL}(k,\mathbb{F})$,
we form a representation
$$\rho\otimes\phi\colon \pi_1(X)\to \mathrm{GL}\left(k,\mathbb{F}(t)\right)$$
by assigning $g\mapsto t^{\phi(g)}\rho(g)$.

When the twisted homology $H^{\rho\otimes\phi}_*(X;\mathbb{F}(t)^k)$ vanishes,
the \emph{twisted Reidemeister torsion} for $(X,\phi,\rho)$, denoted as
$$\tau^{\rho\otimes\phi}\left(X;\mathbb{F}(t)^k\right)\in\mathbb{F}(t)^\times,$$
is considered to be 
well-defined in $\mathbb{F}(t)^\times$
up to monomial factors with nonzero coefficients,
namely, 
$rt^m$ for $r\in\mathbb{F}^\times$ and $m\in\Integral$.
It is represented by any $\tau^{\rho\otimes\phi}(X,\mathfrak{c}_\bullet;\mathbb{F}(t)^r)$,
where $\mathfrak{c}_\bullet$ is a choice about cell lifting, orientation, and ordering,
as explained above.
By convention, 
we define $\tau^{\rho\otimes\phi}(X;\mathbb{F}(t)^k)=0$ in $\mathbb{F}(t)$
if $H^{\rho\otimes\phi}_n(X;\mathbb{F}(t)^k)\neq0$ holds for some dimension $n$.

This version of twisted Reidemeister torsions can be determined by the twisted 
Alexander polynomials for $(X,\rho,\phi)$.
Recall that 
for any finitely generated module over a Noetherian commutative unique factorization domain (UFD),
the \emph{order} of the module is any generator of the smallest principal ideal that contains
the zeroth elementary ideal of that module; the \emph{rank} of the module 
is the dimension over the field of fractions, under extension of scalars.
Note that the order is nonzero if and only if 
the module has nonvanishing rank. 

For each dimension $n$, the $n$--th \emph{twisted Alexander polynomial}
for $(X,\rho,\phi)$ is defined to be the order of 
the twisted homology $H^{\rho\otimes\phi}_n(X;\mathbb{F}[t^{\pm1}]^k)$ over $\mathbb{F}[t^{\pm1}]$,
and denoted by any representative
$$\Delta^{\rho\otimes\phi}_{X,n}\in \mathbb{F}\left[t^{\pm}\right],$$
which is unique up to monomial factors with nonzero coefficients.
Then $\tau^{\rho\otimes\phi}(X;\mathbb{F}(t)^k)=0$
holds if and only if $\Delta^{\rho\otimes\phi}_{X,n}=0$ holds for some $n$.
Otherwise, the following formula holds in $\mathbb{F}(t)^\times$:
\begin{equation}\label{tau_Delta}
\tau^{\rho\otimes\phi}\left(X;\mathbb{F}(t)^k\right)\doteq
\frac{\prod_{n\textrm{ odd}}\Delta^{\rho\otimes\phi}_{X,n}}{\prod_{n\textrm{ even}}\Delta^{\rho\otimes\phi}_{X,n}}.
\end{equation}
Here the dotted equal symbol stands for equality up to monomial factors with nonzero coefficients.
Note that the products in (\ref{tau_Delta}) are essentially finite since 
$H_n^{\rho\otimes\phi}(X;\mathbb{F}[t^{\pm1}]^k)$ vanishes for all but finitely many $n$.
See \cite[Chapter I, Theorem 4.7]{Turaev_book_torsion}; compare \cite[Section 3.2]{Friedl--Vidussi_survey}.

In particular, 
(\ref{tau_Delta}) makes it clear that
the cellular structure is only auxiliary
for defining $\tau^{\rho\otimes\phi}(X;\mathbb{F}(t)^k)$.
We extend the above setting, 
and allow $X$ to be 
any topological space 
which is homotopy equivalent to a connected finite cell complex.
\end{example}

More generally,
there is a notion of torsion for
any finitely generated acyclic based chain complex of free left $\mathcal{A}$--modules.
Here $\mathcal{A}$ can be any associative ring (with identity),
only assuming that 
the cartesian product rings $\mathcal{A}^l$ are mutually non-isomorphic for all $l\in\Natural$.
The torsion lives in the abelian group $K_1(\mathcal{A})$, 
which can be characterized as
the abelianization of the stable general linear group $\mathrm{GL}(\mathcal{A})$ over $\mathcal{A}$.
For any (commutative) field $\Omega$, 
the determinant function yields is a canonical group isomorphism $K_1(\Omega)\cong\Omega^\times$.
Under the acyclicity assumption and allowing some indeterminacy caused by choices,
twisted Reidemeister torsions can be defined for $(X,\alpha)$ 
given $\alpha\colon\pi_1(X)\to\mathrm{GL}(k,\mathcal{A})$.
(See \cite[Chapter II, Section 6]{Turaev_book_torsion}.)
In the general setting, twisted Reidemeister torsions are usually
invariant under simple homotopy equivalence of $X$ along with pull-back of $\alpha$.
However, the version of Example \ref{example_TRT} is 
actually invariant under homotopy equivalence
and representation pull-back,
thanks to the homotopy invariance of twisted Alexander polynomials and twisted homology
 (see Section \ref{Subsec-twisted_homology}).
We record a precise statement of this fact as follows.

\begin{proposition}\label{homotopy_equivalence_invariance}
	Adopt the notations and the assumptions of Example \ref{example_TRT}.
	If $X'$ is a connected finite cell complex, and if $X'\to X$ is a map of homotopy equivalence,
	then 
	$$\tau^{\rho'\otimes\phi'}\left(X';\mathbb{F}(t)^k\right)\doteq \tau^{\rho\otimes\phi}\left(X;\mathbb{F}(t)^k\right)$$
	holds in $\mathbb{F}(t)^\times$ up to monomials with nonzero coefficients,
	where $\rho'\colon \pi_1(X')\to \mathrm{GL}(k,\mathbb{F})$ and $\phi'\in H^1(X';\Integral)$
	denote the pullbacks of $\rho$ and $\phi$, respectively.
\end{proposition}

\section{Twisted homology of mapping tori}\label{Sec-TH_mapping_tori}
In this section, we show that certain twisted homology of a mapping torus
is determined by the procongruent conjugacy class of the mapping class.

\begin{definition}\label{aligned_equivalent_quotient}
	Let $(S,f_A,f_B)$ be a dialogue setting. 
	Let $\Gamma$ be a finite group.
	A pair of homomorphisms $\gamma_A\colon \pi_1({M_{A}})\to\Gamma$ and 
	$\gamma_B\colon \pi_1({M_{B}})\to \Gamma$
	of the mapping-torus groups onto $\Gamma$ are said to be \emph{aligned equivalent} if 
	$\widehat{\gamma}_A=\widehat{\gamma}_B\circ \Psi$ holds
	for some aligned isomorphism between the profinite completions
	$\Psi\colon \widehat{\pi_1({M_{A}})}\to \widehat{\pi_1({M_{B}})}$.
	(See Definition \ref{profinitely_aligned_isomorphic} and Convention \ref{settings}.)
\end{definition}

\begin{theorem}\label{aligned_mapping_torus_homology}
Let $R$ be a commutative profinite ring.
Denote by $R[t^{\pm1}]$ the Laurent polynomial ring over $R$ 
in a fixed indeterminant $t$.

In any dialogue setting $(S,f_A,f_B)$, 
suppose that 
$\gamma_A\colon \pi_1(M_A)\to\Gamma$ and $\gamma_B\colon \pi_1(M_B)\to\Gamma$
are finite quotients of the mapping-torus groups which are aligned equivalent.
Then there is an isomorphism of graded left $R[t^{\pm1}]\Gamma$--modules:
$$H_*\left({M_{A}};R\left[t^{\pm1}\right]\Gamma\right)\cong 
H_*\left({M_{B}};R\left[t^{\pm1}\right]\Gamma\right),$$
between twisted homologies of mapping tori.
(See Example \ref{example_TH}, Definition \ref{aligned_equivalent_quotient}, and Convention \ref{settings}).
\end{theorem}

The rest of this section is devoted to the proof of Theorem \ref{aligned_mapping_torus_homology}.
The idea of the calculation is to apply the Serre spectral sequence, which is quite routine.
On the other hand, we need to unwrap the construction in enough detail
to make it clear that the asserted isomorphism is left $R[t^{\pm1}]\Gamma$--linear.

Let $S$ be an orientable connected compact surface.	
Fix a universal covering space $\widetilde{S}_\univ$ of $S$.
Let $R$ be a commutative profinite ring.
For any orientation-preserving self-homeomophism $f$ of $S$
and any finite quotient of the mapping-torus group $\pi_1(M_f)\to \Gamma$,
the restricted homomorphism $\pi_1(S)\to\Gamma$ extends to be a homomorphism
$\widehat{\pi_1(S)}\to\Gamma$ by completion.
As $\Gamma$ is finite, 
$R\Gamma$ is also a right profinite $[\![R\widehat{\pi_1(S)}]\!]$--module,
(see \cite[Chapter 5, Sections 5.1 and 5.3]{Ribes--Zalesskii_book}
for terms and notations).

Here and below, 
we identify $\Integral[t^{\pm1}]$ canonically with the group algebra
of the infinite cyclic group $\Integral$ over the commutative ring $\Integral$, 
such that $t^m$ corresponds to the group element $m\in \Integral$.
The twisted profinite group homology 
$\widehat{H}_*(\widehat{\pi_1(S)};R\Gamma)$
is a graded $( R\Gamma, \Integral[t^{\pm1}] )$--bimodule 
in a natural way.
The graded left $R\Gamma$--module structure can be described
by the following natural identifications of graded left $R\Gamma$--modules:
\begin{equation}\label{I_inverse_limit}
\widehat{H}_*\left(\widehat{\pi_1(S)};R\Gamma\right)\cong\varprojlim_{I}\,H_*\left(\widehat{\pi_1(S)};(R/I)\Gamma\right)
\cong\varprojlim_{I}\,H_*\left(\pi_1(S);(R/I)\Gamma\right),
\end{equation}
where $I$ runs over the inverse system of the ideals of $R$ 
with finite quotients $R/I$.
(See \cite[Chapter 6, Corollary 6.1.10 (a) and Section 6.3]{Ribes--Zalesskii_book} for the first isomorphism.
The second isomorphism holds since $\pi_1(S)$ is good in the sense of Serre; 
see \cite{GJZ_goodness}.)
The $R\Gamma$--linear (right) multiplication of 
the commutative group ring $\Integral[t^{\pm1}]$
on each $H_*(\pi_1(S);(R/I)\Gamma)$ can be described concretely as follows.

If $S$ is a sphere, since $\pi_1(M_f)\to\Gamma$ 
is only a finite cyclic quotient ${\Integral}\to\Gamma$,
the action of the infinite cyclic group ${\Integral}$ 
is nothing but the multiplication using its image in $\Gamma$.
Otherwise $S$ is aspherical,
then there is a natural identification of graded left $R\Gamma$--modules
\begin{equation}\label{surface_Shapiro}
H_*\left(\pi_1(S);(R/I)\Gamma\right)\cong H_*\left(\widetilde{S}_{\Gamma};R/I\right),
\end{equation}
where $\widetilde{S}_{\Gamma}$ stands for the fiber product 
$\Gamma\times_{\pi_1(M_f)}(\widetilde{S}_\univ\times\Integral)$,
by Shapiro's lemma (see \cite[Chapter III, Proposition 6.2]{Brown_book}).
In this model,
the $R$--linear action of $\Gamma$ on $H_*(\widetilde{S}_{\Gamma};R/I)$
is induced by the left multiplication on the factor $\Gamma$ of the fiber product,
which can be viewed as deck transformation over $S$.
The $R\Gamma$--linear (right) action of $t^m$, for any $m\in\Integral$,
is induced by $(x,r)\mapsto(x,r+m)$
on the factor $\widetilde{S}_\univ\times\Integral$,
which can be viewed as the forward suspension flow in discrete times.
These $R\Gamma$--linear actions of $t^m$
on all $H_*(\pi_1(S);(R/I)\Gamma)$ are obviously compatible 
with the inverse system of finite ring quotients $R\to R/I$,
so their inverse limit is 
a graded $R\Gamma$--linear action of $t^m$
on $\widehat{H}_*(\widehat{\pi_1(S)};R\Gamma)$.

\begin{lemma}\label{profinite_group_homology_surface}
	The isomorphism class of $\widehat{H}_*(\widehat{\pi_1(S)};R\Gamma)$,
	as a graded $(R\Gamma,\Integral[t^{\pm1}])$--bimodule,
	depends only on the procongruent conjugacy class of the mapping class of $f$
	and the align-equivalence class of the finite quotient $\pi_1(M_f)\to\Gamma$.
\end{lemma}

\begin{proof}
	Although it should be parallel to work with twisted profinite-group homology,
	we argue with twisted group homology, as it is conceptually simpler.
	This is done below by taking advantage of the canonical identification
	\begin{equation}\label{K_inverse_limit}
	\widehat{H}_*\left(\widehat{\pi_1(S)};R\Gamma\right)\cong
	\varprojlim_{K}\,H_*\left(\pi_1(S)/K;R\Gamma\right),
	\end{equation}
	where $K$ runs over the inverse system of finite-index normal subgroups $K$ of $\pi_1(S)$,
	(see \cite[Chapter 6, Corollary 6.1.10 (b) and Section 6.3]{Ribes--Zalesskii_book}).
	It should also be possible to give an {\it ad hoc} proof
	based on (\ref{I_inverse_limit}) and (\ref{surface_Shapiro}).
	We prefer to present a general one, 
	which, in particular, does not appeal to goodness of $\pi_1(S)$.
	When working with $H_*\left(\pi_1(S)/K;R\Gamma\right)$,
	it suffices to assume $R$ to be a commutative ring. 
	However, the profiniteness assumption is needed for legitimating
	$\widehat{H}_*(\widehat{\pi_1(S)};R\Gamma)$ 
	and for establishing (\ref{K_inverse_limit}).

	For any finite-index normal subgroup $K$ of $\pi_1(S)$
	the twisted group homology $H_*(\pi_1(S)/K;R\Gamma)$ is a graded left $R\Gamma$--module.
	More precisely, 
	it can be defined as the homology of the chain complex 
	$(R\Gamma\otimes_{\Integral(\pi_1(S)/K)}C_\bullet(\pi_1(S)/K),\mathbf{1}\otimes\partial_\bullet)$,
	where $(C_\bullet(\pi_1(S)/K),\partial_\bullet)$ 
	stands for the right-homogeneous chain complex of $\pi_1(S)/K$.
	For any dimension $n$, the left $\Integral(\pi_1(S)/K)$--module $C_n(\pi_1(S)/K)$ 
	is freely generated by the right-homogeneous tuples $[g_0,\cdots,g_n]$ for all $g_0,\cdots,g_n\in\pi_1(S)/K$,
	identifying $[g_0,\cdots,g_n]$ with $[g_0g,\cdots,g_ng]$ for all $g\in\pi_1(S)/K$.
	The left action of $\pi_1(S)/K$ is given by $[g_0,\cdots,g_n]\mapsto[hg_0,\cdots,hg_n]$ for all $h\in\pi_1(S)/K$,
	and the $n$--th boundary operator $\partial_n$ is $\Integral$--linearly determined by
	$\partial_n[g_0,\cdots,g_n]=\sum_i(-1)^{i}[g_0,\cdots,\hat{g}_i,\cdots,g_n]$.
	Tensoring with $R\Gamma$ on the left over $\Integral(\pi_1(S)/K)$
	turns the chain complex into a left $R\Gamma$--module.
	(See \cite[Chapter III, Section 1]{Brown_book}, 
	flipping left and right to match with our convention.)
	
	Moreover, if $K$ is a characteristic subgroup of $\pi_1(S)$, 
	we obtain a short exact sequence of group homomorphisms,
	$$\xymatrix{
	\{1\} \ar[r] & \pi_1(S)/K \ar[r] & \pi_1(M_f)/K \ar[r] & {\Integral} \ar[r] & \{0\},
	}$$
	observing that $K$ is normal in $\pi_1(M_f)$. The inner automorphic action of $\pi_1(M_f)/K$
	on $\pi_1(S)/K$ induces an $R\Gamma$--linear right action of $\pi_1(M_f)/K$
	on $R\Gamma\otimes_{\Integral(\pi_1(S)/K)}C_n(\pi_1(S)/K)$ for each $n$.
	Explicitly, 
	any $h\in\pi_1(M_f)/K$ sends any $w\otimes[g_0,\cdots,g_n]$ to $w\otimes[hg_0h^{-1},\cdots,hg_nh^{-1}]$,
	which equals $w\gamma(h)\otimes[g_0,\cdots,g_n]$, ($\gamma$ denoting 
	$\pi_1(M_f)\to\Gamma$ as given).
	Passing to homology, the action descends to a graded $R\Gamma$--linear action of 
	the infinite cyclic group ${\Integral}$ on $H_*(\pi_1(S)/K;R\Gamma)$.
	We also point out that 
	the limit action of ${\Integral}$ on $\widehat{H}_*(\widehat{\pi_1(S)};R\Gamma)$
	via (\ref{K_inverse_limit})	actually agrees with the one described via
	(\ref{I_inverse_limit}) and (\ref{surface_Shapiro}).
	(The agreement follows easily from natural translation between homology theories.
	It is not needed for proving Theorem \ref{aligned_mapping_torus_homology}.)
	
	Suppose that $f_A$ and $f_B$ are orientation-preserving self-homeomorphisms of $S$
	of procongruently conjugate mapping classes,
	and that $\gamma_A\colon\pi_1({M_{A}})\to\Gamma$
	and $\gamma_B\colon \pi_1(M(f(B))\to \Gamma$ 
	are aligned equivalent finite quotients.
	Let $\Psi$ be a witnessing aligned isomorphism between the profinite completions of
	$\pi_1({M_{A}})$ and $\pi_1({M_{B}})$,
	such that $\widehat{\gamma}_B=\widehat{\gamma}_A\circ\Psi$.
	For any finite-index characteristic subgroup $K$ of $\pi_1(S)$
	contained in the (same) kernels of $\gamma_A$ and $\gamma_B$ in $\pi_1(S)$, 
	we see that $\Psi$ induces a group automorphism $\Psi_{K}$ of $\pi_1(S)/K$ 
	and also an $R\Gamma$--linear chain isomorphism
	$$R\Gamma_A\otimes_{\Integral(\pi_1(S)/K)}C_\bullet(\pi_1(S)/K)
	\cong R\Gamma_B\otimes_{\Integral(\pi_1(S)/K)}C_\bullet(\pi_1(S)/K),$$
	by the expression $w\otimes[g_0,\cdots,g_n]\mapsto w\otimes[\Psi_K(g_0),\cdots,\Psi_K(g_n)]$
	for all $n$. 
	Here $R\Gamma_A,R\Gamma_B$ stand for $R\Gamma$ 
	with the indicated right actions of $\pi_1(S)/K$, accordingly.
	Denote by 
	$$\Psi_{K*}\colon H_*(\pi_1(S)/K;R\Gamma_A)\to H_*(\pi_1(S)/K;R\Gamma_B)$$
	the induced isomorphism on homology, as $R\Gamma$--modules.
	The graded $R\Gamma$--linear action of the infinite cyclic group ${\Integral}$
	on $H_*(\pi_1(S)/K;R\Gamma_A)$ coincides with the pull-back via $\Psi_{K*}$
	of the action on $H_*(\pi_1(S)/K;R\Gamma_B)$, by the way they are defined,
	and by our assumption that $\Psi$ witnesses the aligned equivalence. 
	Therefore, $\Psi_{K*}$ is an isomorphism between
	$H_*(\pi_1(S)/K;R\Gamma_A)$ and $H_*(\pi_1(S)/K;R\Gamma_B)$
	as graded $(R\Gamma,\Integral[t^{\pm1}])$--bimodules.

	Since $\pi_1(S)$ is finitely generated and
	$\Gamma$ is finite, the finite-index characteristic subgroups $K$ of $\pi_1(S)$
	form an inverse subsystem,
	which yields the same inverse limit as (\ref{K_inverse_limit}) with respect to $f_A$ and $\gamma_A$.
	The same holds for $f_B,\gamma_B$. 
	Therefore,
	the inverse limit of $\Psi_{K*}$ yields an isomorphism between graded $(R\Gamma,\Integral[t^{\pm1}])$--bimodules
	$$\Psi_*\colon \widehat{H}_*(\widehat{\pi_1(S)};R\Gamma_A)\to\widehat{H}_*(\widehat{\pi_1(S)};R\Gamma_B).$$
	This means precisely the same as the assertion.
\end{proof}

\begin{lemma}\label{npq}
	For all dimensions $n$, 
	there are canonical isomorphisms of left $R[t^{\pm}]\Gamma$--modules:
	$$H_n\left(\pi_1(M_f);R\left[t^{\pm1}\right]\Gamma\right)\cong
	H_0\left({\Integral};\widehat{H}_n\left(\widehat{\pi_1(S)};R\Gamma\right)\otimes_\Integral\Integral\left[t^{\pm1}\right]\right),$$
	where the infinite cyclic group ${\Integral}$ acts on 
	$\widehat{H}_n(\widehat{\pi_1(S)};R\Gamma)\otimes_\Integral\Integral[t^{\pm1}]$ 
	via the induced $(R\Gamma,\Integral[t^{\pm1}])$--bimodule structure on the first factor, 
	and simultaneously, by the multiplication of $t^m$ for all $m\in\Integral$ on the second factor.
\end{lemma}

\begin{proof}
	For all $q$, there are canonical $(R[t^{\pm1}]\Gamma,\Integral[t^{\pm1}])$--bimodule isomorphisms:
	\begin{equation}\label{tensor_flat}
	H_q\left(\pi_1(S);R\left[t^{\pm1}\right]\Gamma\right)\cong
	H_q\left(\pi_1(S);R\Gamma\right)\otimes_\Integral\Integral\left[t^{\pm1}\right].
	\end{equation}
	In fact, 
	there are natural homomorphisms 
	from the right-hand side of (\ref{tensor_flat})
	to the left-hand side, by change of coefficients.
	(That is, applying the functor $\otimes_\Integral\Integral[t^{\pm1}]$
	to the short exact sequence $0\to B_\bullet(\pi_1(S);R\Gamma)\to 
	Z_\bullet(\pi_1(S);R\Gamma)\to H_\bullet(\pi_1(S);R\Gamma)\to 0$,
	relating boundaries, cycles, and homology.)
	Note that $\pi_1(S)$ acts trivially on $\Integral[t^{\pm1}]$.
	The homomorphisms are isomorphisms
	because $\Integral[t^{\pm1}]$ is flat over $\Integral$.
	
	For all $q$, there are canonical isomorphisms of $(R\Gamma,\Integral[t^{\pm1}])$--bimodules:
	$$H_q\left(\pi_1(S);R\Gamma\right) 
	\cong \widehat{H}_q\left(\widehat{\pi_1(S)};R\Gamma\right),$$
	coming from canonical isomorphisms on the chain level, thanks to the finiteness of $\Gamma$.
	
	Furthermore,
	there are canonical isomorphisms 
	between left $R[t^{\pm1}]\Gamma$--modules, for all $n$:
	\begin{equation}\label{SS_Serre}
	H_n\left(\pi_1(M_f);R\left[t^{\pm1}\right]\Gamma\right)
	\cong H_0\left(\Integral;H_n\left(\pi_1(S);R[t^{\pm1}]\Gamma\right)\right).
	\end{equation}
	This follows from the Hochschild--Serre spectral sequence  \cite[Chapter VII, Theorem 6.3]{Brown_book},
	applied with respect to 
	the short exact sequence of group homomorphisms
	$\{1\}\to\pi_1(S)\to\pi_1(M_f)\to{\Integral}\to\{0\}$
	and the $(R[t^{\pm1}]\Gamma,\Integral\pi_1(M_f))$--bimodule $R[t^{\pm1}]\Gamma$,
	as follows.
	
	The $E^2$ page reads
	$E^2_{p,q}=H_p(\Integral;H_q(\pi_1(S);R[t^{\pm1}]\Gamma))$.
	By (\ref{tensor_flat}), we observe that $H_q(\pi_1(S);R[t^{\pm1}]\Gamma)$
	is a finitely generated free $R[t^{\pm1}]$--module.
	(In fact, $H_q(\pi_1(S);R\Gamma)$	can be identified with $H_q(S'_\Gamma;R)$
	where $S'_\Gamma$ is a possibly disconnected regular finite cover over $S$,
	obtained as the quotient space of	$\Gamma\times\widetilde{S}_\univ$
	by the action $u.(h,x)\mapsto (h\gamma(u)^{-1},u.x)$ for all $u\in\pi_1(S)$.)
	The action of any $m\in\Integral$ on $H_q(\pi_1(S);R[t^{\pm1}]\Gamma)$
	takes any homogeneous summand $H_q(\pi_1(S);R\Gamma)\otimes t^j$
	($R$--linearly) isomorphically onto $H_q(\pi_1(S);R\Gamma)\otimes t^{j+m}$.
	Simple calculation shows $E^2_{p,q}=0$ unless $p=0$.
	In particular, all differentials vanish from the $E^2$ page on,
	and there is no extension problem.
	(The associated graded module 
	of the fibration filtration $\mathscr{F}_*H_n$ 
	of $H_n=H_n(\pi_1(M_f);R[t^{\pm1}]\Gamma)$ has homogeneous summands
	$\mathscr{F}_p H_n / \mathscr{F}_{p-1}H_n=E^\infty_{p,n-p}=E^2_{p,n-p}$,
	which is nontrivial only if $p=0$.)
	This yields the isomorphism (\ref{SS_Serre}).
	See Remark \ref{left_linearity} for extra justification regarding 
	the $R[t^{\pm1}]\Gamma$--linearity of (\ref{SS_Serre}).
		
	Putting things together, the asserted canonical isomorphisms follow for all $n$.	
\end{proof}

\begin{remark}\label{left_linearity}
	The $R[t^{\pm1}]\Gamma$--linearity of the isomorphisms 
	in Lemmas \ref{profinite_group_homology_surface} and \ref{npq}
	is obvious 
	only if one keeps track of the constructions of the canonical isomorphisms
	(\ref{surface_Shapiro}), (\ref{K_inverse_limit}), (\ref{SS_Serre}), and (\ref{tensor_flat}).
	For (\ref{surface_Shapiro}) and (\ref{K_inverse_limit}),
	the linearity is inherited from the chain level;
	see \cite[Chapter 6, Corollary 6.1.10 and Section 6.3]{Ribes--Zalesskii_book} for detail.
	For (\ref{SS_Serre}), the point is that
	the Hochschild--Serre spectral sequence
	can be constructed using a filtration of $C_\bullet(\pi_1(M_f))$ which comes from
	the grading of $C_\bullet({\Integral})$,
	and that the filtration is compatible with the left $R[t^{\pm1}]\Gamma$--module structure
	under the twisting operation $R[t^{\pm1}]\Gamma\otimes_{\pi_1(M_f)}$;
	see \cite[Chapter VII, Theorem 6.3]{Brown_book}.
	For (\ref{tensor_flat}), the $R[t^{\pm1}]\Gamma$--linearity can be observed,
	as we have described the construction explicitly.
\end{remark}

To prove Theorem \ref{aligned_mapping_torus_homology}, 
suppose that $f\colon S\to S$ is an orientation-preserving self-homeomorphism
of an orientable connected compact surface $S$,
and $\gamma\colon \pi_1(M_f)\to \Gamma$ is a finite quotient.
For any commutative profinite ring $R$, we can calculate directly when $S$
is a sphere. The result yields $H_*(M_f;R[t^{\pm1}]\Gamma)\cong H_*(S;R\Gamma)$ 
as an isomorphism between graded left $R[t^{\pm1}]\Gamma$--modules,
where $t$ acts trivially on $R\Gamma$. 
When $S$ is aspherical, $M_f$ is also aspherical,
so we obtain 
$H_*(M_f;R[t^{\pm1}]\Gamma)\cong H_*(\pi_1(M_f);R[t^{\pm1}]\Gamma)$
as graded left $R[t^{\pm1}]\Gamma$--modules.
The latter is dealt with by Lemma \ref{npq}.
These results, together with Lemma \ref{profinite_group_homology_surface},
show that the isomorphism class of 
the left $R[t^{\pm1}]\Gamma$--module $H_*(M_f;R[t^{\pm1}]\Gamma)$
depends only on the procongruent conjugacy class of the mapping class of $f$
and the aligned equivalence class of $\gamma$.

This completes the proof of Theorem \ref{aligned_mapping_torus_homology}.

\section{Twisted Reidemeister torsions of mapping tori}\label{Sec-TRT_mapping_tori}
In this section, we show that certain twisted Reidemeister torsion of a mapping torus 
is determined by the procongruent conjugacy class of the mapping class.

\begin{definition}\label{aligned_equivalent_representation}
	Let $k$ be a natural number and $\mathbb{F}$ be a (commutative) field of characteristic $0$.
	Denote by $\mathbb{F}^k$ the column vector space of dimension $k$ over $\mathbb{F}$.
	By a \emph{finite representation} of a group on $\mathbb{F}^k$,
	we mean a homomorphism of the group to 
	the general linear group $\mathrm{GL}(k,\mathbb{F})$ with finite image.
	In any dialogue setting $(S,f_A,f_B)$,
	suppose that the mapping classes of $f_A,f_B$ are procongruently conjugate.	

	A pair of finite representations $\rho_A\colon \pi_1({M_{A}})\to\mathrm{GL}(k,\mathbb{F})$ and 
	$\rho_B\colon \pi_1({M_{B}})\to \mathrm{GL}(k,\mathbb{F})$ 
	of the mapping-torus groups on $\mathbb{F}^k$ are said to be
	\emph{aligned equivalent} if 
	$\widehat{\rho}_A=\widehat{\rho}_B\circ \Psi$ holds
	for some aligned isomorphism between the profinite completions
	$\Psi\colon \widehat{\pi_1({M_{A}})}\to \widehat{\pi_1({M_{B}})}$.
	(See Definition \ref{profinitely_aligned_isomorphic} and Convention \ref{settings}.)
\end{definition}

\begin{theorem}\label{TRT_dialogue}
Let $k$ be a natural number and $\mathbb{F}$ be a field of characteristic $0$.
Denote by $\mathbb{F}(t)$ the field of rational functions over $\mathbb{F}$ in a fixed indeterminant $t$.

In any dialogue setting $(S,f_A,f_B)$, suppose that 
$\rho_A\colon \pi_1(M_A) \to \mathrm{GL}(k,\mathbb{F})$
and $\rho_B\colon \pi_1(M_B) \to \mathrm{GL}(k,\mathbb{F})$
are representations of the mapping torus groups
which are finite and aligned equivalent.
Then the following equality holds in $\mathbb{F}(t)^\times$ up to 
monomial factors with nonzero coefficients:
$$\tau^{\rho_A\otimes\phi_A}\left({M_{A}};\mathbb{F}(t)^k\right)\doteq
\tau^{\rho_B\otimes\phi_B}\left({M_{B}};\mathbb{F}(t)^k\right),$$
between twisted Reidemeister torsions of mapping tori.
(See Example \ref{example_TRT}, Definition \ref{aligned_equivalent_representation}, and Convention \ref{settings}).
\end{theorem}

The rest of this section is devoted to the proof of Theorem \ref{TRT_dialogue}.
The proof relies on Theorem \ref{aligned_mapping_torus_homology}.
We first prove for $\mathbb{F}$ any algebraic closure of a $p$--adic field,
and then deduce the general case by a trick of scalar extensions and restrictions.

Let $k$ be a natural number and $\mathbb{F}$ be a field of characteristic $0$.
Denote by $\mathbb{F}[t^{\pm1}]$ the Laurent polynomial ring over $\mathbb{F}$
in a fixed indeterminant $t$.
Let $f\colon S\to S$ be an orientation-preserving self-homeomophism of 
an orientable connected compact surface $S$, and 
$\rho\colon \pi_1(M_f)\to \mathrm{GL}(k,\mathbb{F})$
be a finite representation of the mapping-torus group $\pi_1(M_f)$ on $\mathbb{F}^k$.

The statement of Theorem \ref{TRT_dialogue}
can be paraphrased as follows: 
For any dimension $n$ and up to units of $\mathbb{F}[t^{\pm1}]$,
the twisted Alexander polynomials $\Delta^{\rho\otimes\phi_f}_{M_f,n}$
depend only on the procongruent conjugacy class of the mapping class of $f$ 
and the aligned equivalence class of $\rho$.
For brevity below, we treat 
the above classes about $f$ and $\rho$ as a package of information,
and refer to it as the aligned equivalence class for $(f,\rho)$.

We adopt the following notations.
Factorize $\rho$ uniquely as the composition of group homomorphisms:
\begin{equation}\label{factor_thru_Gamma}
\xymatrix{ \pi_1(M_f) \ar[r]^-{\gamma} & \Gamma \ar[r]^-{\iota} & \mathrm{GL}(k,\mathbb{F}) }
\end{equation}
where $\Gamma$ stands for the image of $\rho$, 
so $\gamma$ is a finite quotient and $\iota$ an inclusion.
For any $\mathbb{F}$--linear representation $\sigma\colon \Gamma\to \mathrm{GL}(V)$
of $\Gamma$ on a linear space $V$ over $\mathbb{F}$,
using the distinguished cohomology class $\phi_f\in H^1(M_f;\Integral)$
of the mapping torus $M_f$,
we form the $(\mathbb{F}[t^{\pm1}]\Gamma,\Integral\pi_1(M_f))$--bimodule 
$$\mathrm{End}_{\mathbb{F}}(V)\left[t^{\pm1}\right]=\mathrm{End}_{\mathbb{F}}(V)\otimes_{\mathbb{F}}\mathbb{F}\left[t^{\pm1}\right],$$
whose left action of $\Gamma$ is determined by $u\cdot (h\otimes t^m)=(\sigma(u)\circ h)\otimes t^m$, 
and whose right action of $\pi_1(M_f)$ is determined by 
$(h\otimes t^m)\cdot g=(h\circ\sigma(\gamma(g^{-1})))\otimes t^{m+\phi_f(g)}$.
We also form the $(\mathbb{F}[t^{\pm1}],\Integral\pi_1(M_f))$--bimodule
$$V\left[t^{\pm1}\right]=V\otimes_{\mathbb{F}}\mathbb{F}\left[t^{\pm1}\right],$$
whose right action of $\pi_1(M_f)$ is determined by 
$(v\otimes t^m)\cdot g=\sigma(\gamma(g^{-1}))(v)\otimes t^{m+\phi_f(g)}$.

\begin{lemma}\label{isotypic_decomposition}
	Suppose in addition that $\mathbb{F}$ is algebraically closed. 
	Enumerate the isomorphism classes of irreducible representations of $\Gamma$ over $\mathbb{F}$
	by a list of representatives $(V_\chi,\sigma_\chi)$, 
	indexed by the irreducible characters $\chi$ of $\Gamma$ over $\mathbb{F}$.
	
	Then the following statements hold true:
	\begin{enumerate}
	\item For all $n$, there are isomorphisms between left $\mathbb{F}[t^{\pm1}]\Gamma$--modules:
	$$H_n\left(M_f;\mathbb{F}\left[t^{\pm1}\right]\Gamma\right)	\cong
	\bigoplus_\chi\,H_n\left(M_f;\mathrm{End}_{\mathbb{F}}(V_\chi)\left[t^{\pm1}\right]\right).$$
	\item For all $n$ and $\chi$, by restriction of scalars,
	there are isomorphisms between finitely generated left $\mathbb{F}\Gamma$--modules:
	$$H_n\left(M_f;\mathrm{End}_{\mathbb{F}}(V_\chi)\left[t^{\pm1}\right]\right)	\cong 
	\mathrm{Hom}_{\mathbb{F}}\left(V_\chi,\mathbb{F}\right)^{\oplus\dim_{\mathbb{F}}\left(H_n(M_f;V_\chi[t^{\pm1}])\right)},$$
	and isomorphisms between finitely generated $\mathbb{F}[t^{\pm1}]$--modules of vanishing rank:
	$$H_n\left(M_f;\mathrm{End}_{\mathbb{F}}(V_\chi)\left[t^{\pm1}\right]\right)	\cong 
	H_n\left(M_f;V_\chi\left[t^{\pm1}\right]\right)^{\oplus \dim_{\mathbb{F}} \left(V_\chi\right)}.$$
	\end{enumerate}
\end{lemma}

\begin{proof}
	By representation theory of finite groups over algebraically closed fields of characteristic $0$,
	the chain complex 
	$\mathbb{F}[t^{\pm1}]\Gamma\otimes_{\Integral\pi_1(M_f)}C_\bullet(\widetilde{S}_\univ\times\Real)$,
	as a chain complex of left $\mathbb{F}[t^{\pm1}]\Gamma$--modules,
	is isomorphic to the direct sum of
	$\mathrm{End}_{\mathbb{F}}(V_\chi)[t^{\pm1}]\otimes_{\Integral\pi_1(M_f)}C_\bullet(\widetilde{S}_\univ\times\Real)$
	over all the irreducible characters $\chi$ of $\Gamma$ over $\mathbb{F}$.
	So we derive a direct-sum decomposition on the homology level as asserted by the statement (1).
	The direct summands, as left $\mathbb{F}\Gamma$--modules, 
	are isotypic of mutually distinct types $\mathrm{Hom}_{\mathbb{F}}(V_\chi,\mathbb{F})$.
	They are all finite dimensional over $\mathbb{F}$ because $\phi_f$ is a fibered class.
	(More concretely, one may identify $H_*(M_f;\mathbb{F}[t^{\pm1}]\Gamma)$ 
	with the ordinary homology over $\mathbb{F}$
	of the quotient space of $\Gamma\times(\widetilde{S}_\univ\times\Real)$
	by the action $u.(h,(x,r))\mapsto(h\gamma{u}^{-1},u.(x,r))$ of all $u\in\pi_1(S)$,
	using our description of twisted homology in Section \ref{Subsec-twisted_homology};
	the quotient space is homotopy equivalent to a possibly disconnected regular finite cover of $S$
	with deck transformation group $\Gamma$. Then the finite dimensionality follows
	from the statement (1).)	
	Moreover, 
	each $\mathrm{End}_{\mathbb{F}}(V_\chi)[t^{\pm1}]\otimes_{\Integral\pi_1(M_f)}C_\bullet(\widetilde{S}_\univ\times\Real)$
	is isomorphic to 
	$\mathrm{Hom}_\mathbb{F}(V_\chi,\mathbb{F})\otimes_{\mathbb{F}}
	(V_\chi[t^{\pm1}]\otimes_{\Integral\pi_1(M_f)}C_\bullet(\widetilde{S}_\univ\times\Real))$,
	as a chain complex of left $\mathbb{F}[t^{\pm1}]\Gamma$--modules.
	(The first tensorial factor is a left $\mathbb{F}\Gamma$--module 
	and the rest is a chain complex of $\mathbb{F}[t^{\pm1}]$--modules.)
	The asserted isomorphisms of the statement (2) follow immediately from 
	the above decompositions on the chain complex level
	and count of dimensions over $\mathbb{F}$.
\end{proof}

For any rational prime $p\in\Natural$, denote by $\Integral_p$ 
the profinite ring of $p$--adic integers,
and by $\Rational_p$ the field of fractions of $\Integral_p$.
Note that $\Integral_p$ is an integral domain and that $\Rational_p$ has characteristic $0$.
Fix an algebraic closure $\overline{\Rational_p}$ of $\Rational_p$.

\begin{lemma}\label{TRT_dialogue_reduced}
	For any rational prime $p\in\Natural$,
	Theorem \ref{TRT_dialogue} holds true
	with $\mathbb{F}$ being $\overline{\Rational_p}$.
\end{lemma}

\begin{proof}
	Take $\mathbb{F}=\overline{\Rational_p}$ for any rational prime $p$.
	Let $\rho=\iota\circ\gamma$ be the factorization (\ref{factor_thru_Gamma})
	of $\rho$ through its finite image $\Gamma$.
	Enumerate the isomorphism classes of irreducible representations of $\Gamma$ over $\mathbb{F}$
	by a fixed list of representatives $(V_\chi,\sigma_\chi)$, 
	indexed by the irreducible characters $\chi$ of $\Gamma$ over $\mathbb{F}$.
	
	The representation $(\mathbb{F}^k,\iota)$ of $\Gamma$ over $\mathbb{F}$ is 
	isomorphic to a direct sum of irreducible representations,
	where each $(V_\chi,\sigma_\chi)$ occurs with some (possibly zero) multiplicity $m_\chi$
	determined by $(\mathbb{F}^k,\iota)$.
	There is an isomorphism between $\mathbb{F}[t^{\pm1}]$--modules:
	$$H^{\rho\otimes\phi_f}_n\left(M_f;{\mathbb{F}\left[t^{\pm1}\right]}^k\right)
	\cong \bigoplus_\chi H_n\left(M_f;V_\chi\left[t^{\pm1}\right]\right)^{\oplus m_\chi},$$
	coming from a parallel direct-sum decomposition on the chain level.
	By taking the order, we obtain a factorization up to units of $\mathbb{F}[t^{\pm1}]$:
	\begin{equation}\label{factorization_Delta}
	\Delta^{\rho\otimes\phi_f}_{M_f,n}\doteq
	\prod_\chi \left(\Delta^{\gamma^*(\sigma_\chi)\otimes \phi_f}_{M_f,n}\right)^{m_\chi}.
	\end{equation}
	Therefore, it suffices to argue that this factorization,
	up to units of $\mathbb{F}[t^{\pm1}]$, is determined 
	by the aligned equivalence class for $(f,\rho)$.
	
	To this end, we apply Theorem \ref{aligned_mapping_torus_homology}
	for $R=\Integral_p$, 
	so the twisted homology	$H_n(M_f;R[t^{\pm1}]\Gamma)$ 
	is determined by the aligned equivalence class for $(f,\rho)$,
	up to isomorphism of left $R[t^{\pm1}]\Gamma$--modules.
	Observe that there is a canonical isomorphism between left $\mathbb{F}[t^{\pm1}]\Gamma$--modules:
	$$H_n\left(M_f;\mathbb{F}\left[t^{\pm1}\right]\Gamma\right)\cong 
	H_n\left(M_f;R\left[t^{\pm1}\right]\Gamma\right)\otimes_{R}\mathbb{F},$$
	because $\Rational_p$, and hence $\mathbb{F}=\overline{\Rational_p}$,
	are flat over $R=\Integral_p$, 
	(arguing similarly as with (\ref{tensor_flat})).
	It follows that the isomorphism class of
	$H_n(M_f;\mathbb{F}[t^{\pm1}]\Gamma)$,
	as a left $\mathbb{F}[t^{\pm1}]\Gamma$--module,
	is determined by the aligned equivalence class for $(f,\rho)$.
	For any $n$ and $\chi$,
	the left $\mathbb{F}\Gamma$--isotypic component
	of $H_n(M_f;\mathbb{F}[t^{\pm1}]\Gamma)$
	with type (oppositely) dual to $(V_\chi,\sigma_\chi)$ 
	is a finitely generated $\mathbb{F}[t^{\pm1}]$--module of vanishing rank and of order
	\begin{equation}\label{order_Delta}
	\mathrm{ord}_{\mathbb{F}[t^{\pm1}]}\left(H_n\left(M_f;\mathrm{End}_{\mathbb{F}}(V_\chi)\left[t^{\pm1}\right]\right)\right)
	\doteq
	\left(\Delta^{\gamma^*(\sigma_\chi)\otimes\phi_f}_{M_f,n}\right)^{\dim_{\mathbb{F}}(V_\chi)},
	\end{equation}
	up to units of $\mathbb{F}[t^{\pm1}]$, by Lemma \ref{isotypic_decomposition}.
	
	Since the aligned equivalence class for $(f,\rho)$
	determines $\Gamma$ and its irreducible characters over $\mathbb{F}$,
	it also determines the above orders of isotypic components.
	By (\ref{order_Delta}) and the unique factorization property of $\mathbb{F}[t^{\pm1}]$, 
	we see that $\Delta^{\gamma^*(\sigma_\chi)\otimes\phi_f}_{M_f,n}$ 
	are all determined, up to units of $\mathbb{F}[t^{\pm1}]$.
	Observe that the aligned equivalence class for $(f,\rho)$
	determines $(\mathbb{F}^k,\iota)$ and the multiplicities $m_\chi$ as well.
	By (\ref{factorization_Delta}), we see that 
	the $n$--th twisted Alexander polynomial $\Delta^{\rho\otimes\phi_f}_{M_f,n}$ is determined
	by the aligned equivalence class for $(f,\rho)$, for each $n$.
	This completes the proof of Theorem \ref{TRT_dialogue}
	for the special case $\mathbb{F}=\overline{\Rational_p}$.
\end{proof}

To prove Theorem \ref{TRT_dialogue},
we derive the general case from Lemma \ref{TRT_dialogue_reduced}.
This is done by change of scalars through a sequence of field extensions and restrictions.

By representation theory of finite groups over fields of characteristic $0$,
there exists a finite Galois extension $E$ over $\Rational$, depending on $\Gamma$,
and satisfying the following property:
For any field embedding of $E$ into another field $E'$,
a representation of $\Gamma$ over $E$ is irreducible
if and only if the induced representation via the embedding is irreducible over $E'$.
(For example, we can take a finite Galois extension $E/\Rational$
in $\Complex$,
such that every $\Complex$--irreducible representation of $\Gamma$ is 
$\Complex$--isomorphic to a representation over $E$.
The asserted property follows from the fact that
any $\Complex$--representation of $\Gamma$ is 
$\Complex$--isomorphic to some $E$--representation,
which is unique up to $E$--isomorphisms;
see \cite[Chapter 12, Section 12.1]{Serre_book_rep},
and in particular, Proposition 33 and its note therein.)

Take some $E$ as above.
Fix an algebraic closure $\overline{\mathbb{F}}$ over $\mathbb{F}$,
and fix a choice of embeddings $E\to\overline{\mathbb{F}}$
and $E\to\overline{\Rational_p}$ for some rational prime $p$.
We speak of induced representations by change of ground fields
with respect to these embeddings.

Denote by $\rho_{\overline{\mathbb{F}}}$ 
the representation of $\pi_1(M_f)$ over $\overline{\mathbb{F}}$ induced by $\rho$ over $\mathbb{F}$.
Observe that $\rho_{\overline{\mathbb{F}}}$ is conjugate over $\overline{\mathbb{F}}$ 
to a representation induced by some $\rho_E$ over $E$, 
which follows easily from the declared property of $E$.
Denote by $\rho_{\overline{\Rational_p}}$ the representation over $\overline{\Rational_p}$ induced by $\rho_E$.
Through the natural embedding $\mathbb{F}[t^{\pm1}]\to\overline{\mathbb{F}}[t^{\pm1}]$,
we can identify $\Delta^{\rho_{\overline{\mathbb{F}}}\otimes\phi_f}_{M_f,n}$ 
with $\Delta^{\rho\otimes\phi_f}_{M_f,n}$ up to units,
(by directly checking on the twisted homology level).
Similarly, using the fixed embeddings of $E$, we can further identify
 $\Delta^{\rho\otimes\phi_f}_{M_f,n}$ with $\Delta^{\rho'\otimes\phi_f}_{M_f,n}$
up to units, for $\rho'$ being either $\rho_E$ or $\rho_{\overline{\Rational_p}}$.
By Lemma \ref{TRT_dialogue_reduced},
the aligned equivalence class for $(f,\rho)$ determines
$\Delta^{\rho_{\overline{\Rational_p}}\otimes\phi_f}_{M_f,n}$ up to units of
$\overline{\Rational_p}[t^{\pm1}]$.
As it also determines $\Gamma$,
the above identifications can be fixed for all $(f,\rho)$ of the same aligned equivalence class.
Therefore, 
the aligned equivalence class for $(f,\rho)$ determines
$\Delta^{\rho\otimes\phi_f}_{M_f,n}$ up to units of $\mathbb{F}[t^{\pm1}]$.

This completes the proof of Theorem \ref{TRT_dialogue}.

\section{Review on fixed point theory}\label{Sec-fpt_review}
In this section, we review fixed point theory and its interpretation in terms of mapping tori.
Both the terminology and the perspective of this section
follow Jiang's expository book \cite{Jiang_book}.
Example \ref{example_index} contains the case of our primary interest.
We introduce a version of twisted periodic Lefschetz numbers,
which is slightly more general than usually considered (Definition \ref{L_m_chi_def}).

\subsection{Fixed point classes and periodic orbit classes}\label{Subsec-fixed_point_theory}
Fixed point theory studies self-maps of topological spaces by their fixed point classes and indices.
The fixed point set of a general self-map can be topologically very nasty,
so the first novel idea of the theory is to consider those homotopically invariant aspects.
With this idea, all homotopically possible fixed points are grouped into 
what are called fixed point classes.
Each fixed point class is associated with an integer called the index.
Intuitively,
some fixed points may emerge or disappear as the self-map varies continuously.
However, there may also be some essential fixed point classes, having nonzero index,
and those classes always predict actual fixed points.
The theory then extends to periodic orbit classes and their indices, 
easily by iterating the self-map.
Via the mapping torus interpretation,
periodic orbit classes can be recognized as 
those free homotopy loops in the mapping torus
which intersect algebraically with the distinguished fiber positively many times.

Let $X$ be a connected compact topological space that is homeomorphic to 
a cell complex,
and $f\colon X\to X$ a self-map of homotopy equivalence.
Fix a universal covering space $\widetilde{X}_\univ$ of $X$.
Denote by $\pi_1(X)$ the deck transformation group of $\widetilde{X}_\univ$
and $\kappa_\univ\colon \widetilde{X}_\univ\to X$ the covering projection.
%

In an abstract form, 
a \emph{fixed point class} $\mathbf{p}$ of $f$ can be characterized as
any conjugacy class of elevations of $f$ to $\widetilde{X}_\univ$,
which we adopt as our definition,
(compare \cite[Chapter I, Section 1]{Jiang_book}).
In other words, $\mathbf{p}$
is represented by some self-map $\tilde{f}\colon \widetilde{X}_\univ\to\widetilde{X}_\univ$
with the property $\kappa_\univ\circ\tilde{f}=f\circ\kappa_\univ$,
and any other representative of $\mathbf{p}$
can be written as $\tau\circ\tilde{f}\circ\tau^{-1}$ for some $\tau\in\pi_1(X)$.
For each fixed point class $\mathbf{p}$,
denote by $\FPS(f;\mathbf{p})\subset X$ the set of those fixed points $x$ of $f$,
such that some lift $\tilde{x}\in\widetilde{X}_\univ$ of $x$ is fixed by some elevation 
$\tilde{f}$ that represents $\mathbf{p}$.
(Fixed points in $\FPS(f;\mathbf{p})$ are 
heuristically thought of as \emph{belonging to the class} $\mathbf{p}$.)
It is easy to see that the subsets $\FPS(f;\mathbf{p})$ are all closed
and mutually isolated. As $X$ is compact,
there are at most finitely many nonempty $\FPS(f;\mathbf{p})$,
and their union is the fixed point set $\FPS(f)\subset X$ of $f$.
Denote by $\FPC(f)$ the set of fixed point classes of $f$.

For any fixed point class $\mathbf{p}\in\FPC(f)$, 
the \emph{index} of $f$ at $\mathbf{p}$ is a well-defined integer, 
which we denote as
$\mathrm{ind}(f;\mathbf{p})\in\Integral$,
(see \cite[Chapter I, Section 3]{Jiang_book} for an exposition 
with direct references).
A fixed point class $\mathbf{p}$ is said to be \emph{essential} if
$\mathrm{ind}(f;\mathbf{p})\neq0$ holds.
In this case, $\FPS(f;\mathbf{p})$ is necessarily nonempty,
so there are at most finitely many essential fixed point classes.
The precise definition of $\mathrm{ind}(f;\mathbf{p})$ is somewhat technical,
but not particularly necessary for our application;
for surface mapping classes,
we review in Section \ref{Subsec-classification_fpc}
the classification of essential fixed point classes,
and an explicit description of their indices, as obtained by Jiang and Guo \cite{Jiang--Guo}.

For any natural number $m\in\Natural$,
an \emph{$m$--periodic point class} of $f$ is defined to be 
a fixed point class of the $m$--th iteration $f^m$.
We denote by $\PPC_{m}(f)$ the set of $m$--periodic point classes of $f$.
There is a natural action of the $m$--cyclic group on $\PPC_{m}(f)$, 
heuristically as induced by the action of $f$ on $m$--periodic points.
Formally, one may take an elevation $\tilde{f}$ of $f$ to $\widetilde{X}_\univ$.
For elevation $\widetilde{f^m}$ of $f^m$ to $\widetilde{X}_\univ$,
an elevation $\tau_{\tilde{f}}(\widetilde{f^m})$ of $f^m$ is uniquely defined
by the relation $\tilde{f}\circ\widetilde{f^m}=\tau_{\tilde{f}}(\widetilde{f^m})\circ\tilde{f}$.
One may check that the conjugation class of $\tau_{\tilde{f}}(\widetilde{f^m})$ does not depend on 
the choice of $\tilde{f}$, and that the $m$--th iteration $\tau_{\tilde{f}}^m(\widetilde{f^m})$
is conjugate to $\widetilde{f^m}$, so the $m$--cyclic action is induced by 
$\widetilde{f^m}\mapsto\tau_{\tilde{f}}(\widetilde{f^m})$.
This action gives rise to a partition of $\PPC_m(f)$ into a disjoint union of orbits,
each called an \emph{$m$--periodic orbit class}.
We denote by $\POC_m(f)$ the set of $m$--periodic orbit classes of $f$.

For any $m\in\Natural$, 
the \emph{$m$--index} of $f$ can be defined for any $m$--periodic orbit class $\mathbf{O}\in\POC_m(f)$,
and we denote it as
\begin{equation}\label{ind_m}
\mathrm{ind}_m(f;\mathbf{O})\in\Integral.
\end{equation}
To be precise, for any $m$--periodic point class $\mathbf{p}\in\mathbf{O}$,
define the $m$--index of $f$ at $\mathbf{p}$ to be the index of $f^m$ at $\mathbf{p}$,
and denote as $\mathrm{ind}_m(f;\mathbf{p})=\mathrm{ind}(f^m;\mathbf{p})$.
The $m$-index of $\mathbf{O}$
is defined as the total index of $f^m$ over $\mathbf{O}$ regarded as a subset of $S$,
or to paraphrase,
$\mathrm{ind}_m(f;\mathbf{O})=\sum_{\mathbf{p}\in\mathbf{O}}\mathrm{ind}_m(f;\mathbf{p})$.
(See \cite[Sections 1.3 and 1.4]{Jiang_periodic}, and in particular, 
compare the notion of $n$--orbit classes and their indices therein.)
For all $\mathbf{p}\in\mathbf{O}$, 
it turns out that $\mathrm{ind}_m(f;\mathbf{p})$ are all equal;
so $\mathrm{ind}_m(f;\mathbf{O})$ is 
that common value multiplied by some positive integer that divides $m$.
For any $m\in\Natural$, 
there are at most finitely many essential $m$--periodic point classes,
namely, those of nonzero $m$--index, 
so the same holds with essential $m$--periodic orbit classes.

For any self-map $f'\colon X\to X$ which is homotopic to $f$,
it is known that there are natural bijections between
$\FPC(f')$ and $\FPC(f)$,
and moreover, 
the index of $f$ and $f'$ at corresponding fixed point classes are equal. 
Such natural homotopy invariance also holds with $m$--periodic point classes,
$m$--periodic orbit classes, and their $m$--indices.


\subsection{Twisted periodic Lefschetz numbers}\label{Subsec-twisted_Lefschetz}
Denote by $M_f=(X\times[0,1])/\sim$ the mapping torus of $f$,
where $\sim$ stands for the equivalence relation $(x,1)\sim(f(x),0)$. 
We associate to any $m$--periodic orbit class
a free-homotopy loop of the mapping torus, as follows.

For any $m$--periodic orbit class $\mathbf{O}\in\POC_m(f)$,
and for any $m$--periodic point class $\mathbf{p}\in\mathbf{O}$,
take an elevation $\widetilde{f^m}\colon \widetilde{X}_\univ\to\widetilde{X}_\univ$
of $f^m$ that represents $\mathbf{p}$.
Then there is a map 
$\widetilde{X}_\univ\times[0,1]\to X\times[0,1]$,
defined piecewise by $(\tilde{x},r)\mapsto(f^j(\kappa_\univ(\tilde{x})),mr-j)$
for $r\in[j/m,(j+1)/m)$ with $j=0,1,\cdots,m-1$, and for $r=1$ with $j=m$.
It induces a (continuous) map $M_{\widetilde{f^m}}\to M_f$.
When $f$ is a homeomorphism, this agrees with the map induced by 
$\widetilde{X}_\univ\times\Real\to X\times\Real\colon
(\tilde{x},r)\mapsto(\kappa_\univ(\tilde{x}),mr)$.
Since the natural projection $M_{\widetilde{f^m}}\to\Real/\Integral$ induced by $(\tilde{x},r)\mapsto r$
is a homotopy equivalence, its homotopy inverse followed by the above map
gives rise to a free-homotopy class of a directed loop $\ell_m(f;\mathbf{p})\in \left[\Real/\Integral,M_f\right]$,
where $[\Real/\Integral,M_f]$ stands for 
the set of the homotopy classes of (continuous) maps $\Real/\Integral\to M_f$.
In fact, the homotopy loop $\ell_m(f;\mathbf{p})$ depends only on $f$, $m$, and $\mathbf{O}$,
so we also denote it as $\ell_m(f;\mathbf{O})$.

Observe that we may canonically identify the set $[\Real/\Integral,M_f]$ with 
the set of conjugation orbits of the mapping-torus group $\pi_1(M_f)$,
which we denote as $\conjorb(\pi_1(M_f))$.
Therefore, for any $m$--periodic orbit class $\mathbf{O}\in\POC_m(f)$,
we have introduced a conjugation orbit of the mapping-torus group:
\begin{equation}\label{gamma_m}
\ell_m(f;\mathbf{O})\in \conjorb\left(\pi_1(M_f)\right).
\end{equation}
We refer to $\ell_m(f;\mathbf{O})$ as an \emph{$m$--periodic trajectory class} of $M_f$.
One may think of it intuitively as a loop up to free homotopy.
We also say that $\ell_m(f;\mathbf{O})$ is \emph{essential}
if $\mathbf{O}$ is essential.

\begin{definition}\label{L_m_chi_def}
Let $R$ be a commutative ring. 
Suppose that $\chi$ be an $R$--valued, conjugation-invariant function on $\pi_1(M_f)$,
or equivalently, an $R$--valued function on $\conjorb(\pi_1(M_f))$.
For any $m\in\Natural$,
we introduce the \emph{$\chi$--twisted $m$--th Lefschetz number} $L_m(f;\chi)\in R$
by the expression:
\begin{equation}\label{L_m_chi}
L_m(f;\chi)=\sum_{\mathbf{O}\in\POC_{m}(f)} \chi\left(\ell_m(f;\mathbf{O})\right)\cdot\mathrm{ind}_m(f;\mathbf{O}),
\end{equation}
with the notations (\ref{ind_m}) and (\ref{gamma_m}).
\end{definition}

\begin{example}\label{example_index}\
	Let $k$ be a natural number and $\mathbb{F}$ be a (commutative) field of characteristic $0$.
	In any monologue setting $(S,f)$, suppose that 
	$\rho\colon \pi_1(M_f) \to \mathrm{GL}(k,\mathbb{F})$
	is a representation of the mapping-torus group.
	Denote by $\chi_\rho$ the character of $\rho$,
	namely,	the conjugation-invariant function 
	$\pi_1(M_f)\to \mathbb{F}\colon g\mapsto \mathrm{tr}_{\mathbb{F}}(\rho(g))$.
	For any $m\in\Natural$, 
	the  $m$--th twisted Lefschetz number for $(f,\rho)$ is denoted as 
	$$L_m(f;\chi_\rho)\in\mathbb{F}.$$
	
	For $k=1$ and $\rho$ the trivial representation on $\mathbb{F}=\Rational$,
	we drop the notation $\chi_\rho$.
	The classical Lefschetz trace formula implies
	$$L_m(f)=\sum_{n\in\Integral}(-1)^n\mathrm{tr}_\Rational\left(f^m_*\colon H_n(S;\Rational)\to H_n(S;\Rational)\right).$$
	General twisted Lefschetz numbers can be computed using Nielsen--Thurston normal forms
	and (\ref{L_m_chi}), see Section \ref{Subsec-classification_fpc}.
\end{example}

\section{Twisted periodic Lefschetz numbers of mapping classes}\label{Sec-L_m_mapping_class}
In this section, we show that certain twisted periodic Lefschetz numbers
are determined by the procongruent conjugacy class of the mapping class.

\begin{theorem}\label{L_m_dialogue}
	Let $k$ be a natural number and $\mathbb{F}$ be a (commutative) field of characteristic $0$.
	
	In any dialogue setting $(S,f_A,f_B)$, suppose that 
	$\rho_A\colon \pi_1(M_A) \to \mathrm{GL}(k,\mathbb{F})$
	and $\rho_B\colon \pi_1(M_B) \to \mathrm{GL}(k,\mathbb{F})$
	are representations of the mapping-torus groups
	which are finite and aligned equivalent.
	Then the following equality holds for all $m\in\Natural$:
	$$L_m(f_A;\chi_A)=L_m(f_B;\chi_B),$$
	between the twisted $m$--th Lefschetz numbers for $(f_A,\rho_A)$ and $(f_B,\rho_B)$.
	(See Example \ref{example_index}, Definition \ref{aligned_equivalent_representation}, and Convention \ref{settings}.)
\end{theorem}

We prove Theorem \ref{L_m_dialogue} in the rest of this section.
The proof relies on Theorem \ref{TRT_dialogue} and uses certain well-known identities
between twisted Reidemeister torsions and twisted Lefschetz zeta functions.
We supply enough details to adapt book-keeping results to our context.

Let $k$ be a natural number and $\mathbb{F}$ be a (commutative) field of characteristic $0$.
In any monologue setting $(S,f)$,
suppose that $\rho\colon \pi_1(M_f)\to\mathrm{GL}(k,\mathbb{F})$
is a representation of the mapping-torus group.
Using the twisted Lefschetz numbers for $(f,\rho)$,
we form the twisted Lefschetz zeta function:
\begin{equation}\label{L_zeta}
	\zeta_L(f;\rho)=\exp\left(\sum_{m\in\Natural} \frac{L_m(f;\chi_\rho)}m\cdot t^m\right),
\end{equation}
where $\exp(z)$ stands for the formal power series $\sum_{m=0}^\infty \frac{z^m}{m!}$,
(see Example \ref{example_index}).
We treat $\zeta_L(f;\rho)$ as a formal power series, living in $\mathbb{F}[[t]]$.

\begin{lemma}\label{tau_zeta}
	The twisted Lefschetz zeta function $\zeta_L(f;\rho)$ 
	is the formal Taylor expansion at $t=0$ of a unique rational function in $t$ over $\mathbb{F}$,
	with the constant term $1$.
	Moreover, 
	it equals the twisted Reidemeister torsion for $(M_f,\phi_f,\rho)$,
	in $\mathbb{F}(t)^\times$ up to monomial factors with nonzero coefficients:
	$$\tau^{\rho\otimes\phi_f}\left(M_f;\mathbb{F}(t)^k\right)\doteq\zeta_{L}\left(f;\rho\right).$$	
\end{lemma}

\begin{proof}
Because both sides of the asserted equality in Theorem \ref{tau_zeta}
are naturally invariant under homotopy of $f$ and pull-back of $\rho$, 
we may as well prove the equality with a homotopically modified $f$,
(see Example \ref{example_TRT} and Proposition \ref{homotopy_equivalence_invariance}, 
regarding the homotopy invariance of the right-hand side;
see our review of periodic orbit classes and their indices
in Section \ref{Subsec-fixed_point_theory}, 
regarding the homotopy invariance of the left-hand side).
We choose a cell decomposition of $S$, and 
homotope $f$ into a (possibly non-homeomorphic) cellular map
which is piecewise linear on each cell.

The mapping torus $M_f=(S\times[0,1])/\sim$ with $(x,1)\sim(f(x),0)$ 
is naturally homotopy equivalent to the original one,
but also enriched with an induced cell decomposition.
Each (open) cell of $M_f$ 
is homeomorphically projected by 
either a product of some cell of $S$ with $\{0\}$, or with $(0,1)$.
We refer to those with $\{0\}$ as \emph{fiber cells} and those with $(0,1)$ as \emph{flow cells}.
Take a universal cover $\widetilde{M_f}\to M_f$ with the lifted cell decomposition.
Denote by $\pi_1(M_f)$ the deck transformation group acting on $\widetilde{M_f}$.
The originally given representation becomes 
a representation $\rho\colon\pi_1(M_f)\to\mathrm{GL}(k,\mathbb{F})$
for the cellular mapping torus.

To calculate the twisted Reidemeister torsion $\tau^{\rho\otimes\phi_f}(M_f;\mathbb{F}(t)^k)$,
we assign suitable bases for the twisted cellular chain modules,
and work out the relevant part of the boundary operators by their matrix blocks.
To this end, choose a lift to $\widetilde{S}_\univ$ together with an orientation for each cell of $S$.
For each dimension $n$, choose an order 
and enumerate the oriented lifted $n$--cells as $c_n^1,\cdots,c_n^{r_n}$,
(null unless $n=0,1,2$).
Choose a lift $\widetilde{S}_\univ\to \widetilde{M_f}$ 
of the obvious composed map $\widetilde{S}_\univ\to S\to S\times\{0\}\to M_f$.
Identify each $c^i_n$ as a fiber cell of $\widetilde{M_f}$ via the lift.
Denote by $d^i_{n+1}=c^i_n\times(0,1)$ the flow cell of $\widetilde{M_f}$ next to $c^i_n$,
oriented as a product.
Denote by $\varepsilon^1,\cdots,\varepsilon^k$ 
the standard basis of $\mathbb{F}(t)^k$ over $\mathbb{F}(t)$.

The $n$--th twisted chain module 
$\mathbb{F}(t)^k\otimes_{\Integral\pi_1(M_f)}C_n(\widetilde{M_f})$,
for each $n$,
is spanned linearly independently
over $\mathbb{F}(t)$ by all $\varepsilon^l\otimes c^i_n$ and $\varepsilon^l\otimes d^j_n$.
We order them by listing the former ones first, lexicographically in $(i,l)$,
and the latter ones next, likewise in $(j,l)$.
Fully written down, 
our ordered basis for the $n$--th twisted chain module 
is
$\varepsilon^1\otimes c^1_n,\cdots,\varepsilon^k\otimes c^1_n,
\cdots,\varepsilon^1\otimes c^{r_n}_n,\cdots,\varepsilon^k\otimes c^{r_n}_n,
\varepsilon^1\otimes d^1_n,\cdots,\varepsilon^k\otimes d^1_n,
\cdots,\varepsilon^1\otimes d^{r_n}_n,\cdots,\varepsilon^k\otimes d^{r_{n-1}}_n$.
On the basis,
twisted cellular $n$--chains are represented as column vectors,
over $\mathbb{F}(t)$ of dimension $kr_n+kr_{n-1}$.

The $n$--th twisted boundary operator 
$\mathbf{1}\otimes\partial_n$,
for each $n$,
is represented
as a matrix
over $\mathbb{F}(t)$ of size $(kr_{n-1}+kr_{n-2})\times (kr_n+kr_{n-1})$.
The boundary operator $\partial_n\colon C_n(\widetilde{M_f})\to C_{n-1}(\widetilde{M_f})$
sends each of the flow cells $d^i_n$ to a linear combination over $\Integral\pi_1(M_f)$:
\begin{equation}\label{flow_cell_boundary}
\partial_nd^i_n=(-1)^{n-1}\left(\sum_{j=1}^{r_{n-1}} F^{ij}_{n-1}\cdot c^j_{n-1}-c^i_{n-1}\right)+
(\textrm{sum of flow cell terms}),
\end{equation}
where $F^{ij}_{n-1}$ are $\Integral$--linear combinations 
of group elements $g\in\pi_1(M_f)$ with $\phi_f(g)=1$.
Note that each $F^{ij}_{n-1}$ is unique subject to (\ref{flow_cell_boundary}).
Extend $\rho\colon \pi_1(M_f)\to \mathrm{GL}(k,\mathbb{F})$ linearly
as an algebra homomorphism 
$\rho\colon \Integral\pi_1(M_f)\to \mathrm{Mat}_{k\times k}(\mathbb{F})$.
Denote by 
$$\rho_*(F_{n-1})\in\mathrm{Mat}_{kr_{n-1}\times kr_{n-1}}(\mathbb{F})$$
the $r_{n-1}$--square block matrix 
whose $(i,j)$--block is the $k$--square matrix $\rho(F^{ij}_{n-1})$.
For $n=1,2,3$, the matrices take the following block forms:
$$
\xymatrix{
\mathbf{1}\otimes\partial_3={\left[\begin{array}{cc} * & t\rho_*(F_2)-\mathbf{1}\end{array}\right]},
&
\mathbf{1}\otimes\partial_2={\left[\begin{array}{cc} * & \mathbf{1}-t\rho_*(F_1) \\ * & *\end{array}\right]},
&
\mathbf{1}\otimes\partial_1={\left[\begin{array}{c} t\rho_*(F_0)-\mathbf{1} \\ * \end{array}\right]},
}
$$
which are actually over $\mathbb{F}[t]$.
For any other $n$, the matrices for $\mathbf{1}\otimes\partial_n$ are null of size.

The twisted Reidemeister torsion for $(M_f,\rho,\phi_f)$
can be expressed conveniently using the above matrices, 
according to a well-known formula \cite[Theorem 2.2]{Turaev_book_torsion}:
\begin{equation}\label{tau_F}
\tau^{\rho\otimes\phi_f}\left(M_f;\mathbb{F}(t)^k\right)
\doteq
\frac
{\mathrm{det}_{\mathbb{F}[t]}\left(\mathbf{1}-t\rho_*(F_1)\right)}
{\mathrm{det}_{\mathbb{F}[t]}\left(\mathbf{1}-t\rho_*(F_0)\right)\cdot\mathrm{det}_{\mathbb{F}[t]}\left(\mathbf{1}-t\rho_*(F_2)\right)},
\end{equation}
both sides living in $\mathbb{F}(t)^\times$, and being equal up to monomial factors with nonzero coefficients.

The right-hand side of (\ref{tau_F}) can be recognized as the twisted Lefschetz zeta function for $(f,\rho)$:
\begin{equation}\label{F_zeta}
\zeta_L(f;\rho)=
\frac
{\mathrm{det}_{\mathbb{F}[t]}\left(\mathbf{1}-t\rho_*(F_1)\right)}
{\mathrm{det}_{\mathbb{F}[t]}\left(\mathbf{1}-t\rho_*(F_0)\right)\cdot\mathrm{det}_{\mathbb{F}[t]}\left(\mathbf{1}-t\rho_*(F_2)\right)},
\end{equation}
both sides living in $\mathbb{F}[[t]]$.
In fact, this is exactly the kind of twisted Lefschetz zeta functions 
considered by \cite[Theorem 1.2]{Jiang_periodic}.

Joining up (\ref{tau_F}) and (\ref{F_zeta}), we obtain the asserted formula of Lemma \ref{tau_zeta}.
\end{proof}

To prove Theorem \ref{L_m_dialogue},
we observe that under the assumptions,
the aligned equivalence between $(f_A,\rho_A)$ and $(f_B,\rho_B)$ implies
an equality between the twisted Reidemeister torsions:
$$\tau^{\rho_A\otimes\phi_A}\left(M_A;\mathbb{F}(t)^k\right)\doteq \tau^{\rho_B\otimes\phi_B}\left(M_B;\mathbb{F}(t)^k\right),$$
in $\mathbb{F}(t)^\times$ up to monomial factors with nonzero coefficients, (Theorem \ref{TRT_dialogue}).
Note that $\mathbb{F}$ has characteristic $0$.
Therefore, the only representative for $\tau^{\rho_A\otimes\phi_A}(M_A;\mathbb{F}(t)^k)$ which takes value $1$ at $t=0$
must be $\zeta_L(f_A;\rho_A)$, by Lemma \ref{tau_zeta}.
The same holds with $\zeta_L(f_B;\rho_B)$.
It follows that the equality
$$\zeta_L(f_A;\rho_A)=\zeta_L(f_B;\rho_B)$$
holds in $\mathbb{F}[[t]]$.
It follows that the equalities
$$L_m(f_A;\chi_A)=L_m(f_B;\chi_B)$$
hold for all $m\in\Natural$, by comparing the formal Taylor expansions of both sides at $t=0$, term by term.

This completes the proof of Theorem \ref{L_m_dialogue}.

\section{Indexed orbit numbers of mapping classes}\label{Sec-indexed_orbit_numbers}
In this section, we show that indexed orbit numbers of mapping classes
are procongruent conjugacy invariants.
In any monologue setting $(S,f)$, we count the $m$--periodic orbit classes of $f$ 
separately for each $m$--index $i$.
For all $m\in\Natural$ and $i\in\Integral$,
the amounts are denoted as
\begin{equation}\label{nu_m_i}
\nu_m(f;i)=\#\left\{\mathbf{O}\in\POC_m(f)\colon \mathrm{ind}_m(f;\mathbf{O})=i\right\},
\end{equation}
valued in $\Natural\cup\{0,\infty\}$.
Keeping $m$ fixed, $\nu_m(f;i)$ is nonzero for all but finitely many $i$,
and finite for any nonzero $i$.
The classical $m$--orbit Nielsen numbers $N_m(f)$,
which count the essential $m$--periodic orbit classes,
can be expressed as
\begin{equation}\label{N_m_def}
N_m(f)=\sum_{i\in\Integral\setminus\{0\}} \nu_m(f;i),
\end{equation}
for all $m\in\Natural$.

\begin{theorem}\label{N_m_dialogue}
In any dialogue setting $(S,f_A,f_B)$, 
suppose that the mapping classes of $f_A$ and $f_B$ are procongruently conjugate.
Then the following equalities hold for all $m\in\Natural$ and $i\in\Integral$:
$$\nu_m(f_A;i)=\nu_m(f_B;i),$$
between the indexed orbit numbers of mapping classes.
(See (\ref{nu_m_i}) and Convention \ref{settings}.)
\end{theorem}

The rest of this section is devoted to the proof of Theorem \ref{N_m_dialogue}.
The proof relies on Theorem \ref{L_m_dialogue} and 
the conjugacy separability of $3$--manifold groups.

In any monologue setting $(S,f)$,
suppose that $\gamma\colon\pi_1(M_f)\to \Gamma$ is a quotient of the mapping-torus group
onto any finite group $\Gamma$.
Denote by $\conjorb(\Gamma)$ the \emph{orbit space} of $\Gamma$,
namely, the quotient set of $\Gamma$ modulo conjugation. 
The characteristic function of every conjugation orbit $\mathbf{c}\in\conjorb(\Gamma)$ 
is a conjugation-invariant function 
$\chi_{\mathbf{c}}\colon \Gamma\to \Integral$,
namely,
\begin{equation}\label{chi_c}
\chi_{\mathbf{c}}(g)=\begin{cases} 1 & g\in\mathbf{c}\\ 0 & \textrm{otherwise}\end{cases}
\end{equation}
for all $g\in\Gamma$.

\begin{lemma}\label{L_chi_c}
For any $\mathbf{c}\in\conjorb(\Gamma)$,
the following equalities in $\Complex$ hold for all $m\in\Natural$:
$$L_m(f;\gamma^*\chi_{\mathbf{c}})=
\frac{\#\mathbf{c}}{\#\Gamma}\cdot\sum_{\chi_\rho\in\Gamma^\vee}\overline{\chi_\rho(\mathbf{c})}\cdot L_m(f;\gamma^*\chi_\rho),$$
where $\Gamma^\vee$ stands for the finite set of irreducible complex characters of $\Gamma$.
\end{lemma}

\begin{proof}
	By representation theory of finite groups over $\Complex$, 
	the set of irreducible characters 
	$\Gamma^\vee$ form an (unordered) orthonormal basis of the conjugation-invariant complex-valued functions on $\Gamma$,
	with respect to the inner product 
	$$(\xi,\eta)=\frac{1}{\#\Gamma}\cdot\sum_{g\in\Gamma}\xi(g)\cdot\overline{\eta(g)}
	=\frac{1}{\#\Gamma}\cdot\sum_{\mathbf{g}\in\conjorb(\Gamma)}\#\mathbf{g}\cdot \xi(\mathbf{g})\cdot\overline{\eta(\mathbf{c})}.$$
	It follows that $\chi_{\mathbf{c}}$ equals 
	$\frac{\#\mathbf{c}}{\#\Gamma}$ times the sum of $\overline{\chi_\rho(\mathbf{c})}\cdot \chi_\rho$
	over all $\chi_\rho\in\Gamma^\vee$.
	Then the asserted formula follows as $L_m(f;\gamma^*\chi)$ is linear in $\chi$ over $\Complex$,
	see (\ref{L_m_chi}).
\end{proof}

\begin{lemma}\label{N_m_inequality}
	For any finite quotient $\gamma\colon\pi_1(M_f)\to\Gamma$,
	the following estimate hold for all $m\in\Natural$:
	$$N_m(f)\geq \#\{\mathbf{c}\in\conjorb(\Gamma) \colon L_m(f;\gamma^*\chi_{\mathbf{c}})\neq0\}.$$
	Moreover,
	assume that the equality is achieved for some given $\gamma$ and $m$, 
	then the following equalities hold for all $i\in\Integral$:
	$$\nu_m(f,i)=\#\{\mathbf{c}\in\conjorb(\Gamma) \colon L_m(f;\gamma^*\chi_{\mathbf{c}})=i\}.$$	
\end{lemma}

\begin{proof}
	Let $\gamma\colon\pi_1(M_f)\to\Gamma$ be a finite quotient.
	For any conjugation orbit $\mathbf{c}\in\conjorb(\Gamma)$,
	we obtain 
	$$L_m(f;\gamma^*\chi_{\mathbf{c}})=\sum_{\mathbf{O}\in\POC_m(f)}\chi_{\mathbf{c}}(\gamma(\ell_m(f;\mathbf{O})))\cdot\mathrm{ind}_m(f;\mathbf{O}),$$
	by (\ref{L_m_chi}).
	We say that an $m$--periodic orbit class $\mathbf{O}\in\POC_m(f)$ 
	\emph{occupies} $\mathbf{c}$
	if the conjugation orbit $\ell_m(f;\mathbf{O})\in\conjorb(\pi_1(M_f))$ is projected onto $\mathbf{c}$ under $\gamma$.
	Then there are two simple observations:
	If there are no essential $m$--periodic orbit classes $\mathbf{O}$ occupying $\mathbf{c}$,
	we observe $L_m(f;\gamma^*\chi_{\mathbf{c}})=0$.
	If there is exactly one essential $m$--periodic orbit class $\mathbf{O}$ occupying $\mathbf{c}$,
	we observe $L_m(f;\gamma^*\chi_{\mathbf{c}})=\mathrm{ind}_m(f;\mathbf{O})\neq0$.
	
	The first observation implies 
	that the number of $\mathbf{c}\in\conjorb(\Gamma)$ with $L_m(f;\gamma^*\chi_{\mathbf{c}})\neq0$
	is at most the number of essential $m$--periodic orbit classes, namely,
	the $m$--orbit Nielsen number $N_m(f)$.
	This is the asserted estimate of Lemma \ref{N_m_inequality}.
	
	The second observation implies a useful characterization:
	For any given $m$ and $\gamma$,	the asserted estimate achieves an equality 
	if and only if every conjugation orbit of $\Gamma$ is occupied by 
	at most one essential $m$--periodic orbit class of $f$.
	In fact, there would be strictly fewer occupied conjugation orbits than $N_m(f)$
	if some occupied conjugation orbit was shared by at least two essential $m$--periodic orbit classes;
	on the other hand, 
	when every essential $m$--periodic orbit classes occupies a distinct conjugation orbit,
	there are $N_m(f)$ such conjugation orbits, 
	which are all counted for the lower bound.
	
	The asserted equalities about $\nu_m(f;i)$ follow immediately from the above characterization.	
\end{proof}

\begin{lemma}\label{N_m_detect}
	Given any $m\in\Natural$, there exists some finite quotient $\gamma\colon\pi_1(M_f)\to \Gamma$
	which satisfies the equality:
	$$N_m(f)=\#\left\{\mathbf{c}\in\conjorb(\Gamma) \colon L_m(f;\gamma^*\chi_{\mathbf{c}})\neq0\right\}.$$
\end{lemma}

\begin{proof}
	We make use of the characterization provided in the proof of Lemma \ref{N_m_inequality}.
	Then the asserted equality can be achieved by some finite quotient $\gamma$, 
	thanks to the conjugacy separability of the mapping-torus group.
	
	To elaborate, recall that a group $G$ is said to be \emph{conjugacy separable}
	if the following statement holds true:
	For any finitely many mutually distinct conjugation orbits 
	$\mathbf{g}_1,\cdots,\mathbf{g}_r$ of $G$,
	there exists a finite quotient $G\to\Gamma$, 
	and $\mathbf{g}_i$ are projected onto mutually distinct conjugation orbits of $\Gamma$.
	For general compact $3$--manifolds, conjugacy separability of their fundamental groups
	has been proved by Hamilton--Wilton--Zalesskii \cite{HWZ_conjugacy_separability},
	based on deep works on virtual specialization of hyperbolic $3$--manifolds,
	due to Agol \cite{Agol_VHC} and Wise \cite{Wise_book,Wise_notes}.
	For graph manifolds, the result was proved earlier by Wilton--Zalesskii \cite[Theorem D]{WZ_graph_manifolds},
	using techniques of profinite group actions on profinite trees.
	For hyperbolic and mixed $3$--manifolds,
	conjugacy separability also follows from virtual specialization of their fundamental groups,
	due to Agol \cite{Agol_VHC}, Wise \cite{Wise_book,Wise_notes}, Przytycki--Wise \cite{Przytycki--Wise_mixed},
	combined with the hereditary conjugacy separability of right-angled Artin groups,
	due to Minasyan \cite{Minasyan_RAAG}.	
	
	In a monologue setting $(S,f)$, it follows that the mapping-torus group $\pi_1(M_f)$ 
	is conjugacy separable. 
	Given any $m\in\Natural$, 
	there are only finitely many essential $m$--periodic orbit classes $\mathbf{O}\in\POC_m(f)$,
	so there exists some finite quotient $\gamma\colon\pi_1(M_f)\to\Gamma$
	under which they occupy distinct conjugation orbits of $\Gamma$.
	For any such $\gamma$, 
	the asserted equality is achieved,
	by the characterization explained in the proof of Lemma \ref{N_m_inequality}.
\end{proof}

To prove Theorem \ref{N_m_dialogue}, we argue as follows.
In any dialogue setting $(S,f_A,f_B)$, suppose that
the mapping classes of $f_A$ and $f_B$ are procongruently conjugate.
Take an aligned isomorphism $\Psi\colon\widehat{\pi_1(M_A)}\to\widehat{\pi_1(M_B)}$,
as guaranteed by Proposition \ref{characterization_mapping_tori}.
Then $\Psi$ induces bijective correspondences 
between finite quotients of $\pi_1(M_A)$ and $\pi_1(M_B)$,
and between finite representations of them over $\Complex$.
It follows from Theorem \ref{L_m_dialogue} and Lemma \ref{L_chi_c}
that $L_m(f_A;\gamma_A^*\chi_\mathbf{c})=L_m(f_B;\gamma_B^*\chi_\mathbf{c})$
holds for any aligned finite quotient $\gamma_A\colon\pi_1(M_A)\to\Gamma$ 
and $\gamma_B\colon\pi_1(M_B)\to\Gamma$, and for any conjugation orbit $\mathbf{c}\in\conjorb(\Gamma)$.
By Lemma \ref{N_m_detect},
we obtain $N_m(f_A)=N_m(f_B)$ for all $m\in\Natural$.
Then by Lemma \ref{N_m_inequality},
we obtain $\nu_m(f_A;i)=\nu_m(f_B;i)$ for all $m\in\Natural$ and $i\in\Integral$.

This completes the proof of Theorem \ref{N_m_dialogue}.

\section{Review on the Nielsen--Thurston classification}\label{Sec-review_NT}
		The Nielsen--Thurston classification makes it possible to
		understand surface mapping classes through their representatives of normal forms.
		In this section, we review the theory and describe the models that we adopt (Section \ref{Subsec-NT_normal_form}).
		We review the classification of essential fixed point classes and their indices
		following Jiang and Guo \cite{Jiang--Guo} (Section \ref{Subsec-classification_fpc}).
		We introduce dilatation and deviation as complexity measures of mapping classes,
		and provide a finiteness criterion in terms of them (Section \ref{Subsec-NT_complexity}).
		
		\subsection{Nielsen--Thurston normal form}\label{Subsec-NT_normal_form}
		According to the Nielsen--Thurston theory,
		any mapping class of a connected orientable compact surface
		admits a representative 
		which is either reducible, or periodic, or pseudo-Anosov.
		The reduction procedure leads to 
		a classification of mapping classes,
		by representatives of certain normal forms
		with respect to a certain canonical decomposition of the surface.
		From a geometric point of view,
		it seems more suitable to assume the surface 
		to have negative Euler characteristic 
		as we	employ the Nielsen--Thurston decomposition.
		This excludes disks, spheres, annuli, and tori,
		whose mapping classes are already classified by more elementary means.
		After all, there are obvious normal forms for those cases,
		as either spherical isometric transformations or affine linear transformations.
		
		We adopt the following terminology 
		to describe our normal forms.
		Let $R$ be a possibly disconnected orientable compact surface
		with possibly empty boundary, and $h\colon R\to R$ be an orientation-preserving self-homeomorphism.
		
		We say that $(R,h)$ is of \emph{hyperbolic periodic form} 
		if $R$ has only components of negative Euler characteristic,
		and if some positive iteration of $h$ equals identity.
		
		We say that $(R,h)$ is of \emph{pseudo-Anosov form} 
		if $R$ has only components of negative Euler characteristic,
		and admits an ordered pair of measured foliations
		which are transverse in the interior of $R$,
		and if $h$ preserves the foliations
		and rescales the measures
		by $\lambda$ and $\lambda^{-1}$ respectively,
		for some ratio $\lambda\in(1,+\infty)$ 
		which stays constant on each component of $R$.
		The ratio $\lambda$ is called the \emph{stretching factor} of $h$.
		We allow the foliations to have prong singularities,
		in the interior with prong number at least $3$,
		or on the boundary with prong number $3$.
		We require the singular points of the foliations
		to coincide in the interior of $R$, 
		and to sit alternately on each component of $\partial R$
		in cyclical order.
		We also require $h$ to be periodic restricted to $\partial R$.
		If there exists some vector field on $R$,
		such that the vectors are (nonzero and) 
		transverse to the leaves of a foliation everywhere, 
		except at the singularities or on the boundary,
		we say that the foliation is \emph{transversely orientable}.
		Note that the pair of invariant foliations under $h$
		must have identical transverse orientability.
		
		We say that $(R,h)$ is of \emph{fractional Dehn twist form} 
		if $R$ has only annuli components,
		and admits a parametrization by $[-1,1]\times\Real/\Integral$	for each component,
		and if $h$ is an affine isomorphism which is periodic restricted to $\partial R$.
		In this case, 
		some positive iteration of $h$ is an integral Dehn twist
		restricted to each component of $R$.
		With the parameters $(s,[u])\in [-1,1]\times \Real/\Integral$ of that component,
		this means $h^m(s,[u])=(s,[u+ks])$, for some $m\in\Natural$ and $k\in\Integral$.
		
		\begin{definition}\label{NT_normal_form}		
		Let $S$ be a connected orientable compact surface
		with possibly empty boundary and of negative Euler characteristic.
		An orientation-preserving self-homeomorphism $f\colon S\to S$
		is said to be of \emph{Nielsen--Thurston normal form},
		with respect to a \emph{Nielsen--Thurston decomposition}
		of $S$ into
		a \emph{periodic part} $S_\periodic$
		and a \emph{pseudo-Anosov part} $S_\pA$ 
		and a \emph{reduction part} $S_\reduction$,
		if the following conditions are all satisfied:
		\begin{itemize}
		\item
		The terms	$S_{\periodic},S_{\pA},S_{\reduction}$ 
		are all embedded compact subsurfaces of $S$ with mutually disjoint interior.
		Every component of $\partial S_{\periodic}$ or $\partial S_{\pA}$
		is a component of $\partial S_{\reduction}$ or $\partial S$,
		and vice versa.		
		\item
		The subsurfaces $S_{\periodic},S_{\pA},S_{\reduction}$ are all invariant under $f$.
		The restriction of $f$ to $S_{\periodic}$ is of hyperbolic periodic form;
		the restriction of $f$ to $S_{\pA}$ is of pseudo-Anosov form;
		the restriction of $f$ to $S_{\reduction}$ is of fractional Dehn twist form.
		\item
		If the boundary $\partial U$ of a component $U$ of $S_{\reduction}$
		is contained in $\partial S_{\periodic}$,
		no nonzero iterations of $f$ fix $U$.
		\end{itemize}
		\end{definition}
		
		According to the Nielsen--Thurston theory, 
		any mapping class	of an orientable connected compact surface of negative Euler characteristic
		admits a representative of Nielsen--Thurston normal form.
		In fact, it is unique up to conjugation by homeomorphisms that preserves the mapping class.
		
	\subsection{Essential fixed point classes and indices}\label{Subsec-classification_fpc}
		For Nielsen--Thurston normal forms, the essential fixed point classes and their indices
		have been completely classified by Jiang--Guo \cite[Section 3.4]{Jiang--Guo}.
		We paraphrase their result below for the reader's reference.
		
		Let $S$ be a connected orientable compact surface of negative Euler characteristic.
		Suppose that $f\colon S\to S$ is an orientation-preserving self-homeomorphism of Nielsen--Thurston normal form,
		with respect to a Nielsen--Thurston decomposition $(S_\periodic,S_\pA,S_\reduction)$ of $S$.
		For any essential fixed point class $\mathbf{p}\in\FPC(f)$,
		the subset $\FPS(f;\mathbf{p})$
		turns out to be a (connected) component of the fixed point set $\FPS(f)$.
		All the possible forms are listed	in Example \ref{example_fpc_classification}.
		(The case names are added as quick description.)
		On the other hand, the inessential components of $\FPS(f)$ 
		are all fixed circles in the interior of $S_\reduction$.
		In fact, these are the orientation-preserving cases of \cite[Lemma 3.6]{Jiang--Guo},
		namely, (1a),(1b),(1c),(2a),(2c),(2d), and (4).
						
		\begin{example}\label{example_fpc_classification}
		For any essential fixed point classes $\mathbf{p}\in\FPC(f)$, 
		only the following cases may occur as $\FPS(f;\mathbf{p})$.
		\begin{enumerate}
		\item \emph{Elliptic or parabolic points}. 
		The subset $\FPS(f;\mathbf{p})$ is a fixed point $x$ in the interior of a non-fixed component of $S_\periodic$
		or $S_\reduction$. The fixed point index $\mathrm{ind}(f;\mathbf{p})$ equals $1$.
		\item \emph{Prong singularities or saddle points}.
		The subset $\FPS(f;\mathbf{p})$ is a fixed point $x$ in the interior of a component of $S_\pA$.
		When $x$ is a prong singularity of the invariant (stable of unstable) foliation,
		the fixed point index $\mathrm{ind}(f;\mathbf{p})$ equals $1$ 
		if $f$ permutes the prongs at $x$ cyclically and nontrivially,
		or it equals $1-k$ where $k\geq3$ stands for the number of prongs at $x$.
		When $x$ is a regular point, it may be treated as a $2$--prong center,
		and $\mathrm{ind}(f;\mathbf{p})$ equals $1$ or $-1$ likewise.
		Hence, 
		the negative value $1-k$ occurs precisely when $f$ preserves every prong at $x$.
		\item \emph{Crown circles}.
		The subset $\FPS(f;\mathbf{p})$ is a fixed boundary component $c$ of $S$ 
		or of a non-fixed component of $S_\reduction$.
		Then some component of $S_\pA$ must be adjacent to $c$.
		The fixed point index $\mathrm{ind}(f;\mathbf{p})$ equals $-k$
		where $k\geq1$ stands for the number of the foliation singular points on $c$.		
		\item \emph{Crown annuli}.
		The subset $\FPS(f;\mathbf{p})$ is a fixed component $A$ of $S_\reduction$.
		Then both of the components of $S_\pA\cup S_\periodic$ adjacent to $\partial A$ 
		must be from the pseudo-Anosov part, possibly coincident.
		(If some component from the periodic part was adjacent to $\partial A$,
		it would be fixed, falling into the next case.)
		Then the fixed point index $\mathrm{ind}(f;\mathbf{p})$ equals $-k$
		where $k\geq2$ stands for the total number 
		of the foliation singular points on $\partial A$.		
		\item \emph{Crown hyperbolic subsurfaces}.
		The subset $\FPS(f;\mathbf{p})$ is a fixed component $E$ of $S_\periodic$
		union with some (possibly none) adjacent fixed components $A_1,\cdots,A_r$ of $S_\reduction$.
		Then at least one component of $S_\pA$ must be adjacent to $A_s$, 
		for each $s\in\{1,\cdots,r\}$.
		The fixed point index $\mathrm{ind}(f;\mathbf{p})$ equals $\chi(E)-k$
		where $k\geq0$ stands for the total number of the foliation singular points on 
		$\partial (A_1\cup\cdots\cup A_r)$.		
		\end{enumerate}
		\end{example}
		
		By iterating $f$, one may also obtain 
		the classification of essential periodic orbit classes and their indices.
		The result is only notationally more involved.
		The description in Example \ref{example_fpc_classification}
		is sufficient for the subsequent sections.
		
	\subsection{Complexity of mapping classes}\label{Subsec-NT_complexity}
	Two useful numerical quantities can be extracted out of a Nielsen--Thurston normal form.
	They measure the asymptotic complexity of the mapping class under iteration.

	Let $S$ be a connected orientable compact surface of negative Euler characteristic.
	Suppose that $f\colon S\to S$ is an orientation-preserving self-homeomorphism of Nielsen--Thurston normal form,
	with respect to a Nielsen--Thurston decomposition $(S_\periodic,S_\pA,S_\reduction)$ of $S$.
	Denote by $d\in\Natural$ the smallest power such that	$f^d$ fixes $S_\periodic$ and $\partial S_\pA$.
	Then $f^d$ also fixes the boundary of the reduction part $\partial S_\reduction$,
	and preserves every component of $S_\periodic,S_\pA,S_\reduction$.
	Restricted to any component $U$ of $S_\pA$,
	$f^d$ acts as a pseudo-Anosov automorphism with some stretching factor $\lambda_U\in(1,+\infty)$.
	We define the (normalized maximal) \emph{dilatation} $\dilatation(f)\in[1,+\infty)$	of $f$
	to be the maximum of $\lambda_U^{1/d}$, as $U$ running over all the components of $S_\pA$,
	or $1$ if $S_\pA$ is empty.
	Restricted to any component of the reduction part,
	$f^d$ acts as an integral Dehn twist with some shearing degree $k_U\in\Integral/\{\pm1\}$.
	We define the (normalized maximal) \emph{deviation} $\deviation(f)\in[0,+\infty)$ of $f$ 
	to be the maximum of $|k_U|/d$, as $U$ running over all the components of $S_\reduction$,
	or $0$ is $S_\pA$ is empty.
	We sometimes refer to the above power $d\in\Natural$ as the \emph{split order} of $f$,
	since it is the smallest power such that $f^d$ factorizes as a commutative product
	of pseudo-Anosov factors supported on the pseudo-Anosov part and integral Dehn multi-twists
	supported on the reduction part.

	For any mapping class $[f]\in\mcg(S)$,
	we speak of the (normalized maximal) \emph{dilatation} and the (normalized maximal) \emph{deviation} of $[f]$ 
	by taking any representative $f$ of Nielsen--Thurston normal form, denoted as
	\begin{equation}\label{dil_dev_def}
	\begin{array}{cc}\dilatation([f])\in[1,+\infty),& \deviation([f])\in[0,+\infty),\end{array}
	\end{equation} 
	respectively.
	As such representatives of Nielsen--Thurston normal form
	are unique up to conjugation by homeomorphisms that are isotopic to the identity,	
	these quantities are well defined.
	Deviation and dilatation of mapping classes remain invariant under conjugation.
	They behave nicely under iteration, as indicated by the following relations:
	\begin{equation}\label{dil_dev_power}
	\begin{array}{cc}\dilatation([f]^m)=\dilatation([f])^m,& \deviation([f]^m)=m\cdot\deviation([f]),\end{array}
	\end{equation} 
	for all $[f]\in\mcg(S)$ and $m\in\Natural$.
	
	The following criterion says that there are only finitely many mapping classes
	up to conjugacy with uniformly bounded dilatation and deviation. 

	\begin{proposition}\label{bound_finiteness}
	Let $S$ be an orientable connected compact surface of negative Euler characteristic.
	For any constants $K>1$ and $C>0$,
	there exist finitely many mapping classes 
	$[f_1],\cdots,[f_r]\in \mcg(S)$
	with the following property:
	For any mapping class $[f']\in\mcg(S)$ with dilatation at most $K$ and deviation at most $C$,
	there exists some $s\in\{1,\cdots,r\}$
	such that $[f']$ is conjugate to $[f_s]$ in $\mcg(S)$.
	\end{proposition}

	\begin{proof}
	With the dilatation bound, we show that there are only finitely many allowable
	Nielsen--Thurston decompositions, 
	and only finitely many allowable restricted mapping classes to the periodic part and the pseudo-Anosov part.
	Then we show finiteness of the restrictions to the reduction annuli
	using the deviation bound.	
	
	For any $[f']\in\mcg(S)$,
	suppose that $f'$ is a representative of Nielsen--Thurston normal form,
	with respect to some Nielsen--Thurston decomposition $(S'_\periodic,S'_\pA,S'_\reduction)$ of $S$.
	Up to conjugation by orientation-preserving self-homeomorphisms of $S$,
	there are at most finitely many $(S'_\periodic,S'_\pA,S'_\reduction)$,
	as decompositions of $S$ along curves into subsurfaces.
	Indeed, the number of possible decomposition patterns can be bounded 
	in terms of the Euler characteristic of $S$.
	Up to isotopy of the periodic part $S'_\periodic$, 
	there are at most finitely many possible restriction of $f'$ 
	to $S'_\periodic$.
	Indeed, the number of possible isotopy classes of $(S'_\periodic,f')$
	can be bounded in terms of the Euler characteristic of $S'_\periodic$.
	(For example, this follows from \cite[Chapter 7, Theorem 7.14]{Farb--Margalit_book},
	applying the doubling trick for the bounded case.)
	Under the assumptions $\dilatation([f'])\leq K$,
	and up to isotopy of the pseudo-Anosov part $S'_\pA$,
	there are at most finitely many possible restrictions of $f'$ to $S'_\pA$.
	In fact,
	finiteness of pseudo-Anosov automorphisms with uniformly bounded stretching factor
	is a theorem due to Arnoux--Yoccoz \cite{Arnoux--Yoccoz_pA} and Ivanov \cite{Ivanov_pA},
	(see also \cite[Chapter 14, Theorem 14.9]{Farb--Margalit_book}).
	
	With the above finiteness results,
	there are only finitely many possible Nielsen--Thurston decomposition
	$(S'_\periodic,S'_\pA,S'_\reduction)$ and restrictions of $f'$ to $S'_\periodic\sqcup S'_\pA$,
	up to conjugacy.
	To consider each possibility individually,
	we fix a reference Nielsen--Thurston decomposition $(S_\periodic,S_\pA,S_\reduction)$, 
	and a reference orientation-preserving self-homeomorphism $f_0\colon S\to S$ of Nielsen--Thurston normal form.
	Choose an auxiliary orientation of $S$ and enumerate the core curves of $S_\reduction$
	as $a_1,\cdots,a_n$.
	Denote by $D_{a_j}$ the right-hand Dehn twist of $S$ along $a_j$, 
	supported on the annulus component of $S_\reduction$ that contains $a_j$.
	Note that these Dehn twists are mutually commutative.
	Suppose that $(S'_\periodic,S'_\pA,S'_\reduction)$ is conjugate to $(S_\periodic,S_\pA,S_\reduction)$,
	and that $f'$ is conjugate to $f_0$ restricted to the periodic and pseudo-Anosov parts.
	Then we can conjugate $f'$ to $f_0\circ D_{a_1}^{e_1}\cdots D_{a_n}^{e_n}$
	for some power $e_1,\cdots,e_n\in\Integral$.
	Moreover,	if $f_0(a_j)$ equals some $a_k$ other than $a_j$,
	the conjugation of $f_0\circ D_{a_1}^{e_1}\cdots D_{a_n}^{e_n}$ by $D_{a_j}^{-1}$
	will have the effect of increasing $e_j$ by $1$ and decreasing $e_k$ by $1$. 
	(Note that $D_{f_0(a_j)}\circ f_0$ equals $f_0\circ D_{a_j}$.)
	Therefore, 
	for any $f_0$--orbit of reduction curves $\{a_{j_1},\cdots,a_{j_t}\}$, 
	with $1\leq j_1<\cdots<j_t\leq n$,
	we can require furthermore that 
	$e_{j_2},\cdots,e_{j_t}$ are all $0$. Assuming $\deviation(f')<C$,
	it follows that $|e_{j_1}|$ must be bounded by $C+\deviation(f_0)$.
	In other words, there are at most finitely many possibilities of
	$f'$ restricted to the reduction part, up to conjugacy.
	This complete the proof of the asserted finiteness.
	\end{proof}

\section{Application to pseudo-Anosov mappings}\label{Sec-application_pA}
In this section, we prove Theorem \ref{main_entropy}.
We restate it formally as Theorem \ref{pA_dialogue}.
We mention Corollary \ref{pA_procongruent_almost_rigidity} as an immediate consequence
of Theorem \ref{pA_dialogue} (1) and Proposition \ref{bound_finiteness}.
This may be thought of as a special case of our main result Theorem \ref{main_procongruent_almost_rigidity}.

\begin{theorem}\label{pA_dialogue}
	Let $(S,f_A,f_B)$ be a dialogue setting where $S$ is closed of negative Euler characteristic
	and where $f_A,f_B$ are pseudo-Anosov mappings.
	Denote by $\mathscr{F}^{\mathtt{s/u}}_A,\mathscr{F}^{\mathtt{s/u}}_B$ 
	the (stable or unstable) invariant foliation of $f_A,f_B$, respectively.
	Suppose that the mapping classes $[f_A],[f_B]\in\mcg(S)$ are procongruently conjugate.
	Then the following statements all hold true:
	\begin{enumerate}
	\item The stretching factors for $f_A$ and $f_B$ are equal.
	\item For all $i\in\Integral$ other than $0$,
	the numbers of index--$i$ fixed points for $f_A$ and $f_B$ are equal.
	\item For all $k\in \Natural$ at least $3$,
	the numbers of $k$--prong singularities 
	for $\mathscr{F}^{\mathtt{s/u}}_A$ and $\mathscr{F}^{\mathtt{s/u}}_B$
	are equal.
	\item The foliation $\mathscr{F}^{\mathtt{s/u}}_A$ is transversely orientable
	if and only if $\mathscr{F}^{\mathtt{s/u}}_B$ is transversely orientable.	
\end{enumerate}
(See Section \ref{Subsec-NT_normal_form} and Convention \ref{settings}.)
\end{theorem}

\begin{corollary}\label{pA_procongruent_almost_rigidity}
Given any pseudo-Anosov mapping $f\colon S\to S$ of an orientable connected closed surface $S$
of negative Euler characteristic,
the procongruent conjugacy class of $[f]\in\mcg(S)$ 
contains at most finitely many conjugacy classes of pseudo-Anosov mapping classes.
\end{corollary}

The rest of this section is devoted to the proof of Theorem \ref{pA_dialogue}.
The proof is mainly an application of Theorem \ref{N_m_dialogue} to
pseudo-Anosov mappings.

It is known that the stretching factor of pseudo-Anosov mappings
can be detected by the exponential growth rate of the periodic Nielsen numbers.
To be precise, denote by $\lambda_A,\lambda_B$ the stretching factors
of $f_A,f_B$ respectively.
By Theorem \ref{N_m_dialogue} and (\ref{N_m_def}), 
we see that the $m$--periodic Nielsen number $N_m(f_A)$ equals $N_m(f_B)$, for all $m\in\Natural$. 
Then we obtain the equality
$$\lambda_A=\limsup_{m\to\infty} N_m(f_A)^{1/m}=\limsup_{m\to\infty} N_m(f_B)^{1/m}=\lambda_B,$$
(see \cite[Theorem 3.7]{Jiang_periodic}).
This proves Theorem \ref{pA_dialogue} (1).

According to the classification of essential fixed point classes 
for Nielsen--Thurston normal forms (Section \ref{Subsec-classification_fpc}),
the essential fixed point classes of $f_A$ correspond bijectively to the fixed points of $f_A$.
In fact, prong singularities or saddle points are the only possible essential
fixed point classes when we consider a pseudo-Anosov mapping of a closed surface
(see Example \ref{example_fpc_classification}), and these are all isolated fixed points.
With the notation of (\ref{nu_m_i}),
the number of index--$i$ fixed points of $f_A$ 
equals $\nu_1(f_A;i)$, for all $i\in\Integral\setminus\{0\}$.
The same correspondence and equality hold for $f_B$. 
We obtain $\nu_m(f_A;i)=\nu_m(f_B;i)$ for all $m\in\Natural$ and $i\in\Integral\setminus\{0\}$,
by Theorem \ref{N_m_dialogue}.
The special case $m=1$ implies Theorem \ref{pA_dialogue} (2).

According to the list of fixed point types and their indices (Example \ref{example_fpc_classification}),
the number of fixed $k$--prong singularities of $\mathscr{F}^{\mathtt{s/u}}_A$
is at least $\nu_1(f_A;1-k)$,
for any $k\in\Natural$ at least $3$.
By iterating $f_A$,
we can detect the number of $k$--prong singularities of $\mathscr{F}^{\mathtt{s/u}}_A$
by the maximum of $\nu_1(f_A^m;1-k)$ over all $m\in\Natural$.
In fact, the iterates $f_A^m$ are all in normal form 
with respect to the same measured foliations as those of $f_A$,
and $\nu_1(f_A^m;1-k)$ is maximized,
and equal to the number of $k$--prong singularities,
precisely when $f_A^m$ fixes all those singularities and preserves each of their prongs,
(see Example \ref{example_fpc_classification}).
The same detection holds for $f_B$.
Note that $f_A^m$ and $f_B^m$ are procongruently conjugate for all $m\in\Natural$.
Then Theorem \ref{N_m_dialogue},
applied to $\nu_1$ and the pairs of $m$--iterates,
implies Theorem \ref{pA_dialogue} (3).

The (stable or unstable) invariant foliation of a pseudo-Anosov mapping
is transversely orientable if and only if 
the stretching factor occurs as an eigenvalue of the induced homological action.
In fact, if the invariant foliation is transversely orientable,
any defining quadratic differential (with respect to any auxiliary complex structure) 
is the square of a closed $1$--form,
so the induced action on the first de Rham cohomolgy 
has an eigenvalue equal to the stretching factor;
if the invariant foliation is not transversely orientable,
the stretching factor still occurs as a homological eigenvalue of multiplicity one for
the lift of the pseudo-Anosov mapping to a double branched cover of the surface,
so no eigenvalues of the original homological action equals the stretching factor;
(see \cite[Chapter 14, Theorems 14.2 and 14.8]{Farb--Margalit_book} and the arguments thereof). 
By Theorem \ref{TRT_dialogue} (applied to the trivial complex representation),
the eigenvalues of $f_{A*},f_{B*}\colon H_1(S;\Complex)\to H_1(S;\Complex)$ agree.
Since the stretching factors $\lambda_A,\lambda_B$ also agree,
we see that $\mathscr{F}^{\mathtt{s/u}}_A$ is transversely orientable
if and only if $\mathscr{F}^{\mathtt{s/u}}_B$ is transversely orientable,
as asserted by Theorem \ref{pA_dialogue} (4).

This completes the proof of Theorem \ref{pA_dialogue}.

\section{Profinite structures associated to mapping classes}\label{Sec-profinite_structures}
	As we have seen, Corollary \ref{pA_procongruent_almost_rigidity} is 
	simpler than Theorem \ref{main_procongruent_almost_rigidity}
	because the complexity from deviation disappears for pseudo-Anosov mappings.
	For the general case, we need to find some suitable characterization of deviation
	on the level of profinite completions of mapping torus groups.
	In this section, we investigate profinite analogues of the geometric decomposition
	and essential indexed periodic trajectories classes.
	For procongruently conjugate pairs of mapping classes,
	we establish correspondence between the above objects (Theorem \ref{correspondence_profinite_objects}),
	based on the work of Wilton and Zalesskii \cite{WZ_decomposition} 
	and the proof of Theorem \ref{N_m_dialogue}.
	
	Below we provide an introduction to the picture,
	ending up with a comment on our exposition.
	The aim of this section is 
	to import the results of Wilton and Zalesskii
	into the context of surface mapping classes,
	and the task is mainly to translate between
	several canonical decompositions.
	These include the Nielsen--Thurston surface decomposition,
	and the geometric $3$--manifold decomposition, 
	and their tree action descriptions,
	and their profinite analogues.
	
	To illustrate with a simple example, 
	consider an orientation-preserving self-homeomorphism $f\colon S\to S$,
	where $S$ is a closed orientable surface of genus $2$,
	and where $f$ is obtained by first applying
	a Dehn twist on $S$ along a separating essential simple closed curve $c$,
	and then applying an involution on $S$ 
	flipping $c$ and switching the two one-hole tori separated by $c$.
	The Nielsen--Thurston decomposition is schematically a graph 
	with two vertices that correspond to the components of $S\setminus c$,
	and one edge that corresponds to $c$, joining the vertices;
	similarly,
	the lifted decomposition of $\widetilde{S}_\univ$ is schematically a tree
	with infinitely many vertices and edges, 
	and furnished with an isomorphic action of $\pi_1(S)$,
	whose orbits correspond to the decomposition graph of $S$.
	The geometric decomposition of the mapping torus $M_f$
	is obtained by cutting along a Klein bottle (which is the mapping torus of $f$ restricted to $c$),
	resulting in a single piece,
	so the decomposition should schematically be an `orbi-graph', 
	with one vertex and one `semi-edge' 
	(that is, the orbifold quotient of a segment by its central reflection);
	the lifted decomposition of $\widetilde{S}_\univ$ is a tree isomorphic to 
	the above tree of $\widetilde{S}_\univ$, 
	but furnished with an isomorphic action of $\pi_1(M_f)$ allowing edge flips.
	
	Heuristically, 
	the profinite tree associated to the geometric decomposition 
	should be the inverse limit 
	of the geometric decomposition graphs associated to finite covers of $M_f$,
	and there should be an action of $\widehat{\pi_1(M_f)}$,
	such that the vertex orbits and the edge orbits recover 
	that geometric decomposition graph.
	Apart from the orbi-graph notion,
	profinite graph theory is already available in the literature \cite{Ribes_book},
	and the case with profinite $3$--manifold groups 
	has been studied in great detail, 
	(in terms of the very simiar JSJ decomposition) \cite{WZ_decomposition}.
	However, profinite graph theory is more sophisticated than usual graph theory
	in some significant aspects:
	First, it involves topology more than the incidence relation,
	(for example, the vertex set is required to be a closed subset);
	secondly, the notion of being simply connected
	is generalized as vanishing of the first profinite homology;
	and thirdly, in many references, 
	a profinite graph is defined with a model
	somewhat like the set theoretic model for directed graphs,
	but formulated somewhat unusually to incorporate topology.
	
	The last point has direct influence on the exposition of this section.
	To translate between different decompositions,
	we formally construct our needed graphs (or orbi-graphs)
	in this section based on abstract directed graphs, 
	as introduced in \cite[Appendix A]{Ribes_book},
	because they are 
	parallel to the aforementioned model in profinite graph theory.
	On the other hand, we try to keep the terminology 
	close to its motivation as suggested by our example of illustration.
	Hopefully this would facilitate informal or blackbox use of 
	the dictionary in other circumstances.
	
	\subsection{Geometric decomposition of mapping tori}
		
		An \emph{abstract directed graph} 
		refers to a triple $(\mathfrak{G},d_0,d_1)$ as follows:
		The term $\mathfrak{G}$ is a set, and the terms 
		$d_0,d_1\colon\mathfrak{G}\to\mathfrak{G}$ 
		are maps;
		the maps $d_0,d_1$ are required to 
		be retractions onto one and the same subset of $\mathfrak{G}$,
		and the condition means $d_0^2=d_0$, and $d_1^2=d_1$, and
		$d_0(\mathfrak{G})=d_1(\mathfrak{G})$.
		The common image of $d_0$ and $d_1$ in $\mathfrak{G}$
		is called the set of \emph{abstract vertices},
		and the complement in $\mathfrak{G}$ is called 
		the set of \emph{abstract directed edges}.
		The image of any abstract directed edge under $d_0$ and $d_1$ are therefore
		called the \emph{initial} and the \emph{terminal} abstract vertex of that abstract directed edge.
		We often simply denote an abstract directed graph by $\mathfrak{G}$,
		with $d_0,d_1$ implicitly assumed.
		To declare an abstract direct graph,
		it suffices to point out the abstract vertices and the abstract directed edges,
		and the rule of assignment for the initial and the terminal abstract vertices.		
		A \emph{quasi-morphism} of abstract directed graphs
		$\psi\colon (\mathfrak{G}',d'_0,d'_1)\to (\mathfrak{G},d_0,d_1)$
		is a map $\psi\colon \mathfrak{G}'\to\mathfrak{G}$
		which satisfies $\psi\circ d'_0=d_0\circ\psi$ and $\psi\circ d'_1=d_1\circ\psi$.
		It is called a \emph{morphism} if it sends abstract directed edges to abstract directed edges,
		which is necessarily direction preserving.
		We refer the reader to \cite[Appendix A]{Ribes_book} for more details
		on graph theory in this language.
		
		With the notations of Definition \ref{NT_normal_form},
		we construct an abstract directed graph $\mathfrak{G}_\NT(f)$ as follows:
		The set of abstract vertices consists of the components of 
		the three parts $S_\periodic\sqcup S_\pA\sqcup S_\reduction$;
		the set of abstract directed edges consists of all the boundary components
		shared by a pair of components of the three parts;
		since any such boundary component is shared by 
		a component of $S_\reduction$ and a component of $S_\periodic\sqcup S_\pA$,
		the initial abstract vertex is assigned to be the former
		and the terminal abstract vertex to be the latter.
		Note that $\mathfrak{G}_\NT(f)$ can be viewed as 
		a finite directed simplicial graph in the usual sense.
		We denote the automorphism of the abstract directed graph
		induced by $f$ as
		$f_\sharp\colon \mathfrak{G}_\NT(f)\to\mathfrak{G}_\NT(f)$.
		We refer to the components of $S_\periodic\sqcup S_\pA$ 
		as the (periodic or pseudo-Anosov) \emph{Nielsen--Thurston vertices},
		and	the components of $S_\reduction$
		as the \emph{Nielsen--Thurston edge centers},
		and the directed edges as the \emph{Nielsen--Thurston edge ends}.
		
		\begin{definition}
		Let $S$ be a connected orientable compact surface of negative Euler characteristic.
		For any orientation-preserving self-homeomorphism $f\colon S\to S$ of 
		Nielsen--Thurston normal form	with respect to a Nielsen--Thurston decomposition of $S$,
		we refer to the above abstract directed graph $\mathfrak{G}_{\mathtt{NT}}(f)$	
		together with the infinite cyclic action generated by $f_\sharp$
		as (the \emph{barycentric model} of) the \emph{Nielsen--Thurston decomposition group-graph}.
		\end{definition}		
		
		Suppose that $(S,f)$ is a monologue setting
		where $S$ is of negative Euler characteristic 
		and $f$ of Nielsen--Thurston normal form.
		
		The geometric decomposition of the mapping torus $M_f$
		agrees with the suspension of the Nielsen--Thurston decomposition,
		and the dual graph is naturally modeled by
		the $f_\sharp$--orbits of the Nielsen--Thurston graph $\mathfrak{G}_\NT(f)$.
		To be precise,
		the components of the suspension of $S_\periodic$
		agree with the $\mathbb{H}^2\times\mathbb{E}^1$--geometric pieces,
		and correspond naturally with the $f_\sharp$--orbits of the periodic Nielsen--Thurston vertices;
		the components of the suspension of $S_\pA$ 
		agree with the $\mathbb{H}^3$--geometric pieces,
		and correspond naturally with the $f_\sharp$--orbits of the pseudo-Anosov Nielsen--Thurston vertices;
		the components of the suspension of $S_\reduction$ 
		agree with the thickened decomposition tori or Klein bottles,
		and correspond naturally with the $f_\sharp$--orbits of the Nielsen--Thurston edge centers.
		(The \emph{suspension} of any $f$--invariant subspace of $S$
		refers to the subspace of $M_f$ obtained as the mapping torus	of the restriction of $f$.)
		The adjacency relation between the geometric pieces and the thickened decomposition tori or Klein bottles
		is naturally encoded by the $f_\sharp$--orbits of the Nielsen--Thurston edge ends.
		We denote by $\mathfrak{G}_\gd(M_f)$ the quotient abstract directed graph of $\mathfrak{G}_\NT(f)$
		by the automorphism $f_\sharp$. 
		Note that $\mathfrak{G}_\gd(M_f)$ can be viewed as 
		a finite directed simplicial graph in the usual sense.
		The abstract vertices of $\mathfrak{G}_\gd(M_f)$
		are called the \emph{geometric vertices} or the \emph{geometric edge centers}, accordingly,
		and its abstract directed edges are called the \emph{geometric edge ends}.
		
		\begin{definition}
		Given any monologue setting $(S,f)$ where $S$ is of negative Euler characteristic 
		and $f$ of Nielsen--Thurston normal form,
		we refer to the above abstract directed graph 
		$\mathfrak{G}_{\mathtt{gd}}(M_f)$ as (the \emph{barycentric model} of) 
		the \emph{geometric decomposition graph} for $M_f$.		
		\end{definition}
		
		The universal covering space $\widetilde{S}_\univ\times\Real$ of $M_f$
		is also enriched with an induced decomposition,
		whose components are given by	the preimage components of the geometric pieces and 
		the thickened decomposition tori or Klein bottles of $M_f$.
		The dual graph $\mathfrak{T}_\gd(M_f)$ can be viewed naturally as
		a connected directed simplicial tree in the usual sense.
		Formally we construct it
		as an abstract directed graph in the obvious way.
		The deck transformation group $\pi_1(M_f)$ acts 
		naturally on $\mathfrak{T}_\gd(M_f)$ by automorphisms.
		The covering projection induces 
		a canonical surjective morphism of abstract directed graphs
		$\mathfrak{T}_\gd(M_f)\to \mathfrak{G}_\gd(M_f)$.
		The abstract vertices of $\mathfrak{T}_\gd(M_f)$
		that lies over the geometric vertices of $\mathfrak{G}_\gd(M_f)$
		are called the \emph{universal geometric vertices}.
		We also speak of the \emph{universal geometric edge centers} and 
		the \emph{universal geometric edge ends} of $\mathfrak{T}_\gd(M_f)$, 
		according to the objects of $\mathfrak{G}_\gd(M_f)$
		over which they lie.
		
		\begin{definition}
		Given any monologue setting $(S,f)$ where $S$ is of negative Euler characteristic 
		and $f$ of Nielsen--Thurston normal form,
		we refer to the above abstract directed graph 
		$\mathfrak{T}_{\mathtt{gd}}(M_f)$ 
		together with the action of $\pi_1(M_f)$
		as (the \emph{barycentric model} of) 
		the \emph{universal geometric group-tree} for $M_f$.		
		\end{definition}
		
		The inverse system of finite-index normal subgroups of $\pi_1(M)$
		induces a canonical inverse system of abstract directed graphs.
		It consists of 
		the finite quotient abstract directed graphs
		$\mathfrak{T}_\gd(M_f)/N$ 
		for all finite-index normal subgroups $N$ of $\pi_1(M)$,
		and the natural morphisms
		$\mathfrak{T}_\gd(M_f)/N'\to\mathfrak{T}_\gd(M_f)/N$
		for any $N'$ contained in $N$.
		The inverse limit 
		$\widehat{\mathfrak{T}_\gd(M_f)}=\varprojlim_N \,\mathfrak{T}_\gd(M_f)/N$
		is again an abstract directed graph.
		The profinite completion $\widehat{\pi_1(M_f)}$ 
		acts naturally on $\widehat{\mathfrak{T}_\gd(M_f)}$ 
		by automorphisms of the abstract directed graph.
		There is a canonical surjection morphism of abstract directed graphs
		$\widehat{\mathfrak{T}_\gd(M_f)}\to\mathfrak{G}_\gd(M_f)$,
		by the universal property of inverse limit.
		Therefore, we speak of 
		the \emph{profinite geometric vertices}, the \emph{profinite geometric edge centers},
		and \emph{profinite geometric edge ends} of $\widehat{\mathfrak{T}_\gd(M_f)}$,
		according to the objects of $\mathfrak{G}_\gd(M_f)$
		over which they lie.
		
		\begin{definition}
		Given any monologue setting $(S,f)$ where $S$ is of negative Euler characteristic 
		and $f$ of Nielsen--Thurston normal form,
		we refer to the above abstract directed graph $\widehat{\mathfrak{T}_{\gd}(M_f)}$	
		together with the action of $\widehat{\pi_1(M_f)}$
		as (the \emph{barycentric model} of)
		the \emph{profinite geometric group-tree} for the mapping torus $M_f$.
		\end{definition}
		
		
		The natural morphism of $\mathfrak{T}_\gd(M_f)$ onto $\mathfrak{G}_\gd(M_f)$
		can be identified with the quotient of $\mathfrak{T}_\gd(M_f)$ by the action of $\pi_1(M_f)$.
		In particular, 
		elements of $\mathfrak{T}_\gd(M_f)$ that lie over one and the same element
		of $\mathfrak{G}_\gd(M_f)$ have mutually conjugate stabilizers in $\pi_1(M_f)$.
		The natural morphism of $\widehat{\mathfrak{T}_\gd(M_f)}$ onto $\mathfrak{G}_\gd(M_f)$
		can be identified 
		with the quotient of $\widehat{\mathfrak{T}_\gd(M_f)}$ by the action of $\widehat{\pi_1(M_f)}$,
		using a standard convergence argument.		
		The natural morphism of $\mathfrak{T}_\gd(M_f)$ to $\widehat{\mathfrak{T}_\gd(M_f)}$
		is injective. 
		This follows from a general criterion
		and the fact that 
		the stabilizers of the universal geometric edge ends in $\pi_1(M_f)$ are separable,
		(see \cite[Proposition 2.5]{Cotton-Barratt} and \cite{Hamilton_abelian_separability}).
		Moreover,
		the natural inclusion of $\mathfrak{T}_\gd(M_f)$ into $\widehat{\mathfrak{T}_\gd(M_f)}$
		is equivariant with respect to
		the natural inclusion of $\pi_1(M_f)$ into $\widehat{\pi_1(M_f)}$.
		We draw the following schematic diagram
		to help organize
		the group actions and the abstract directed graph morphisms:
		\begin{equation}\label{TG_diagram}
		\xymatrix{
		\pi_1(M_f) \ar[d]_{\textrm{incl.}} \ar@{.}[r]^-\circlearrowleft & \mathfrak{T}_\gd(M_f) \ar[r]^-{\textrm{quot.}} \ar[d]_{\textrm{incl.}} 
		& \mathfrak{G}_\gd(M_f) \ar[d]^{\mathrm{id}} \\
		\widehat{\pi_1(M_f)} \ar@{.}[r]^-\circlearrowleft & \widehat{\mathfrak{T}_\gd(M_f)} \ar[r]^-{\textrm{quot.}} & \mathfrak{G}_\gd(M_f).
		}
		\end{equation}
		The square on the left is equivariant, and 
		the square on the right is commutative.
		
		\begin{proposition}\label{graph_element_stabilizers}
			In a monologue setting $(S,f)$,
			suppose that $S$ is of negative Euler characteristic 
			and $f$ of Nielsen--Thurston normal form.
			Treat $\pi_1(M_f)$ and $\mathfrak{T}_\gd(M_f)$
			as naturally included by $\widehat{\pi_1(M_f)}$ and $\widehat{\mathfrak{T}_\gd(M_f)}$,
			respectively.			
			
			Then any element $\hat{\mathfrak{u}}$ of $\widehat{\mathfrak{T}_\gd(M_f)}$ 
			is conjugate under $\widehat{\pi_1(M_f)}$
			to some element $\tilde{\mathfrak{u}}$ of $\mathfrak{T}_\gd(M_f)$.
			Hence $\hat{\mathfrak{u}}$ and $\tilde{\mathfrak{u}}$ have 
			mutually conjugate stabilizers in $\widehat{\pi_1(M_f)}$.
			Moreover, 
			the stabilizer of $\tilde{\mathfrak{u}}$ in $\widehat{\pi_1(M_f)}$
			equals the profinite closure of 
			the stabilizer of $\tilde{\mathfrak{u}}$ in $\pi_1(M_f)$.
		\end{proposition}
		
		\begin{proof}
			Denote by $\bar{\mathfrak{u}}$ the element of $\mathfrak{G}_\gd(M_f)$
			over which lies $\hat{\mathfrak{u}}$.
			Denote $U$ be the subspace of $M_f$ that corresponds to $\bar{\mathfrak{u}}$ according to the geometric decomposition.
			So $U$ is either a geometric piece, or a thickened decomposition torus or Klein bottle,
			or a boundary torus of a thickened decomposition torus or Klein bottle.
			Take a preimage component $\tilde{U}$ in the universal covering space $\widetilde{S}_\univ\times\Real$ of $M_f$.
			Denote by $\tilde{\mathfrak{u}}$ the corresponding element of $\mathfrak{T}_\gd(M_f)$.
			For any finite-index normal subgroup $N$ of $\pi_1(M_f)$,
			denote by $\widehat{N}$ the profinite closure of $N$ in $\widehat{\pi_1(M_f)}$.
			Then $\hat{\mathfrak{u}}$ and $\tilde{\mathfrak{u}}$ have conjugate images
			in the finite quotient graph $\widehat{\mathfrak{T}_\gd(M_f)}/\widehat{N}$,
			which equals $\mathfrak{T}_\gd(M_f)/N$.	
			The elements in $\widehat{\pi_1(M_f)}$ that send the image of $\hat{\mathfrak{u}}$
			to the image of $\tilde{\mathfrak{u}}$ 
			form a finite union of cosets of $\widehat{N}$,
			which is closed with respect to the profinite topology.
			By a standard convergence argument,
			there exists an element $\hat{h}\in\widehat{\pi_1(M_f)}$
			which lies in the above union of cosets for all $\widehat{N}$.
			In other words, the action of $\hat{h}$ on $\widehat{\mathfrak{T}_\gd(M_f)}$ 
			sends $\hat{\mathfrak{u}}$ to $\tilde{\mathfrak{u}}$,
			and conjugates the stabilizer of $\hat{\mathfrak{u}}$ to that of $\tilde{\mathfrak{u}}$.
			It is clear that the stabilizer of $\tilde{\mathfrak{u}}$ in $\widehat{\pi_1(M_f)}$
			contains the profinite closure in $\widehat{\pi_1(M_f)}$
			of the stabilizer of $\tilde{\mathfrak{u}}$ in $\pi_1(M_f)$.
			On the other hand,
			suppose that $\hat{g}\in\widehat{\pi_1(M_f)}$ stabilizes $\tilde{\mathfrak{u}}$.
			Then for any finite-index normal subgroup $N$ of $\pi_1(M_f)$,
			the image $\hat{g}\widehat{N}$ of $\hat{g}$ in the finite quotient group 
			$\widehat{\pi_1(M_f)}/\widehat{N}$
			acts on the finite covering space $(\widetilde{S}_\univ\times\Real)/N$ of $M_f$,
			as a deck transformation that preserves the image of $\widetilde{U}$.
			It follows that $\hat{g}$ lies in the profinite closure of the stabilizer of $\tilde{\mathfrak{u}}$ in $\pi_1(M_f)$.
			This completes the proof.
		\end{proof}

		\subsection{Essential periodic trajectory classes of mapping tori}
		In any monologue setting $(S,f)$,
		we have introduced $m$--periodic trajectory classes
		$\ell_m(f;\mathbf{O})\in\conjorb(\pi_1(M_f))$,
		for $m$--periodic orbit classes $\mathbf{O}\in\POC_m(f)$, (see (\ref{gamma_m})).
		By naturally including $\pi_1(M_f)$ into its profinite completion
		$\widehat{\pi_1(M_f)}$,
		we denote by 
		\begin{equation}\label{gamma_hat_m}
		\hat{\ell}_m(f;\mathbf{O})\in\conjorb\left(\widehat{\pi_1(M_f)}\right)
		\end{equation}
		the conjugation orbit of $\widehat{\pi_1(M_f)}$
		that contains the image of $\ell_m(f;\mathbf{O})$.
		
		For any conjugation orbit $\hat{\mathbf{c}}\in\widehat{\pi_1(M_f)}$,
		we say that $\hat{\mathbf{c}}$ 
		is an \emph{essential profinite $m$--periodic trajectory class}
		if $\hat{\mathbf{c}}$ equals 
		$\hat{\ell}_m(f;\mathbf{O})$ for some essential $m$--periodic orbit class
		$\mathbf{O}\in\POC_m(f)$.
		One may think of such classes
		as profinite analogues of essential periodic trajectories
		up to free homotopy.
		For any $m\in\Natural$ and $i\in\Integral\setminus\{0\}$,
		we denote by
		\begin{equation}\label{profinite_essential_orbit_set}
		\conjorb_m\left(\widehat{\pi_1(M_f)};i\right)=
		\left\{\hat{\ell}_m(f;\mathbf{O})\colon \mathbf{O}\in\POC_m(f),\,\mathrm{ind}_m(f;\mathbf{O})=i\right\},
		\end{equation}
		the set of essential profinite $m$--periodic trajectory class of index $i$.
		It is a subset of $\conjorb(\widehat{\pi_1(M_f)})$.
				
		\begin{proposition}\label{profinite_orbit_set_description}
		In any monologue setting $(S,f)$, 
		for all $m\in\Natural$ and $i\in\Integral\setminus\{0\}$,
		the sets of essential profinite periodic trajectory classes
		$\conjorb_m(\widehat{\pi_1(M_f)};i)$
		are mutually disjoint,
		and are finite of cardinality $\nu_m(f;i)$,
		as introduced by (\ref{nu_m_i}).
		\end{proposition}

		\begin{proof}
		For all $m\in\Natural$ and $i\in\Integral\setminus\{0\}$, 
		the maps $\POC_m(f;i)\to \conjorb(\pi_1(M_f))$ given by $\mathrm{O}\mapsto \ell_m(f;\mathbf{O})$
		are by definition injective with mutually disjoint images of finite cardinality $\nu_m(f;i)$,
		(see Section \ref{Subsec-fixed_point_theory} and (\ref{nu_m_i})).
		Denote by $\conjorb_m(\pi_1(M_f);i)$ the image of $\POC_m(f;i)$
		contained in $\conjorb(\pi_1(M_f))$.
		As $\pi_1(M_f)$ is conjugacy separable,
		the natural inclusion of $\pi_1(M_f)$ into $\widehat{\pi_1(M_f)}$ induces
		an inclusion of $\conjorb(\pi_1(M_f))$ into $\conjorb(\widehat{\pi_1(M_f)})$,
		(see the proof of Lemma \ref{N_m_detect} for elaboration).
		Therefore, $\conjorb_m(\pi_1(M_f);i)$ are mapped
		bijectively onto $\conjorb_m(\widehat{\pi_1(M_f)};i)$ 
		for all $m\in\Natural$ and $i\in\Integral\setminus\{0\}$.
		It follows that $\conjorb_m(\widehat{\pi_1(M_f)};i)$
		are mutually disjoint of finite cardinality $\nu_m(f;i)$, as asserted.
		\end{proof}

		\begin{remark}
			For all $m\in\Natural$,
			it seems also reasonable to define 
			the	\emph{inessential profinite $m$--periodic trajectory classes}
			as the conjugation orbits of $\widehat{\pi_1(M_f)}$
			which have $\widehat{\phi}_f$--grading $m$ and which
			are not contained in any $\conjorb_m(\widehat{\pi_1(M_f)};i)$
			with $i\in\Integral\setminus\{0\}$.
			Then they form an infinite subset
			$\conjorb_m(\widehat{\pi_1(M_f)};0)$ of $\conjorb(\widehat{\pi_1(M_f)})$
			with cardinality $\aleph_1$.
			Here $\widehat{\phi}_f$ stands for the profinite extension of the distinguished cohomology class,
			see (\ref{ses_mapping_torus_group}).
		\end{remark}

		\subsection{Correspondence between profinite objects}
		
		\begin{theorem}\label{correspondence_profinite_objects}
			In a dialogue setting $(S,f_A,f_B)$,
			suppose that $S$ is of negative Euler characteristic 
			and $f_A,f_B$ of Nielsen--Thurston normal form 
			representing procongurently conjugate mapping classes.
			Then any aligned isomorphism $\Psi\colon\widehat{\pi_1(M_A)}\to\widehat{\pi_1(M_B)}$
			induces an equivariant isomorphism of abstract directed graphs with group actions:
			$$\Psi_\sharp\colon \widehat{\mathfrak{T}_\gd(M_A)}\to \widehat{\mathfrak{T}_\gd(M_B)}$$
			between the profinite geometric group-trees.
			Hence, $\Psi$ induces	an isomorphism of abstract directed graphs:
			$$\overline{\Psi_\sharp}\colon \mathfrak{G}_\gd(M_A)\to \mathfrak{G}_\gd(M_B)$$
			between the geometric decomposition graphs.
			Moreover, for all $m\in\Natural$ and $i\in\Integral\setminus\{0\}$,
			$\Psi$ induces bijections:
			$$\conjorb_m(\Psi;i)\colon \conjorb_m\left(\widehat{\pi_1(M_A)};i\right)\to\conjorb_m\left(\widehat{\pi_1(M_B)};i\right),$$
			between the sets of indexed profinite periodic trajectory classes.
		\end{theorem}
		
		\begin{proof}
			The isomorphisms $\Psi_\sharp$ and $\overline{\Psi_\sharp}$ 
			follow from the deep work of Wilton--Zalesskii \cite{WZ_decomposition},
			which shows that profinite completion detects the Jaco--Shalen--Johanson decomposition
			and also the equivariant profinite-isomorphism type of the profinite Bass--Serre tree.
			(See also \cite[Theorem I]{Wilkes_thesis} for a partial result toward the same direction.)
			Their result can be adapted here as follows.
			In any monologue setting $(S,f)$ where $S$ is of negative Euler characteristic
			and $f$ of Nielsen--Thurston normal form,
			the Jaco--Shalen--Johanson (JSJ) decomposition of $M_f$ is closely related to the geometric decomposition.
			The JSJ tori are the geometric decomposition tori together with the boundary tori of 
			the thickened geometric decomposition Klein bottles;
			the JSJ pieces are the geometric pieces together with the thickened geometric decomposition Klein bottles.
			Therefore, 
			the usual JSJ graph of $M_f$ can be recovered from
			our barycentric model of the geometric decomposition graph $\mathfrak{G}_\gd(M_f)$
			by flattening out the torus-type geometric edge centers in the obvious way.
			Using the well-known correspondence between universal deck transformation groups
			and fundamental groups,
			a standard Bass--Serre tree for the induced graph-of-groups decomposition of $\pi_1(M_f)$
			can also be recovered from 
			our barycentric model of the universal geometric decomposition tree $\mathfrak{T}_\gd(M_f)$
			by flatting out the overlying geometric edge centers of the same type.
			Then the isomorphisms $\Psi_\sharp$ and $\overline{\Psi_\sharp}$ 
			are implied by \cite[Theorem 4.3]{WZ_decomposition}.
			Note that \cite[Theorem 4.3]{WZ_decomposition}
			also works for the compact case with tori boundary,
			which following from a well-known doubling trick,
			(for example, 
			see the argument of \cite[Proof of Corollary C]{WZ_decomposition}).
			
			The bijections $\conjorb_m(\Psi;i)$ are essentially implied by the proof
			of Theorem \ref{N_m_dialogue}.
			As we have argued at the end of Section \ref{Sec-indexed_orbit_numbers},
			$L_m(f_A;\gamma_A^*\chi_\mathbf{c})=L_m(f_B;\gamma_B^*\chi_\mathbf{c})$
			holds	for any $\Psi$--aligned finite quotient $\gamma_A\colon\pi_1(M_A)\to\Gamma$ 
			and $\gamma_B\colon\pi_1(M_B)\to\Gamma$, 
			and for any conjugation orbit $\mathbf{c}\in\conjorb(\Gamma)$,
			(Theorem \ref{L_m_dialogue} and Lemma \ref{L_chi_c}).
			By (\ref{L_m_chi}) and (\ref{chi_c}),
			this means that an conjugation orbit $\mathbf{c}\in\conjorb(\Gamma)$ 
			lies in the image of $\conjorb_m(\widehat{\pi_1(M_A)};i)$
			if and only if it lies in the image of $\conjorb_m(\widehat{\pi_1(M_B)};i)$,
			so the images coincide for all $m\in\Natural$ and $i\in\Integral$.
			Moreover, for any $\Gamma$ that satisfies the conclusion of Lemma \ref{N_m_detect},
			both $\conjorb_m(\widehat{\pi_1(M_A)};i)$ and $\conjorb_m(\widehat{\pi_1(M_B)};i)$
			must be mapped injectively into $\conjorb(\Gamma)$,
			for all $i\in\Integral\setminus\{0\}$.
			Note that $\conjorb(\widehat{\pi_1(M_A)})$ 
			can be naturally identified as the inverse limit of
			the inverse system of finite quotients
			$\conjorb(\gamma_A)\colon \conjorb(\pi_1(M_A))\to \conjorb(\Gamma)$.
			The same holds for $\conjorb(\widehat{\pi_1(M_B)})$
			with the $\Psi$--aligned inverse system
			$\conjorb(\gamma_B)\colon \conjorb(\pi_1(M_B))\to \conjorb(\Gamma)$.
			By the standard construction of inverse limits,
			we see that $\Psi$ induces bijective correspondences $\conjorb_m(\Psi;i)$
			between	$\conjorb_m(\widehat{\pi_1(M_A)};i)$ and $\conjorb_m(\widehat{\pi_1(M_B)};i)$,
			for all $m\in\Natural$ and $i\in\Integral\setminus\{0\}$,
			as asserted.
		\end{proof}

\section{Procongruent almost rigidity}\label{Sec-procongruent_almost_rigidity}
In this section, we prove Theorem \ref{main_procongruent_almost_rigidity}.
We restate it formally as Theorem \ref{procongruent_conjugacy_finiteness}.

\begin{theorem}\label{procongruent_conjugacy_finiteness}
	For any mapping class $[f]\in\mcg(S)$
	of an orientable connected compact surface $S$,
	there exist finitely many mapping classes 
	$[f_1],\cdots,[f_r]\in \mcg(S)$
	with the following property:
	For any mapping class $[f']\in\mcg(S)$ which is procongruently conjugate to $[f]$,
	there exists some $s\in\{1,\cdots,r\}$
	such that $[f']$ is conjugate to $[f_s]$ in $\mcg(S)$.
\end{theorem}

The rest of this section is devoted to the proof of Theorem \ref{procongruent_conjugacy_finiteness}.
We start with some simplification of our situation.

\begin{lemma}\label{simplification_lemma}
	The statement of Theorem \ref{procongruent_conjugacy_finiteness} holds true in general if 
	it holds true assuming all the following additional conditions:
	\begin{itemize}
	\item The surface $S$ is of negative Euler characteristic,
	and the mapping class $[f]$ is represented by some $f$ of Nielsen--Thurston norm form,
	with respect to a Nielsen--Thurston decomposition $(S_\periodic,S_\pA,S_\reduction)$ of $S$.
	\item For every component $U$ of the reduction part $S_\reduction$, 
	the homomorphism $H_1(U;\Integral)\to H_1(S;\Integral)$ induced by inclusion is injective.
	\item The representative $f$ fixes the periodic part $S_\periodic$ and
	the pseudo-Anosov part boundary $\partial S_\pA$.
	In other words, $f$ is a commutative product of pseudo-Anosov factors supported on the components of $S_\pA$
	and integral Dehn twists supported on the components of $S_\reduction$.
	\end{itemize}
\end{lemma}

\begin{proof}
	When $S$ is of zero or positive Euler characteristic, it is a disk, a sphere, an annulus, or a torus.
	Only the last case is nontrivial, and we have already explained it in Example \ref{Stebe_example}.
	Below we assume that $S$ is of negative Euler characteristic.

	We apply a familiar trick of taking characteristic abelian finite covers.
	Given any mapping class $[f]\in\mcg(S)$, 
	we take a representative $f\colon S\to S$
	of Nielsen--Thurston normal form with respect to some Nielsen--Thurston decomposition
	$(S_\periodic,S_\pA,S_\reduction)$ of $S$.
	Take $\tilde{S}\to S$ to be the characteristic finite cover 
	that corresponds to the kernel of the natural homomorphism 
	$\pi_1(S)\to H_1(S;\Integral/2\Integral)$.
	Take $\tilde{f}\colon \tilde{S}\to \tilde{S}$ to be any elevation of $f$ to $\tilde{S}$.
	It follows that $\tilde{f}$ is also of Nielsen--Thurston normal form,
	with respect to the induced Nielsen--Thurston decomposition
	$(\tilde{S}_\periodic,\tilde{S}_\pA,\tilde{S}_\reduction)$ of $\tilde{S}$.
	For some $m\in\Natural$, 
	the mapping torus of $\tilde{f}^m$ naturally projects onto the mapping torus of $f$
	as a regular finite cover.
	Take $\tilde{F}$ to be $\tilde{f}^m$ for some such $m$,
	and also require $m$ to be divisible by the split order of $\tilde{f}$.
	
	Observe that the monologue setting $(\tilde{S},\tilde{F})$ 
	satisfies all the listed conditions of Lemma \ref{simplification_lemma}.
	Having assumed Theorem \ref{procongruent_conjugacy_finiteness} for this case,
	we obtain finitely many mapping classes $[\tilde{F}_1],\cdots,[\tilde{F}_R]\in\mcg(\tilde{S})$.
	Apply Lemma \ref{bound_finiteness}	with $C=\max_{t\in\{1,\cdots,R\}}\,\dilatation([\tilde{F}_t])^{1/m}$
	and $K=\max_{t\in\{1,\cdots,R\}}\,\deviation([\tilde{F}_t])\cdot(2/m)$.
	We obtain finitely many mapping classes $[f_1],\cdots,[f_r]\in\mcg(S)$.		
	Since $[f_1],\cdots,[f_r]\in\mcg(S)$ are obtained for any $[f]\in\mcg(S)$
	based on the hypothsis of Lemma \ref{simplification_lemma},
	it remains to verify that they satisfy the asserted property
	of Theorem \ref{procongruent_conjugacy_finiteness}.
	Suppose that $[f']\in\mcg(S)$ is procongruently conjugate to $[f]$.
	Then for some elevation $[\tilde{f}']\in\mcg(\tilde{S})$ of $[f']$ to $\tilde{S}$,
	the mapping class $[\tilde{f}']^m\in\mcg(\tilde{S})$ is procongruently conjugate to $[\tilde{F}]$,
	by Proposition \ref{characterization_mapping_tori}.
	It follows that $[\tilde{f}']^m\in\mcg(\tilde{S})$ 
	is conjugate to $[\tilde{F}'_t]$ for some $t\in\{1,\cdots,R\}$,
	so we observe $\dilatation([f'])\leq C$ and $\deviation([f'])\leq K$.
	It follows that $[f']\in\mcg(S)$ is conjugate to  $[f_s]$
	for some $s\in\{1,\cdots,r\}$, as asserted.
\end{proof}

To simplify notations,
given any monologue setting $(S,f)$ that satisfies the additional conditions listed by Lemma \ref{simplification_lemma},
we denote by $\Pi$ the mapping torus group $\pi_1(M_f)$, and 
by $\widehat{\Pi}$ the profinite completion $\widehat{\pi_1(M_f)}$.
We denote by $\mathfrak{G}$ the geometric decomposition graph $\mathfrak{G}_\gd(M_f)$ of the mapping torus $M_f$,
and by $(\mathfrak{T},\Pi)$ the universal geometric group-tree $(\mathfrak{T}_\gd(M_f),\pi_1(M_f))$,
and by $(\widehat{\mathfrak{T}},\widehat{\Pi})$ the profinite geometric group-tree 
$(\widehat{\mathfrak{T}_\gd(M_f)},\widehat{\pi_1(M_f)})$.
We include $\Pi$ naturally into $\widehat{\Pi}$, 
and $\mathfrak{T}$ equivariantly into $\widehat{\mathfrak{T}}$,
as (\ref{TG_diagram}).
For any graph element $\tilde{\mathfrak{u}}$ of $\mathfrak{T}$,
we denote by $\Pi_{\tilde{\mathfrak{u}}}$ the stabilizer of $\tilde{\mathfrak{u}}$ in $\Pi$.
The notations $\widehat{\Pi}_{\hat{\mathfrak{u}}}$ for any graph element $\hat{\mathfrak{u}}$ of $\widehat{\mathfrak{T}}$
are understood likewise.
		
Under the listed conditions of Lemma \ref{simplification_lemma},
every geometric edge center  $\mathfrak{e}$ of $\mathfrak{G}$ 
corresponds to a thickened decomposition torus $M_{f,\mathfrak{e}}$ of $M_f$.
It is the suspension of 
a unique component $S_{\mathfrak{e}}$ of the reduction part $S_\reduction$.
There are exactly two geometric edge ends of $\mathfrak{G}$ departing from $\mathfrak{e}$.
We denote by $\partial_{\mathfrak{b}} M_{f,\mathfrak{e}}$ 
the boundary component of $M_{f,\mathfrak{e}}$
that corresponds to a geometric edge end $\mathfrak{b}$ departing from $\mathfrak{e}$.
Note that $\partial_{\mathfrak{b}} M_{f,\mathfrak{e}}$ 
is homeomorphic to a torus and 
is foliated by closed $1$--periodic trajectories of the suspension flow.
As these closed trajectories are mutuually parallel,
they represent one and the same free-homotopy loop of $M_f$,
and therefore a unique conjugacy orbit of $\Pi$,
denoted as
$\ell_1(f;\mathfrak{b})\in \conjorb(\Pi)$.
For any $m\in\Natural$, we denote by 
\begin{equation}\label{l_m_f_b}
\ell_m(f;\mathfrak{b})\in\conjorb(\Pi)
\end{equation}
the $m$--th power of $\ell_1(f;\mathfrak{b})$.
In other words, $\ell_m(f;\mathfrak{b})$ is represented by
the closed $m$--periodic trajectories of the suspension flow
immersed in $\partial_{\mathfrak{b}}M_{f,\mathfrak{e}}$.
We denote by
\begin{equation}\label{hat_l_m_f_b}
\hat{\ell}_m(f;\mathfrak{b})\in\conjorb(\widehat{\Pi})
\end{equation}
the unique conjugation orbit of $\widehat{\Pi}$ that contains $\ell_m(f;\mathfrak{b})$
under the natural inclusion of $\Pi$ into $\widehat{\Pi}$.

Our listed conditions of Lemma \ref{simplification_lemma} are convenient
mostly because they simplify the proof for the following lemma.
It should certainly hold in more generality,
but the current version suffices for our argument and also illustrates the point.

\begin{lemma}\label{unique_intersection}
	In a monologue setting $(S,f)$, 
	suppose that the listed conditions of Lemma \ref{simplification_lemma} are satisfied.
	Suppose that $\hat{\mathfrak{e}}$ 
	is any profinite geometric edge center of $\widehat{\mathfrak{T}}$.
	Then for any conjugation orbit $\hat{\mathbf{c}}\in\conjorb(\widehat{\Pi})$,
	there is at most one element $\hat{g}\in\hat{\mathbf{c}}$ which lies in $\widehat{\Pi}_{\hat{\mathfrak{e}}}$.	
\end{lemma}

\begin{proof}
	The additional conditions as listed by Lemma \ref{simplification_lemma}
	implies that for any universal geometric edge center $\tilde{\mathfrak{e}}$ of $\mathfrak{T}$,
	the stabilizer $\Pi_{\tilde{\mathfrak{e}}}$ is isomorphic to a free abelian group of rank $2$.
	Moreover, under the abelianization group homomorphism $\Pi\to H_1(\Pi;\Integral)$, 
	the stabilizer $\Pi_{\mathfrak{e}}$ is projected isomorphically onto its image.
	By Proposition \ref{graph_element_stabilizers}, 
	for any profinite geometric edge center $\hat{\mathfrak{e}}$ of $\widehat{\mathfrak{T}}$,
	the stabilizer $\widehat{\Pi}_{\hat{\mathfrak{e}}}$ is isomorphic to
	a profinite free abelian group of rank $2$.
	Moreover, under the abelianization group homomorphism $\widehat{\Pi}\to \widehat{H}_1(\widehat{\Pi};\widehat{\Integral})$,
	the stabilizer $\widehat{\Pi}_{\hat{\mathfrak{e}}}$ is projected isomorphically onto its image.
	In particular, the elements of $\widehat{\Pi}_{\hat{\mathfrak{e}}}$ 
	are mutually non-conjugate in $\widehat{\Pi}$.
	Therefore, any conjugation orbit $\hat{\mathbf{c}}\in\conjorb(\widehat{\Pi})$
	intersects $\widehat{\Pi}_{\hat{\mathfrak{e}}}$ in at most one element.
\end{proof}

To characterize deviation, the key observation is that the essential periodic trajectories of $M_f$ that can be
freely homotoped into a thickened decomposition torus are all freely homotopic to 
the periodic trajectories in the thickened decomposition torus.
This is the main content of the following Lemma \ref{carry_trajectory_universal}.
It is stated in group theoretic terms, so as to compare with 
its profinite analogue Lemma \ref{carry_trajectory_profinite}.

\begin{lemma}\label{carry_trajectory_universal}
	In a monologue setting $(S,f)$, 
	suppose that the listed conditions of Lemma \ref{simplification_lemma} are satisfied.
	Suppose that $\tilde{\mathfrak{e}}$ is a universal geometric edge center of $\mathfrak{T}$
	which lies over a geometric edge center $\mathfrak{e}$ of $\mathfrak{G}$.
	Then the following statements hold true for all $m\in\Natural$.
	\begin{enumerate}
	\item
	For any geometric edge end $\mathfrak{b}$ of $\mathfrak{G}$ departing from $\mathfrak{e}$,
	$\ell_m(f;\mathfrak{b})\in\conjorb(\Pi)$ is an essential $m$--periodic trajectory class.
	\item
	For any essential $m$--periodic orbit class $\mathbf{O}\in\POC_m(f)$,
	if $\ell_m(f;\mathbf{O})\in\conjorb(\Pi)$
	has nonempty intersection with $\Pi_{\tilde{\mathfrak{e}}}$,
	then $\ell_m(f;\mathbf{O})$ is equal to $\ell_m(f;\mathfrak{b})$
	for some geometric edge end $\mathfrak{b}$ departing from $\mathfrak{e}$.
	\end{enumerate}
\end{lemma}

\begin{proof}
	The first statement follows from Example \ref{example_fpc_classification} (3) and (4),
	directly for $m=1$ and by iteration for general $m\in\Natural$,
	(see definition (\ref{l_m_f_b}) and Section \ref{Subsec-classification_fpc}).
	
	The second statement follows from a case by case analysis according to the classification
	of essential periodic orbit classes (Section \ref{Subsec-classification_fpc}).
	To elaborate, for each essential $\mathbf{O}\in\POC_m(f)$,
	there is a corresponding component $K_\mathbf{O}$ of $\FPS(f^m)$, and 
	$\ell_m(f;\mathbf{O})$ is represented, as a free-homotopy loop of $M_f$,
	by the $m$--periodic trajectory of the suspension flow through any point $x\in K_{\mathbf{O}}$.
	Under the listed conditions of Lemma \ref{simplification_lemma},
	$K_\mathbf{O}$ may occur as Example \ref{example_fpc_classification} (2)--(5),
	with respect to $f^m$.
	Below we assume that $\ell_m(f;\mathbf{O})$, as a conjugation orbit of $\Pi$,
	has nonempty intersection with $\Pi_{\tilde{\mathfrak{e}}}$.
	This means that $\ell_m(f;\mathbf{O})$, as a free loop of $M_f$,
	can be homotoped to the thickened decomposition torus $M_{f,\mathfrak{e}}$.
	
	If $K_\mathbf{O}$ corresponds to an isolated interior fixed point of $S_\pA$ (Example \ref{example_fpc_classification} (2)),
	$\ell_m(f;\mathbf{O})$ is carried by an $\mathbb{H}^3$--geometric piece $M_{f,\mathfrak{v}}$.
	Then the acylindricity property of the JSJ decomposition
	implies that $\mathfrak{e}$ has to be incident to $\mathfrak{v}$, 
	and moreover, 
	there has to be some free homotopy within $M_{f,\mathfrak{v}}$ 
	that moves $\ell_m(f;\mathbf{O})$ into the boundary component adjacent to $M_{f,\mathfrak{e}}$.
	However, this case is impossible,
	because closed trajectories through interior fixed points are quasigeodesics of $M_{f,\mathfrak{v}}$,
	and they never represent parabolic elements under the holonomy representation
	of $\pi_1(M_{f,\mathfrak{v}})$ into $\mathrm{PSL}(2,\Complex)$,
	(see \cite{Hoffoss}).
	We point out that  alternatively, one may deduce the impossibility
	from \cite[Corollary 3.5]{Jiang--Guo},
	using a trick of collapsing $\partial S_{f,\mathfrak{v}}$ to points.

	If $K_\mathbf{O}$ contains a component of $S_\periodic$  (Example \ref{example_fpc_classification} (5)),
	$\ell_m(f;\mathbf{O})$ can be homotoped into an $\mathbb{H}^2\times\mathbb{E}^1$--geometric piece $M_{f,\mathfrak{v}}$.
	If $K_\mathbf{O}$ is a component of $S_\reduction$ or $\partial S_\reduction$ 
	(Example \ref{example_fpc_classification} (3) and (4)),
	$\ell_m(f;\mathbf{O})$ is carried by a thickened decomposition torus $M_{f,\mathfrak{e}'}$
	adjacent to at least one $\mathbb{H}^3$--geometric pieces.
	The acylindricity property of the JSJ decomposition again	implies 
	that $\mathfrak{e}$ is either incident to $\mathfrak{v}$,	or equal to $\mathfrak{e}'$.
	Then $\ell_m(f;\mathbf{O})$ already occurs
	as $\ell_m(f;\mathfrak{b})$, for some geometric edge end $\mathfrak{b}$ departing from $\mathfrak{e}$.
	This proves the second statement.
\end{proof}

\begin{lemma}\label{carry_trajectory_profinite}
	In a monologue setting $(S,f)$, 
	suppose that the listed conditions of Lemma \ref{simplification_lemma} are satisfied.
	Suppose that $\hat{\mathfrak{e}}$ is a profinite geometric edge center of $\widehat{\mathfrak{T}}$
	which lies over a geometric edge center $\mathfrak{e}$ of $\mathfrak{G}$.
	Then the following statements hold true for all $m\in\Natural$.
	\begin{enumerate}
	\item
	For any geometric edge end $\mathfrak{b}$ of $\mathfrak{G}$ departing from $\mathfrak{e}$,
	$\hat{\ell}_m(f;\mathfrak{b})\in\conjorb(\widehat{\Pi})$ is a profinite essential $m$--periodic trajectory class.
	\item
	For any essential $m$--periodic orbit class in $\mathbf{O}\in\POC_m(f)$,
	if $\hat{\ell}_m(f;\mathbf{O})\in\conjorb(\widehat{\Pi})$
	has nonempty intersection with $\widehat{\Pi}_{\hat{\mathfrak{e}}}$,
	then $\hat{\ell}_m(f;\mathbf{O})$ is equal to $\hat{\ell}_m(f;\mathfrak{b})$
	for some geometric edge end $\mathfrak{b}$ departing from $\mathfrak{e}$.
	\end{enumerate}
\end{lemma}

\begin{proof}
	The first statement follows directly from Lemma \ref{carry_trajectory_universal} (1),
	(see (\ref{hat_l_m_f_b}) and (\ref{profinite_essential_orbit_set})).
		
	The second statement follows from Lemma \ref{carry_trajectory_universal} (2)
	and the relation between $\widehat{\Pi}_{\hat{\mathfrak{e}}}$ and $\Pi_{\tilde{\mathfrak{e}}}$
	(Proposition \ref{graph_element_stabilizers}).
	To elaborate, 
	suppose that $\mathbf{O}\in\POC_m(f)$ is an essential periodic orbit
	such that the $\widehat{\Pi}$--conjugacy orbit $\hat{\ell}_m(f;\mathbf{O})$
	intersects $\widehat{\Pi}_{\hat{\mathfrak{e}}}$ in a nonempty set.
	By Proposition \ref{graph_element_stabilizers},
	we may replace $\hat{\mathfrak{e}}$ with some $\tilde{\mathfrak{e}}$ in $\mathfrak{T}$,
	which lies over the same geometric edge center $\mathfrak{e}$.
	Observe that $\Pi_{\tilde{\mathfrak{e}}}$ is projected isomorphically onto its image
	under the abelianization of $\Pi$,
	so $\widehat{\Pi}_{\tilde{\mathfrak{e}}}$ 
	is also projected isomorphically onto its image
	under the abelianization of $\widehat{\Pi}$,
	(Proposition \ref{graph_element_stabilizers}).
	If the image of (any element of) $\ell_m(f;\mathbf{O})$
	was not contained in the image of $\Pi_{\tilde{\mathfrak{e}}}$ 
	under the abelianization of $\Pi$,
	the $\widehat{\Gamma}$--conjugation orbit $\hat{\ell}_m(f;\mathbf{O})$ 
	would have empty intersection with $\widehat{\Pi}_{\tilde{\mathfrak{e}}}$,
	contrary to our assumption.
	It follows that some unique element $g\in\Pi_{\tilde{\mathfrak{e}}}$
	has the same image as with $\ell_m(f;\mathbf{O})$,
	under the abelianization of $\Pi$.
	Moreover,
	$g$ has to be the unique element in the intersection
	between $\hat{\ell}_m(f;\mathbf{O})$ and $\widehat{\Pi}_{\tilde{\mathfrak{e}}}$,
	(see Lemma \ref{unique_intersection}).
	If $g$ was not contained in $\hat{\ell}(f;\mathfrak{b})$
	for any geometric edge end $\mathfrak{b}$ departing from $\mathfrak{e}$,
	it would be implied by Lemma \ref{carry_trajectory_universal} (2) 
	that $g$ does not lie in the $\Pi$--conjugation orbit $\ell_m(f;\mathbf{O})$.
	Then for some finite quotient $\Pi\to\Gamma$,
	the image of $g$ does not lie in the $\Gamma$--conjugation orbit
	projected by $\ell_m(f;\mathbf{O})$,
	because of the conjugacy separability of $\Pi$
	\cite{HWZ_conjugacy_separability},
	(see also the proof of Lemma \ref{N_m_detect} for elaboration).
	Then $g$ would not have laid in the $\widehat{\Pi}$--conjugation orbit
	$\hat{\ell}_m(f;\mathbf{O})$, 
	which is a contradiction.
	Therefore,
	$g$ must be contained in $\hat{\ell}(f;\mathfrak{b})$
	for some geometric edge end $\mathfrak{b}$ departing from $\mathfrak{e}$.
	In other words, $\hat{\ell}_m(f;\mathbf{O})$
	is equal to $\hat{\ell}_m(f;\mathfrak{b})$.
	This proves the second statement.	
\end{proof}

\begin{lemma}\label{complexity_equal}
	In a monologue setting $(S,f)$, 
	suppose that the listed conditions of Lemma \ref{simplification_lemma} are satisfied.
	If $[f']\in\mcg(S)$ is any mapping class which is procongruently conjugate to $[f]$,
	then the following equalities hold:
	$$\begin{array}{cc}\dilatation([f'])=\dilatation([f]),&\deviation([f'])=\deviation([f]).\end{array}$$
\end{lemma}

\begin{proof}
	By Theorem \ref{N_m_dialogue} and (\ref{N_m_def}), we see that the $m$--periodic Nielsen number $N_m(f)$
	depends only on the procongruent conjugacy class of $[f]\in\mcg(S)$.
	It is well-known that the asymptotic growth of $N_m(f)$ is governed by the dilatation of $f$, 
	as $m$ increases to $\infty$:
	$$\limsup_{m\to\infty} \max\{1,N_m(f)^{1/m}\}=\dilatation([f]),$$
	(see \cite[Theorem 3.7]{Jiang_periodic}).	
	Therefore, $\dilatation([f])$ depends only on the procongruent conjugacy class of $f$,
	and the asserted dilatation equality follows.
	
	To prove the asserted deviation equality,	
	we may assume without loss of generality that $(S,f')$ also satisfies the listed conditions
	of Lemma \ref{simplification_lemma}.
	Otherwise, we prove the equality for some suitable power $[f]^m,[f']^m\in\mcg(S)$, 
	which are still procongruently conjugate.
	Then the original equality will follow from (\ref{dil_dev_power}).
	By raising to a suitable power, 
	we can ensure the listed conditions of Lemma \ref{simplification_lemma}:
	In fact, $[f']$ can always be represented
	by a Nielsen--Thurston normal form, with respect to some Nielsen--Thurston decomposition
	$(S'_\periodic,S'_\pA,S'_\reduction)$ of $S$;
	the non-separation condition on the reduction curves follows from 
	the correspondence between the geometric decomposition graphs of the mapping tori
	(Theorem \ref{correspondence_profinite_objects} and Proposition \ref{characterization_mapping_tori});
	the split condition can be achieved by passing to the $d'$--th power of $[f']$ and $[f]$,
	where $d'$ stands for the split order $[f']$.
	(One may actually infer $d'=1$, 
	using Theorem \ref{correspondence_profinite_objects} and \cite[Lemma 3.6]{Jiang--Guo}.)
	
	Below we treat $(S,f,f')$ as a dialogue setting $(S,f_A,f_B)$.
	Obtain an aligned isomorphism $\Psi$ of $\widehat{\Pi}_A=\widehat{\pi_1(M_A)}$
	with $\widehat{\Pi}_B=\widehat{\pi_1(M_B)}$,
	by Proposition \ref{characterization_mapping_tori}.
	For any geometric edge center $\mathfrak{e}_A$ of the geometric decomposition graph $\mathfrak{G}_A$,
	denote by $\mathfrak{e}_B$ the corresponding geometric edge center of $\mathfrak{G}_B$,
	by Theorem \ref{correspondence_profinite_objects}.
	
	We claim that the restriction of $f_A$ to the reduction annulus $S_{\mathfrak{e}_A}$
	is a nontrivial integral Dehn twist if and only if
	the same holds with $f_B$ and $S_{\mathfrak{e}_B}$,
	and that the (unsigned) shearing degrees of $(S_{\mathfrak{e}_A},f_A)$ and $(S_{\mathfrak{e}_B},f_B)$
	are equal.
	
	To prove the claim,
	observe that $(S_{\mathfrak{e}_A},f_A)$ is nontrivial
	if and only if the two geometric edge ends $\mathfrak{b}_A,\mathfrak{b}^*_A$ departing from $\mathfrak{e}_A$
	give rise to two distinct conjugation orbits
	$\ell_1(f_A;\mathfrak{b}_A),\ell_1(f_A;\mathfrak{b}^*_A)\in\conjorb(\Pi_A)$,
	(see Lemma \ref{carry_trajectory_profinite} (1)).
	In this case, 
	$\hat{\ell}_1(f_A;\mathfrak{b}_A),\hat{\ell}_1(f_A;\mathfrak{b}^*_A)\in\conjorb(\widehat{\Pi}_A)$ 
	are the only essential profinite $1$--periodic trajectory classes 
	that have nonempty intersection with the stabilizer $\widehat{\Pi}_{\hat{\mathfrak{e}}_A}$,
	where $\hat{\mathfrak{e}}_A$ 
	stands for any profinite geometric edge center of $\widehat{\mathfrak{T}}_A$ 
	that lies over $\mathfrak{e}_A$,
	(Lemma \ref{carry_trajectory_profinite}).
	Denote by $g_A,g^*_A\in\widehat{\Pi}_{\hat{\mathfrak{e}}_A}$
	the unique elements of $\widehat{\Pi}_{\hat{\mathfrak{e}}_A}$
	intersecting with $\hat{\ell}_1(f_A;\mathfrak{b}_A),\hat{\ell}_1(f_A;\mathfrak{b}^*_A)$,
	respectively.
	Observe that $\widehat{\Pi}_{\hat{\mathfrak{e}}_A}$ is a profinite free abelian group of rank $2$.
	Moreover, the quotient of $\widehat{\Pi}_{\hat{\mathfrak{e}}_A}$
	by the profinite closure of the subgroup generated by $g_A$ and $g^*_A$ 
	is a finite cyclic group $Q_{\mathfrak{e}_A}$.
	By definition,
	$Q_{\mathfrak{e}_A}$
	can be naturally identified with the quotient of $H_1(M_{f,\mathfrak{e}_A};\Integral)$	by the abelian subgroup generated by the homology classes of 
	$\hat{\ell}_1(f_A;\mathfrak{b}_A)$ and $\hat{\ell}_1(f_A;\mathfrak{b}^*_A)$.
	(For example, one may see this by taking $\hat{\mathfrak{e}}_A$ as a universal geometric center $\tilde{\mathfrak{e}}_A$ of $\mathfrak{T}_A$
	and applying Proposition \ref{graph_element_stabilizers}.)
	Therefore,
	the shearing degree of the integral Dehn twist $(S_{\mathfrak{e}_A},f_A)$
	can be characterized as the order of $Q_{\mathfrak{e}_A}$.
	Denote by $\hat{\mathfrak{e}}_B$ the profinite geometric edge center of $\widehat{\mathfrak{T}}_B$
	that corresponds to $\hat{\mathfrak{e}}_A$, (Theorem \ref{correspondence_profinite_objects}).
	We construct similarly a finite quotient $Q_{\mathfrak{e}_B}$ of $\widehat{\Pi}_{\hat{\mathfrak{e}}_B}$,
	whose order equals the shearing degree of $(S_{\mathfrak{e}_B},f_B)$.
	Note that the finite cyclic groups $Q_{\mathfrak{e}_A}$ and $Q_{\mathfrak{e}_B}$ 
	are constructed only with objects on the profinite level, 
	and that the objects are in natural corespondence as induced by $\Psi$.
	Then there is also an induced isomorphism $Q_{\mathfrak{e}_A}\cong Q_{\mathfrak{e}_B}$, 
	(Theorem \ref{correspondence_profinite_objects}).
	In particular, $Q_{\mathfrak{e}_A}$ and $Q_{\mathfrak{e}_B}$ have the same order,
	so the claim is implied.
	
	By the above claim and the definition of deviation (see (\ref{dil_dev_def})),
	we obtain the asserted dilatation equality.
\end{proof}

Combining Lemmas \ref{simplification_lemma}, \ref{complexity_equal}, and Proposition \ref{bound_finiteness},
we complete the proof of Theorem \ref{procongruent_conjugacy_finiteness}.

\appendix

\section{Characteristic versus profinitely characteristic}\label{Sec-profinitely_characteristic}
In this appendix, we supply an example as mentioned in Remark \ref{profinitely_characteristic}.
We exhibit a finitely generated residually finite group $G$, 
such that some pair of outer automorphisms $[\varphi_A],[\varphi_B]\in \outg(G)$
are procongruently conjugate, 
but not congruently conjugate modulo some characteristic finite-index subgroup $K$.
In fact, our example group $G$ is the fundamental group
of the product of a pair of $3$-manifolds,
which are the torus bundles over circles 
with Anosov monodromies as given by Stebe's pair (Example \ref{Stebe_example}).

For simplicity, we rewrite the matrices $[f_A]$ and $[f_B]$ 
in Example \ref{Stebe_example} as $A$ and $B$ respectively.
Denote by $\pi_A$ the semidirect product group $\Integral\ltimes\Integral^2$
whose multiplication is explicitly $(m,x)(n,y)=(m+n, A^{-n}x+y)$.
Denote by $\nu_A\in\autg(\pi_A)$ the involution $\nu_A(m,x)=(m,-x)$.
Similarly, we denote by $\pi_B$ 
the group $\Integral\ltimes\Integral^2\colon (m,x)(n,y)=(m+n, B^{-n}x+y)$,
and by $\nu_B\in\autg(\pi_B)$ the involution $\nu_B(m,x)=(m,-x)$.
Take the direct product group
$$G=\pi_A\times \pi_B.$$
We obtain a pair of involutions 
\begin{equation*}
	\begin{array}{cc}
	{\varphi_A=\nu_A\times\mathrm{id}},&
	{\varphi_B=\mathrm{id}\times\nu_B}
	\end{array}
\end{equation*}
in $\autg(G)$.

Since the matrices $A,B\in\mathrm{GL}(2,\Integral)$ are conjugate in $\mathrm{GL}(2,\widehat{\Integral})$,
there is an aligned isomorphism $\Psi\colon \widehat{\pi}_A\to\widehat{\pi}_B$,
which takes the form $\Psi(m,x)=(m,Sx)$, 
such that $S\in \mathrm{GL}(2,\widehat{\Integral})$ satisfies $SA=BS$.
Observe $\Psi\circ\nu_A=\nu_B\circ\Psi$.
It follows that $\varphi_A,\varphi_B\in\autg(G)$ 
are conjugate in $\autg(\widehat{G})$,
(under the automorphism 
$(\alpha,\beta)\mapsto (\Psi^{-1}\beta,\Psi\alpha)$
of $\widehat{G}=\widehat{\pi}_A\times\widehat{\pi}_B$).
In particular, 
we see that $[\varphi_A],[\varphi_B]\in\outg(G)$ 
become conjugate in $\outg(\widehat{G})$.
In other words, they are procongruently conjugate.

\begin{lemma}\label{aut_split}
	With the above notations, $\autg(G)=\autg(\pi_A)\times\autg(\pi_B)$.
\end{lemma}

\begin{proof}
	We observe that the centralizer in $\pi_A$ of any nontrivial element $\alpha\in\pi_A$ is abelian.
	It follows that a nontrivial element $g=(\alpha,\beta)$ in $G=\pi_A\times\pi_B$
	has nonabelian centralizer if and only if $\alpha$ or $\beta$ is trivial.
	Therefore, any automorphism $\sigma$ of $G$ either preserves  
	each of the factor subgroups $\pi_A$ and $\pi_B$, or switches them.
	However, $\sigma$ cannot switch $\pi_A$ and $\pi_B$,
	since it is Stebe's discovery that these groups are not isomorphic.
	This forces the splitting $\autg(G)=\autg(\pi_A)\times\autg(\pi_B)$, as asserted.
\end{proof}

By Lemma \ref{aut_split}, 
the direct product of any characteristic subgroups of the factors
is a characteristic subgroup of $G$.
To obtain a useful subgroup of such form,
we observe that $[\nu_A]\in\outg(\pi_A)$ is nontrivial.
This is because $\pi_A$ is the fundamental group of a torus bundle over a circle
with Anosov monodromy, so any homotopically nontrivial loop in a fiber torus
is not freely homotopic in the $3$--manifold to its orientation reversal.
Moreover,
by the conjugacy separability of $\pi_A$ \cite{HWZ_conjugacy_separability},
there exists a characteristic finite-index subgroup $K_A$ of $\pi_A$,
such that $[\nu_A]$ is congruently nontrivial modulo $K_A$.

Take the characteristic finite-index subgroup $K=K_A\times\pi_B$ of $G$,
whose quotient group $G/K$ is isomorphic to $\pi_A/K_A$.
We see that $[\varphi_A]\in\outg(G)$ is congruently nontrivial modulo $K$,
while $[\varphi_B]\in\outg(G)$ is obviously congruently trivial modulo $K$.
Therefore, $[\varphi_A],[\varphi_B]\in\outg(G)$
are not congruently conjugate modulo $K$.

We also infer 
from Proposition \ref{characterization_congruent_conjugacy} and Remark \ref{profinitely_characteristic}
that $K$ is not profinitely characteristic in $G$.

\bibliographystyle{amsalpha}


\begin{thebibliography}{}

\bibitem[Ago13]{Agol_VHC} 
{I.~Agol}, {\it The virtual Haken conjecture}, 
with an appendix by I.~Agol, D.~Groves, J.~Manning, Documenta Math.~\textbf{18} (2013), 1045--1087.

\bibitem[ArnY81]{Arnoux--Yoccoz_pA}
P.~Arnoux and J.-C. Yoccoz. 
\textit{Construction de diff\'eomorphismes pseudo-Anosov}. 
C.~R.~Acad.~Sci.~Paris S\'er. I Math., \textbf{292} (1981), 75--78.

\bibitem[AscFW15]{AFW_book_group}  
M.~Aschenbrenner, S.~Friedl, and H.~Wilton, \textit{3-Manifold Groups}, EMS Series of Lectures in Mathematics, 2015.

\bibitem[BerV13]{Bergeron--Venkatesh}
{N.~Bergeron and A.~Venkatesh},
{\it The asymptotic growth of torsion homology for arithmetic groups},
J.~Inst.~Math.~Jussieu \textbf{12} (2013), 391--447.

\bibitem[BoiF20]{Boileau--Friedl_fiberedness}
M.~Boileau and S.~Friedl, ``The profinite completion of 3-manifold
groups, fiberedness and the Thurston norm''. 
In: {\it What's Next?--- the mathematical legacy of William P.~Thurston}, pp.~21--44,
Ann.~Math.~Stud., 205, Princeton Univ.~Press, Princeton, NJ, 2020.


\bibitem[BriMRS20]{BMRS}
M.~Bridson, D.~McReynolds, A.~Reid, and R.~Spitler,
\textit{Absolute profinite rigidity and hyperbolic geometry},
Ann.~Math.~(2) \textbf{192} (2020), no.~3, 679--719.

\bibitem[BriRW17]{BRW} 
{M.~Bridson, A.~Reid, and H.~Wilton}, {\it Profinite rigidity and surface bundles over the circle}, 
Bull.~Lond.~Math.~Soc.~\textbf{49} (2017), 831--841.


\bibitem[Bro82]{Brown_book}
K.~S.~Brown, \textit{Cohomology of Groups},
Springer--Verlag, New York, 1982.

\bibitem[Cob13]{Cotton-Barratt}
O.~Cotton-Barratt, \textit{Detecting ends of residually finite groups in profinite completions},
Math.~Proc.~Cambridge Philos.~Soc. \textbf{155} (2013), 379--389.

\bibitem[Coc09]{Cotton-Clay}
A.~Cotton-Clay, \textit{Symplectic Floer homology of area-preserving surface diffeomorphisms},
Geom.~Topol.~\textbf{13} (2009), 2619--2674.

\bibitem[FatLP12]{FLP_book} 
A.~Fathi, F.~Laudenbach, V.~Po\'enaru, \textit{Thurston's Work on Surfaces}.
Translated from the 1979 French original by D.~M.~Kim and D.~Margalit. 
Princeton University Press, Princeton, NJ, 2012. 

\bibitem[FarM12]{Farb--Margalit_book}
B.~Farb and D.~Margalit,
\textit{A Primer on Mapping Class Groups}, 
Princeton Mathematical Series, 49. 
Princeton University Press, Princeton, NJ, 2012.

\bibitem[Fel08]{Felshtyn_HF}
A.~Fel'shtyn, \textit{Nielsen theory, Floer homology and a generalisation of the Poincar\'e--Birkhoff theorem}
J.~Fixed Point Theory Appl.~\textbf{3} (2008), 191--214.

\bibitem[FrlV11]{Friedl--Vidussi_survey}
S.~Friedl and S.~Vidussi,
``A survey of twisted Alexander polynomials'' in: The Mathematics of Knots, pp.~45--94, 
Contrib.~Math.~Comput.~Sci., 1, Springer, Heidelberg, 2011. 

\bibitem[Fun13]{Funar_torus_bundles}
L.~Funar, \textit{Torus bundles not distinguished by TQFT invariants},
Geom.~Topol.~\textbf{17} (2013), 2289--2344.

\bibitem[GruJZ08]{GJZ_goodness}
F.~J.~Grunewald, A.~Jaikin-Zapirain, and P.~A.~Zalesskii (2008). 
\textit{Cohomological goodness and the profinite completion of Bianchi groups},
Duke Math.~J.~\textbf{144} (2008), 53--72.

\bibitem[Ham01]{Hamilton_abelian_separability}
E.~Hamilton, 
\textit{Abelian subgroup separability of Haken 3-manifolds and closed hyperbolic $n$-orbifolds}, 
Proc.~London Math.~Soc.~\textbf{83} (2001), 626--646.

\bibitem[HamWZ13]{HWZ_conjugacy_separability}
E.~Hamilton, H.~Wilton, and P.~Zalesskii,
\textit{Separability of double cosets and conjugacy classes in 3-manifold groups},
J.~Lond.~Math.~Soc.~(2) \textbf{87} (2013), 269--288. 

\bibitem[Hat02]{Hatcher_book}
A.~Hatcher, \textit{Algebraic Topology},
Cambridge University Press, Cambridge, 2002.


\bibitem[Hem14]{Hempel_quotient} 
{J.~Hempel}, {\it Some 3-manifold groups with the same finite quotients}, preprint 2014, 10 pages: 
\url{https://arxiv.org/abs/1409.3509v2}


\bibitem[Hof07]{Hoffoss}
D.~Hoffoss, \textit{Suspension flows are quasigeodesic}, J.~Diff.~Geom.~\textbf{76} (2007), 215--248.

\bibitem[Iva88]{Ivanov_pA}
N.~V.~Ivanov,
\textit{Coefficients of expansion of pseudo-Anosov homeomorphisms}, (Russian, English summary),
Zap.~Nauchn.~Sem.~Leningrad.~Otdel.~Mat.~Inst.~Steklov.~(LOMI) \textbf{167} (1988), 
Issled.~Topol.~6, 111--116, 191;
English translation in J. Soviet Math. \textbf{52} (1990),
2819--2822.


\bibitem[Jai20]{JZ} 
{A.~Jaikin-Zapirain}, {\it Recogniton of being fibered for compact 3-manifolds}, 
Geom.~Topol.~\textbf{24} (2020), no.~1, 409--420.


\bibitem[Jia83]{Jiang_book} 
B.~Jiang, \textit{Lectures on Nielsen Fixed Point Theory},
Contemp.~Math.~14, Amer.~Math.~Soc., Province, RI, 1983.

\bibitem[Jia96]{Jiang_periodic}
B.~Jiang, \text{Estimation of the number of periodic orbits},
Pacific J.~Math.~\textbf{172} (1996), 151--185.

\bibitem[JiaG93]{Jiang--Guo} B.~Jiang and J.~Guo, \textit{Fixed points of surface diffeomorphisms}, Pacific J.~Math.~\textbf{160} (1993), 67--89.

\bibitem[McM]{McMullen_entropy} C.~T.~McMullen, \textit{Entropy on Riemann surfaces and the Jacobians of finite covers},
	Comm.~Math.~Helv.~\textbf{88} (2013), 953--964.

\bibitem[Min12]{Minasyan_RAAG} 
A.~Minasyan, \textit{Hereditary conjugacy separability of right angled Artin groups and its applications},
Groups Geom.~Dyn.~\textbf{6} (2012) 335--388.

\bibitem[NikS07]{Nikolov--Segal}
N.~Nikolov and D.~Segal, \textit{On finitely generated profinite groups. I. Strong
completeness and uniform bounds}. Ann.~Math.~(2) \textbf{165} (2007), 171--238.

\bibitem[PlaR94]{Platonov--Rapinchuk_book}
V.~Platonov and A.~Rapinchuk, \textit{Algebraic Groups and Number Theory}, 
Pure and Applied Mathematics 139, Academic Press, Boston, 1994.

\bibitem[PrzW18]{Przytycki--Wise_mixed} {P.~Przytycki and D.~Wise}, {\it Mixed $3$-manifolds are virtually special}, J.~Amer.~Math.~Soc. \textbf{31} (2018), 319--347.

\bibitem[Rei18]{Reid_survey} 
{A.~Reid}, ``Profinite rigidity''.
In: {\it Proceedings of the International Congress of Mathematicians--- Rio de Janeiro 2018. Vol. II. Invited lectures},
pp.~1193--1216, World Sci.~Publ., Hackensack, NJ, 2018.

\bibitem[Rib17]{Ribes_book}
L.~Ribes, \textit{Profinite Graphs and Groups}, A Series of Modern Surveys in Mathematics, vol. 66, Springer, 2017.

\bibitem[RibZ10]{Ribes--Zalesskii_book}
L.~Ribes and P.~Zalesskii, \textit{Profinite Groups}, 
2nd ed., Springer--Verlag, Berlin, 2010.

\bibitem[Ser77]{Serre_book_rep}
J.-P.~Serre, 
\textit{Linear Representations of Finite Groups}. 
Translated from the second French edition by Leonard L. Scott. 
Springer-Verlag, New York-Heidelberg, 1977.

\bibitem[Ser97]{Serre_book_Galois}
J.-P.~Serre,
\textit{Galois Cohomology},
Springer-Verlag, Berlin, 1997.

\bibitem[Ste72]{Stebe_integer_matrices}
P.~F.~Stebe, \textit{Conjugacy separability of groups of integer matrices},
Proc.~Amer.~Math.~Soc.~\textbf{32} (1972), 1--7.

\bibitem[Tur01]{Turaev_book_torsion} V. Turaev, \textit{Introduction to Combinatorial Torsions}, Birkh\"auser, Basel, 2001.


\bibitem[Uek18]{Ueki} 
{J.~Ueki}, 
{\it The profinite completions of knot groups determine the Alexander polynomials}, 
Algebr.~Geom.~Topol.~\textbf{18} (2018), no.~5, 3013--3030.


\bibitem[Wlk17a]{Wilkes_thesis} 
G.~Wilkes, \textit{Profinite Properties of 3-Manifold Groups}. PhD thesis, University of Oxford, 2017.

\bibitem[Wlk17b]{Wilkes_sf}
G.~Wilkes, \textit{Profinite rigidity for Seifert fibre spaces},
Geom.~Dedicata \textbf{188} (2017), 141--163.

\bibitem[Wlk18]{Wilkes_graph}
{G.~Wilkes}, {\it Profinite rigidity of graph manifolds and JSJ decompositions of 3-manifolds}, 
J.~Algebra \textbf{502} (2018), 538--587.

\bibitem[WltZ10]{WZ_graph_manifolds} 
H.~Wilton and P.~Zalesskii, 
\textit{Profinite properties of graph manifolds}, Geom.~Dedicata \textbf{147} (2010) 29--45.

\bibitem[WltZ17]{WZ_geometry} 
{H.~Wilton and P.~Zalesskii}, {\it Distinguishing geometries using finite quotients}, 
Geom.~Topol.~\textbf{21} (2017), 345--384.


\bibitem[WltZ19]{WZ_decomposition}
H.~Wilton and P.~Zalesskii,
\textit{Profinite detection of 3-manifold decompositions}, 
Compos.~Math.~\textbf{155} (2019), no.~2, 246--259.

\bibitem[Wis12]{Wise_book}
{D.~Wise},
{\it From Riches to Raags: $3$-Manifolds, Right-Angled Artin Groups, and Cubical Geometry},
CBMS Regional Conference Series in Mathematics, 117, American Mathematical Society, Providence, RI, 2012.


\bibitem[Wis21]{Wise_notes}
D.~Wise, Daniel T.
{\it The structure of groups with a quasiconvex hierarchy}.
Annals of Mathematics Studies, 209. Princeton University Press, Princeton, NJ, [2021].

\end{thebibliography}

\end{document}